\documentclass[11pt, letterpaper, hidelinks]{article}
\usepackage[margin=1in]{geometry}
\usepackage{authblk}

\usepackage{url}
\usepackage{mathtools}
\usepackage[dvipsnames]{xcolor}
\usepackage{amsmath,amsthm, amsfonts, amssymb}
\usepackage{framed}
\renewcommand{\proofname}{\textit{\textbf{Proof}}}

\makeatletter
\renewenvironment{proof}[1][\proofname]{\par\pushQED{\qed}\normalfont \topsep6\p@\@plus6\p@\relax\trivlist\item[\hskip\labelsep#1\@addpunct{}]\ignorespaces}{\popQED\endtrivlist\@endpefalse}
\makeatother
\newcommand{\proofsketchname}{\textit{\textbf{Proof sketch}}}
\newenvironment{proofsketch}
  {\begin{proof}[\proofsketchname]}
  {\end{proof}}
\usepackage[hypertexnames=false]{hyperref}
\newcommand\myshade{85}
\colorlet{mylinkcolor}{violet}
\colorlet{mycitecolor}{YellowOrange}
\colorlet{myurlcolor}{Aquamarine}
\hypersetup{
  linkcolor  = mylinkcolor!\myshade!black,
  citecolor  = mycitecolor!\myshade!black,
  urlcolor   = myurlcolor!\myshade!black,
  colorlinks = true,
}
\usepackage{enumitem, comment, xifthen}
\usepackage{graphicx}
\usepackage{etoolbox}
\usepackage{tikz}
\usetikzlibrary{math}
\usepackage{mathabx}
\usepackage{subfig}
\graphicspath{{figs/}}
\DeclareMathAlphabet{\mathsf}{OT1}{qhv}{m}{n}
\usepackage{cleveref}
\usepackage{autonum}
\usepackage{xspace}

\DeclarePairedDelimiterX{\klx}[2]{(}{)}{#1\;\delimsize\|\;#2}
\newcommand{\R}{\mathbf R}
\newcommand{\Q}{\mathbf Q}
\newcommand{\cA}{\mathcal{A}}
\newcommand{\cB}{\mathcal{B}}
\newcommand{\cE}{\mathcal{E}}
\newcommand{\cL}{\mathcal{L}}
\newcommand{\cP}{\mathcal{P}}
\newcommand{\cQ}{\mathcal{Q}}
\newcommand{\cV}{\mathcal{V}}
\newcommand{\twomax}[2]{\ensuremath{#1 \lor #2}}
\newcommand{\twomin}[2]{\ensuremath{#1 \land #2}}
\newcommand{\ceil}[1]{\left\lceil #1 \right\rceil}
\newcommand{\floor}[1]{\left\lfloor #1 \right\rfloor}
\newcommand{\e}{\mathrm{e}}
\DeclareMathOperator{\sign}{\bf sign}
\newcommand{\ud}[0]{\mathrm{d}}
\newcommand{\1}{\mathbf 1}
\newcommand{\half}{\frac12}
\newcommand{\eps}{\varepsilon}
\let\tilde\widetilde
\let\subseteq\subset
\let\preceq\preccurlyeq
\renewcommand{\le}{\leqslant}
\renewcommand{\ge}{\geqslant}
\renewcommand{\leq}{\leqslant}
\renewcommand{\geq}{\geqslant}
\DeclareMathOperator{\trace}{\bf Tr}
\DeclareMathOperator{\dom}{\bf dom}
\DeclareMathOperator{\rank}{\bf rank}

\let\Span\spn
\newcommand{\T}{\mathsf{T}}
\newcommand{\kl}{D_{\rm kl}\klx}
\newcommand{\iid}{\textnormal{i.i.d.}}
\newcommand{\simiid}{\stackrel{\iid}{\sim}}
\newcommand{\distto}{\stackrel{\rm d}{\longrightarrow}}
\newcommand{\Var}{\mathrm{Var}}
\makeatletter
\newcommand{\E}{\operatorname*{\mathbf{E}}\ilimits@}
\renewcommand{\P}{\operatorname*{\mathbf{P}}\ilimits@}
\makeatother
\newcommand{\ie}{\textit{i}.\textit{e}., }
\newcommand{\eg}{\textit{e}.\textit{g}., }

\theoremstyle{plain}
\newtheorem{theorem}{Theorem}[section]
\newtheorem{lemma}{Lemma}[section]
\newtheorem{proposition}{Proposition}[section]
\newtheorem{corollary}{Corollary}[section]
\theoremstyle{definition}
\newtheorem{definition}{Definition}[section]

\newtheorem{remark}{Remark}[section]
\newtheorem{fact}{Fact}[section]

\newcommand{\thetastar}{\theta^\star}
\newcommand{\wh}{\widehat}
\let\top\T
\newcommand{\Normal}[2]{\mathsf{N}\left(#1, #2\right)}
\newcommand{\defn}{=}
\newcommand{\Quant}[1]{\mathsf{Q}_{#1}}
\let\hat\widehat
\newcommand{\betastar}{\beta^\star}

\newcommand{\Ber}[1]{\mathsf{Ber}\left(#1\right)}
\newcommand{\KSDist}[2]{d_{\mathsf{KS}}\left(#1, #2\right)}
\newcommand{\FD}{\mathsf{FD}}
\newcommand{\RD}{\mathsf{RD}}
\newcommand{\WorstCaseFixedDesign}[2]{\cE^\FD(#1,#2)}
\newcommand{\WorstCaseRandomDesign}[2]{\cE^\RD(#1,#2)}
\newcommand{\WorstCaseFixedDesignQuant}[3]{\cQ_{#1}^\FD(#2,#3)}
\newcommand{\WorstCaseRandomDesignQuant}[3]{\cQ_{#1}^\RD(#2,#3)}
\newcommand{\ran}{\mathrm{ran}}
\usepackage{bm}

\usepackage[most]{tcolorbox}
\tcbuselibrary{listings,breakable}

\newtcblisting{promptbox}{
  breakable,
  listing only,
  colback=gray!3,
  colframe=black!60,
  boxrule=0.5pt,
  arc=1mm,
  left=6pt,
  right=6pt,
  top=6pt,
  bottom=6pt,
  title={Full prompt for \Cref{thm:d2-logloglog-upper}, GPT-5.6-Sol Pro, July 2026},
  fonttitle=\bfseries\small,
  listing options={
    basicstyle=\ttfamily\tiny,
    breaklines=true,
    columns=fullflexible,
    keepspaces=true,
    showstringspaces=false,
    upquote=true,
    literate={“}{{\textquotedblleft}}1 {”}{{\textquotedblright}}1
  }
}

\title{\Large \bfseries Beyond Modern Asymptotics for Log-Likelihood Ratios in Logistic Regression}

\author[1]{Hugo Chardon}
\author[1,2]{Reese Pathak}
\author[1]{Nikita Zhivotovskiy}
\affil[1]{Department of Statistics, University of California, Berkeley}
\affil[2]{School of Operations Research and Information Engineering (ORIE), Cornell University}

\date{July 31, 2026}
\begin{document}
\maketitle
\begin{abstract}
We characterize the finite sample behavior of the log-likelihood ratio statistic in binary logistic regression, uniformly over both the design and the target parameter. For $n\geq d\geq 3$, we determine, up to universal constants, its worst case $(1-\delta)$ quantile over all fixed collections of design vectors and all target parameters:
\[
d\log\left(\frac{\e n}{d}\right)+\log\left(\frac{1}{\delta}\right).
\]
This is a nonasymptotic analogue of the Wilks $\chi^2_d$ phenomenon and requires no regularity assumptions on the design.  The low dimensional cases exhibit unusual behavior. The worst case quantile in dimension $d=2$ is sharply of order
\[
\log\log\log n+\log\left(\frac{1}{\delta}\right).
\]
The worst case quantile in dimension $d=1$ is of order $\log(1/\delta)$, with no dependence on $n$. Finally, i.i.d.\ Gaussian design vectors recover the classical Wilks scale. In the regime $n\gtrsim d+\log(1/\delta)$, we prove the sharp bound
\[
d+\log\left(\frac{1}{\delta}\right).
\]
Unlike existing asymptotic results, our bounds are uniform over the target parameter, which may depend on $n$, $d$, and $\delta$. 
\end{abstract}

\tableofcontents

\section{Introduction}

Understanding the distribution of log-likelihood ratios is one of the
fundamental problems in statistics, with applications to goodness-of-fit
testing, variable selection, and the construction of confidence regions for the target
parameter. In this paper, we focus on the most natural generalized linear model
for classification, namely logistic regression, where
\(Y_1,\dots,Y_n\in\{-1,1\}\) are independent and satisfy
\[
\P_{\thetastar}(Y_i=1)
=
\sigma(\langle x_i,\thetastar\rangle),
\qquad
\sigma(t)=\frac{1}{1+e^{-t}} \, ,
\]
for deterministic covariates \(x_1,\dots,x_n\in\R^d\) and some unknown target
vector \(\thetastar\in\R^d\). In this case, the log-likelihood ratio statistic is
defined as
\begin{equation}
\label{eq:def-logistic-llr}
\Lambda_n^{\log}(\thetastar)
	\equiv \Lambda_n^{\log}(\thetastar;x_1,\ldots,x_n; Y_1, \ldots, Y_n)
	= \log \frac{\sup_{\theta\in\R^d}\prod_{i=1}^n\sigma\bigl(Y_i\langle x_i,\theta\rangle\bigr)}
		{\prod_{i=1}^n \sigma\bigl(Y_i\langle x_i,\thetastar\rangle\bigr)} \, .
\end{equation}
The first question is to understand the distribution of this
statistic. In the case of independent and identically distributed random
covariates \(X_1,\ldots,X_n\), the classical asymptotic behavior as
\(n\to\infty\) is well understood through Wilks' phenomenon; see, for instance,
the classical papers of Wilks~\cite[pp.~62]{wilks1938large} and
Chernoff~\cite[Theorem~1]{chernoff1954likelihood}. When \(d\) and
\(\thetastar\) are fixed and the usual regularity assumptions hold, this
suggests (in \Cref{sec:wilksregime} we discuss a version of this phenomenon in the broader regime $d^{3/2}/n \to 0$)
\[
\Lambda_n^{\operatorname{log}}(\thetastar;X_1,\ldots,X_n; Y_1, \ldots, Y_n)
\distto
\frac12\chi_d^2.
\]

However, logistic regression is well known to exhibit phenomena that are
invisible from this classical asymptotic viewpoint. In particular, even the
existence of the maximum likelihood estimator is a delicate question, closely
related to separation of the data; see the classical work of Albert and
Anderson~\cite{AlbertAnderson1984LogisticMLE} and, for Gaussian random design,
Candès and Sur~\cite{CandesSur2020PhaseTransition}. These difficulties are not only
technical. They reflect the fact that the likelihood geometry of logistic
regression can be far from locally quadratic, as used in the classical Wilks
theorem.

From the modern point of view, the limitation of the classical asymptotic
result is not only that it is asymptotic. More importantly, it describes only
one particular regime, where \(d\) and \(\thetastar\) are fixed and
\(n\to\infty\). High-dimensional asymptotic results of Sur, Chen and
Candès~\cite{sur2019likelihood}, and further developments in this direction
\cite{sur2019modern,ZhaoSurCandes2022LogisticMLE}, show that even under Gaussian
design and proportional
asymptotics \(d/n\to\kappa\), the usual \(\chi^2\) approximation can fail.
These results work in structured random design asymptotic regimes. In
particular, \cite{sur2019likelihood} studies likelihood ratio tests involving a
fixed number of coordinates, while \cite{sur2019modern} develops a broader
high-dimensional MLE theory for Gaussian logistic regression. Our question is
different: we ask what can be guaranteed in finite samples, for the full
log-likelihood ratio, uniformly over the design and over the target parameter.

Given \(\delta\in(0,1)\), write
$
\Quant{1-\delta}(Z)
\defn
\inf\bigl\{t\in\R:\P(Z\le t)\ge 1-\delta\bigr\}
$
for the \((1-\delta)\)-quantile of a scalar random variable \(Z\). We are
interested in the quantile of the log-likelihood ratio for logistic regression:
\[
\Quant{1-\delta}
\Bigl(
\Lambda_n^{\operatorname{log}}(\thetastar;x_{1:n}; Y_{1:n})
\Bigr).
\]
Here and below, $x_{1:n}$ denotes $x_1, \ldots, x_n$, $Y_{1:n}$ denotes $Y_1, \ldots, Y_n$,
and \(a\asymp b\) means that \(c b\le a\le Cb\) for universal
positive constants \(c,C\). The main result of the paper can be summed up as follows.

\begin{center}
\fbox{
\begin{minipage}{0.92\textwidth}
Simplified worst case fixed design bounds.
Assume that \(n\ge d\ge 1\) and \(\delta\in(0,1/\e]\), and, when \(d=2\),
that \(n \ge 16\) and \(\delta\leq1/(16\e^2)\). Then
\[
\sup_{\substack{x_1,\ldots,x_n\in\R^d\\ \thetastar\in\R^d}}
\Quant{1-\delta}\Bigl(
\Lambda_n^{\operatorname{log}}(\thetastar;x_{1:n};Y_{1:n})
\Bigr)
\asymp
\begin{cases}
\log\left(\dfrac{1}{\delta}\right),
& d=1,\\[2mm]
\log\log\log n+\log\left(\dfrac{1}{\delta}\right),
& d=2,\\[2mm]
d\log\left(\dfrac{\e n}{d}\right)+\log\left(\dfrac{1}{\delta}\right),
& 3\le d\le n.
\end{cases}
\]
See \Cref{thm:main-result} for
\(d\geq3\), and \Cref{cor:dimone} (upper bound) and
\Cref{prop:coordinate-subspace-lower} (lower bound) for \(d=1\); for \(d=2\), see
\Cref{cor:d2-logloglog-quantile-lower} and \Cref{thm:d2-logloglog-upper}.
\end{minipage}
}
\end{center}

This should be compared with the nonasymptotic quantiles suggested by the Wilks
limit. Indeed (see e.g., \Cref{lem:chisquarequantiles} below), for
\(\delta\in(0,1/2]\), it holds that
\[
\Quant{1-\delta}(\chi_d^2) \asymp d+\log\Big(\frac{1}{\delta}\Big).
\]
Thus, in particular, for $d\geq3$, in the worst case over deterministic designs, the dimension term is
inflated by the factor \(\log(en/d)\), while the tail term
\(\log(1/\delta)\) remains of the same order as in the \(\chi_d^2\) benchmark.
In this regime, our result says that this logarithmic factor is the only price for a completely
uniform finite sample guarantee. As an upper bound on the quantile, it does not
assume anything about the existence of the MLE and holds for every
deterministic design and every \(\thetastar\). Moreover, we construct explicit
designs and target parameters showing that the logarithmic factor is
unavoidable.

At the same time, the logarithmic price is tightly connected with the geometry
of the design. Our second message is that it disappears for Gaussian random
design.

\begin{center}
\fbox{
\begin{minipage}{0.92\textwidth}
Simplified consequence of \Cref{thm:gaussian-design-logfree} and \Cref{cor:gaussian-design-origin}: Gaussian design.
Let \(X_1,\ldots,X_n\) be i.i.d.\ \(\Normal{0}{I_d}\). Then, for every
\(\delta\in(0,\tfrac{1}{2}]\) such that \(n \ge d+\log(1/\delta)\),
\[
\sup_{\thetastar\in\R^d}
\Quant{1-\delta}
\Bigl(
\Lambda_n^{\operatorname{log}}(\thetastar;X_{1:n}; Y_{1:n})
\Bigr)
\asymp
d+\log\left(\frac{1}{\delta}\right).
\]
\end{minipage}
}
\end{center}

Therefore, arbitrary deterministic designs and Gaussian random designs lead to
different finite sample pictures. For Gaussian random design one always has the
nonasymptotic \(\chi_d^2\)-type upper bound $d+\log(1/\delta)$, uniformly over
the target parameter. At the origin, the exact scale is
$\min\{n,d+\log(1/\delta)\}$, including its natural saturation at $n$.
This is the sense in which the present results go beyond both the classical
Wilks asymptotics and the modern proportional asymptotic theory: the upper bounds are
finite sample, apply to the full likelihood ratio, allow arbitrary target
parameters, do not require the MLE to exist, and in the main worst case theorem
are uniform over all deterministic designs.

\paragraph{Interpretation: the geometry of likelihood-ratio confidence sets.}

One consequence of our upper bounds is that they give genuinely
nonasymptotic confidence sets for fixed design logistic regression. Namely, for
\(n\geq d\), define

\begin{equation}
\label{eq:logistic-nonasymptotic-confidence-set}
\widehat \Theta_{n,\delta}^{\log}
= \left\{ \theta\in\R^d: \Lambda_n^{\log}(\theta) \leq d\log(\e n/d)+\log(1/\delta) \right\}.
\end{equation}
Then, for every deterministic design \(x_1,\ldots,x_n\in\R^d\), every
\(\thetastar\in\R^d\), and every \(\delta\in(0,1)\),
\begin{equation}
\P_{\thetastar}\left(
\thetastar\in
\widehat \Theta_{n,\delta}^{\operatorname{log}}
\right)
\geq 1-\delta .
\end{equation}
This statement does not assume anything about the conditioning of the design,
does not impose any restriction on \(\|\thetastar\|_2\), and does not require
the maximum likelihood estimator to exist.

The price for this level of generality is that the geometry of
\(\widehat \Theta_{n,\delta}^{\operatorname{log}}\) should not be expected to be
ellipsoidal. In the Wilks regime, this geometry is usually read from the local
quadratic approximation of the likelihood. In logistic regression, the
quadratic expansion is driven by the Hessian:
\[
\nabla_\theta^2 \Lambda_n^{\operatorname{log}}(\theta)
\big|_{\theta=\thetastar}
=
\sum_{i=1}^n
\sigma(\langle x_i,\thetastar\rangle)
\sigma(-\langle x_i,\thetastar\rangle)
x_i x_i^\T .
\]
The weights $\sigma(t)\sigma(-t)$ decay like $e^{-|t|}$ for large
$|t|$. Specifically, for all real $t$, $\sigma(t)\sigma(-t)\leq \e^{-|t|}$, hence observations with (even moderately) large margin
$|\langle x_i,\thetastar\rangle|$ contribute very little to the curvature.
Even for a fixed well-specified deterministic design, the local information
matrix may therefore be nearly singular, and the usual ellipsoidal Wilks
geometry may no longer describe the likelihood-ratio confidence set.

This is not a pathological corner case. Large margins correspond to low-noise
logistic models, where the labels are close to deterministic. In the extreme
realizable limit, if the observed labels are separable, the infimum of the
empirical logistic loss is approached along rays going to infinity, and the
natural geometry is closer to a cone of separating directions than to an
ellipsoid; see, for instance,
\cite{AlbertAnderson1984LogisticMLE,Cover1965}. Our results are aimed at the
intermediate regime between these two pictures. The likelihood need not be
locally quadratic in any useful uniform sense, but the model also need not be
fully realizable. The confidence set
\eqref{eq:logistic-nonasymptotic-confidence-set} remains valid throughout this
regime.

\subsection{Our results in the context of the related literature}

Our analysis draws on several related lines of work. We outline them briefly and
explain how they enter our results.

\paragraph{Classical and modern asymptotics.}

The classical route to likelihood ratio inference is local and quadratic.
Wilks' theorem~\cite{wilks1938large} is one manifestation of this principle.
More generally, the local asymptotic normality theory of Le Cam
\cite{LeCam1960LAN,LeCamYang2000}, developed further for instance by
Ibragimov and Khas'minskii~\cite{IbragimovKhasminskii1981} and presented in
modern form in van der Vaart~\cite{vdV1998asymptotic}, is based on a quadratic
approximation of the log-likelihood ratio in a neighborhood of the true
parameter. Related asymptotic Wilks phenomena for generalized likelihood ratio
statistics, including nonparametric and semiparametric settings, were studied
by Fan, Zhang and Zhang~\cite{FanZhangZhang2001Wilks}.

For logistic regression, this local quadratic viewpoint is correct in the
classical regular regime, but it is not robust to the regimes considered in
this paper. In particular, when the margins are large, the Fisher information
may degenerate, the MLE may fail to exist, and no useful uniform quadratic
expansion is available. Our worst case analysis does not rely on such an
expansion. In the Gaussian random design case, where a quadratic expansion can
be useful, we combine it with a separate argument for the complementary regime
where the expansion no longer controls the likelihood globally.

A different asymptotic perspective was developed by Sur, Chen and
Candès~\cite{sur2019likelihood} and Sur and Candès~\cite{sur2019modern}.
These works show that the classical Wilks calibration can fail even for
Gaussian random design in proportional high-dimensional asymptotics. In
particular, \cite{sur2019likelihood} studies likelihood ratio tests for a fixed
number of coordinates in high-dimensional logistic regression and obtains a
rescaled chi-square limit rather than the classical chi-square limit. This is
quite different from our setting. We study the full log-likelihood ratio
statistic in finite samples, with no assumption on the design, no restriction
on \(\thetastar\), and no requirement that the MLE exists.

There is also a secondary asymptotic message in our results. If one insists on
the actual Wilks approximation under Gaussian random design, then the aspect
ratio \(d/n\) is not the only relevant scale. In the null Gaussian design model,
we show that approximation of \(2\Lambda_n^{\operatorname{log}}(0)\) by
\(\chi_d^2\) in Kolmogorov distance is governed by \(d^{3/2}/n\): it holds when
\(d^{3/2}/n\to0\), and fails when this condition does not hold. Thus the
proportional regime \(d/n\to\kappa>0\) is already beyond the classical Wilks
scale, while our finite sample bounds remain valid across all these regimes.
The sufficient condition \(d^{3/2}/n\to0\) has a relevant precedent in
Portnoy's version of the Wilks theorem~\cite{Portnoy1988}, and also appears
in recent fourth order finite sample expansions~\cite{Spokoiny2025SLS} under
different regularity assumptions.

\paragraph{Finite sample likelihood expansions and logistic MLE.}

A closely related line of work studies likelihood ratio statistics through
finite sample versions of Wilks' theorem. The approach of
Spokoiny~\cite{Spokoiny2012Parametric}, and its extensions
\cite{AndresenSpokoiny2014,SpokoinyZhilova2015}, gives nonasymptotic Fisher
and Wilks expansions under possible model misspecification and growing
dimension. These results are closer in spirit to the finite sample likelihood
inference considered here. At the same time, their mechanism is still local and
quadratic: one proves concentration of the MLE in a suitable neighborhood and
then controls the error of a quadratic expansion of the likelihood process. For
logistic regression, this is precisely the step that becomes unstable when the
margins are large.

There is also a substantial finite sample literature on the MLE and empirical
risk minimization in logistic regression. Bach~\cite{Bach2010SelfConcordant}
introduced a self-concordant analysis for logistic regression, and
Ostrovskii and Bach~\cite{ostrovskii2021finite} developed a general finite
sample theory for \(M\)-estimators based on self-concordance. More recently,
Kuchelmeister and van de Geer~\cite{KuchelmeisterVandeGeer2024} and Hsu and
Mazumdar~\cite{HsuMazumdar2024} studied finite sample estimation under Gaussian
designs, with particular attention to small noise or high signal regimes. The
work of Chardon, Lerasle and Mourtada~\cite{chardon2024logistic} gives sharp
finite sample guarantees for the existence and excess risk of the logistic MLE
under regular random designs, with explicit dependence on \(\|\thetastar\|_2\).

In contrast, we do not study the estimation error of the MLE, and in our main
fixed design results we do not require the MLE to exist at all. The statistic
\(\Lambda_n^{\operatorname{log}}(\thetastar)\) remains well defined even when
the infimum of the empirical logistic loss is approached only at infinity. We
note that, in the Gaussian random design case, we do use the quadratic theory
of~\cite{chardon2024logistic} in the regular regime, but we complement it by a
separate argument for the complementary regime where such a quadratic
analysis no longer controls the likelihood globally.

\paragraph{Shtarkov sums and regret bounds for logistic regression.}

A different line of work comes from the connection between likelihood ratios,
logarithmic loss regret, and Shtarkov sums. In the fully sequential setting,
the covariates are also revealed online.  At round
\(i\), after observing \(x_{1:i}\) and \(y_{1:i-1}\), the strategy assigns
probabilities \(q_i(\cdot\mid x_{1:i},y_{1:i-1})\) to the next label. If the
strategy has deterministic regret \(R_n\) against the logistic class, then
\[
\sum_{i=1}^n
-\log q_i(y_i\mid x_{1:i},y_{1:i-1})
-
\inf_{\theta\in\R^d}
\sum_{i=1}^n
-\log \sigma(y_i\langle x_i,\theta\rangle)
\le R_n
\]
implies
\[
\P_{\thetastar}\Big(
\Lambda_n^{\operatorname{log}}(\thetastar)
>
R_n+\log\Big(\frac{1}{\delta}\Big)
\Big)
\le \delta.
\]
This viewpoint is closely related to the idea of converting online guarantees into confidence sets due to
Abbasi-Yadkori, P{\'a}l and Szepesv{\'a}ri~\cite{AbbasiYadkorietal2012} and to
recent sequential confidence constructions for logistic and generalized linear
models~\cite{lee2024improved,kirschner2025confidence,clerico2025confidence}.

However, standard sequential regret bounds for logistic regression
\cite{kakade2005online} are not directly suited to our regime. In the worst case
for online prediction, proper logistic regret necessarily depends on the comparator
norm, which can make the resulting confidence bounds vacuous for large
\(\|\thetastar\|_2\); see \cite{foster2018logistic}. The more relevant setting
is \emph{transductive}: the whole design \(x_1,\ldots,x_n\) is known in advance,
and only the labels are predicted sequentially. Recent work of Qian, Rakhlin
and Zhivotovskiy~\cite{QianRakhlinZhivotovskiy2026} shows that transductive
priors can yield finite sample logistic regret bounds without dependence on the
design vectors or on the norm of the optimal parameter.

The natural extremal object behind sharp transductive regret is the Shtarkov
sum, introduced in universal coding by Shtarkov~\cite{Shtarkov1987}. Related
regret and redundancy quantities have a long history; see, for example,
\cite{Rissanen1996FisherSC}. For logistic models, early geometric analysis
appears in Dowty~\cite{dowty2014volumes}, while more recent work studies
precise minimax regret in special settings
\cite{jacquet2021precise,ShamirSzpankowski2021,JacquetShamirSzpankowski2022,
DrmotaJacquetWuSzpankowski2024,
drmota2025precise,DrmotaJacquetWuSzpankowski2026}. Many of these results, while
sharp, are either asymptotic or tailored to particular design ensembles or
parameter regimes. We also note that nonasymptotic Bayesian analyses of regret
often rely on local quadratic behavior of the Fisher information
\cite{Barron1999BayesPerformance,kakade2005online}. In contrast, our fixed
design analysis gives sharp finite sample control of the logistic Shtarkov sum
for arbitrary deterministic designs and without any restriction on
\(\|\thetastar\|_2\). This yields the upper bound \(d\log(en/d)\), and our
lower bounds show that this scale is unimprovable in the worst case.

\subsection{Structure of the paper}

The rest of the paper is organized as follows.
\begin{itemize}
\item In \Cref{sec:log-reg} we prove the main worst-case result. We first
reduce the problem to a statement about likelihood ratios over subspaces, then
prove the upper bound through Shtarkov sums and hyperplane arrangements, and
finally give a matching lower bound using Vandermonde subspaces.

\item In \Cref{sec:logfree} we study regimes where the logarithmic
factor is absent. We prove nonasymptotic \(\chi^2\)-type bounds for
univariate generalized linear models and for Gaussian random design. We also
study behavior near the origin for rotationally invariant designs, where the
worst-case logarithmic factor does not appear.

\item In \Cref{sec:sharpness-boundaries} we collect several boundary and
sharpness results. We prove sharp bounds for logistic Shtarkov sums, including
matching leading constants when $d$ divides $n$ and $n/d\to\infty$. We also
treat the special two-dimensional case, showing that uniform \(\chi^2\)-type
behavior is already impossible there; the matching upper bound is stated in
\Cref{thm:d2-logloglog-upper}. Finally, under Gaussian random design at the origin, we show that the validity
of the classical Wilks approximation for the full log-likelihood ratio statistic
is governed by the scale \(d^{3/2}/n\), rather than by the usual aspect ratio
\(d/n\).

\item \Cref{app:ai-disclosure-d2} gives the AI disclosure, the matching upper bound in dimension \(d=2\) with its proof sketch, and the final prompt used in our tests for that result. 
\end{itemize}

\section{Worst-case likelihood ratio in high dimensions}
\label{sec:log-reg}

In this section, we analyze the likelihood-ratio statistic in dimensions
\(d\ge 2\). The behavior depends substantially on the dimension. The
univariate case admits sharper \(\chi^2_1\)-type bounds, in fact for a broader
class of generalized linear models, and is treated separately in
\Cref{sec:univariatemodels}. The main lower bound construction in the present
section requires \(d\ge 3\). The remaining case \(d=2\) exhibits unusual
behavior. In \Cref{sec:bivariatecase} we prove a lower bound of the smaller order
\(\log\log\log n\), and the matching upper bound is given in
\Cref{thm:d2-logloglog-upper} in the appendix.

Let us recall our model. Assume that
\(Y_1,\dots,Y_n\in\{-1,1\}\) are independent and satisfy
\begin{equation}
\label{eqn:logistic-model}
\P_{\thetastar}(Y_i=1) = \sigma(\langle x_i,\thetastar\rangle),
\qquad
\sigma(t)=\frac{1}{1+\e^{-t}},
\end{equation}
for deterministic covariates $x_1,\dots,x_n \in \R^d$. In this case, writing
\[
\ell_n(\theta)
\defn
\sum_{i=1}^n \log\big(1+\exp(-Y_i\langle x_i,\theta\rangle)\big),
\]
the log-likelihood ratio statistic~\eqref{eq:def-logistic-llr} takes the form
\[
\Lambda^{\operatorname{log}}_n(\thetastar)
=
\ell_n(\thetastar)-\inf_{\theta\in\R^d}\ell_n(\theta).
\]

In this work, we are interested in the extreme behavior of this quantity, analyzed either through its expected value or its quantiles. In this section, we study the (upper) quantiles of scalar random variables. For $Z$ a scalar random variable, we define 
\[
\Quant{1-\delta}(Z) \defn \inf\Big\{\, z \in \R \mid \P(Z \leq z) \geq 1-\delta\,\Big\}.
\]

First, in the \emph{fixed design} setup consider the quantities
\[
\begin{aligned}
\WorstCaseFixedDesign{n}{d} 
&\defn 
\sup_{\substack{x_1, \dots, x_n \in \R^d \\ \thetastar \in \R^d}}
\E \Lambda_n^{\operatorname{log}}(\thetastar;x_{1:n}; Y_{1:n}), \\
\WorstCaseFixedDesignQuant{1-\delta}{n}{d}
&\defn 
\sup_{\substack{x_1, \dots, x_n \in \R^d \\ \thetastar \in \R^d}}
\Quant{1-\delta}\Big(\Lambda_n^{\operatorname{log}}(\thetastar;x_{1:n}; Y_{1:n})\Big).
\end{aligned}
\]
Above, the randomness is over the realization of $Y_1, \dots, Y_n$.
Similarly, in the case of \emph{random design}, we consider 
\[
\WorstCaseRandomDesign{n}{d} 
\defn 
\sup_{\thetastar \in \R^d, P}
\E \Lambda_n^{\operatorname{log}}(\thetastar;X_{1:n}; Y_{1:n}), 
\quad \mbox{and} \quad 
\WorstCaseRandomDesignQuant{1-\delta}{n}{d}
\defn 
\sup_{\thetastar \in \R^d, P}
\Quant{1-\delta}\Big(\Lambda_n^{\operatorname{log}}(\thetastar;X_{1:n}; Y_{1:n})\Big).
\]
The randomness is over the realization of $Y_{1:n} \mid X_{1:n}$ according 
to the logistic model and $X_{1:n} \simiid P$.
Here, supremum is over probability measures $P$ on $\R^d$.

Our first main result follows. 

\begin{theorem}[Worst-case behavior of the logistic LLR statistic]
\label{thm:main-result}

For $n \geq d \geq 3$ and any $\delta \in (0, \tfrac{1}{\e}]$, it holds that
\begin{subequations}
\begin{equation}
\label{eqn:main-worst-case-result-quantile}
\WorstCaseFixedDesignQuant{1-\delta}{n}{d}
\asymp
\WorstCaseRandomDesignQuant{1-\delta}{n}{d}
\asymp d \log \Big(\frac{\e n}{d}\Big) + \log\frac{1}{\delta}.
\end{equation}
In particular, it holds that
\begin{equation}
\label{eqn:main-worst-case-result-expected}
\WorstCaseFixedDesign{n}{d}
\asymp
\WorstCaseRandomDesign{n}{d}
\asymp
d \log \Big(\frac{\e n}{d}\Big).
\end{equation}
\end{subequations}
\end{theorem}

\noindent In~\Cref{sec:proof-thm-main-result}, we  deduce~\Cref{thm:main-result} as a consequence of~\Cref{lem:reformulation-via-subspaces},~\Cref{prop:upper-bound-on-LLR-fixed-design} and \Cref{prop:strong-lower-bound-vandermonde}, as described below.

\subsection{Reformulation via subspaces and upper bounds of \Cref{thm:main-result}}

The statement in~\Cref{thm:main-result} is perhaps more naturally formulated in terms of subspaces. We define for a subspace $W \subset \R^n$, a vector $v \in \R^n$ and $\eps \in \{-1,1\}^n$, the quantity
\begin{equation}
\label{eq:subspace-reformulation}
\Gamma_W(\eps; v)
	= \sup_{w \in W} \log \frac{p_w(\eps)}{p_v(\eps)},
		\quad \mbox{where} \quad
	p_v(\eps) = \prod_{i=1}^n \sigma(\eps_i v_i).
\end{equation}
Whenever the law of $\eps$ is not displayed explicitly, we understand that $\eps\sim p_v$.
We also define 
\begin{equation}
\label{eqn:definition-of-largest-subspace-LLR}
E^\star(n,k) = 
\sup_{\substack{W \subset \R^n\\ \dim(W) \leq k}} \sup_{v \in W} \E_{\eps \sim p_v} \Gamma_W(\eps; v), 
\quad \mbox{and} \quad 
Q^\star_{1-\delta}(n,k) = 
\sup_{\substack{W \subset \R^n\\ \dim(W) \leq k}} \sup_{v \in W} \Quant{1-\delta}\big(\Gamma_W(\eps; v)\big)
\end{equation}
where the quantile in the definition of $Q^\star_{1-\delta}(n,k)$ is taken under $\eps\sim p_v$.

\begin{lemma}[Reformulation via subspaces]
\label{lem:reformulation-via-subspaces}
Fix integers $n, d \geq 1$.
Let $x_1, \dots, x_n \in \R^d$ and let $\thetastar \in \R^d$. Form a design matrix $X \in \R^{n \times d}$ which has rows $x_1^\T, \dots, x_n^\T$. Set $W = \ran(X)$ and put $v^\star = X\thetastar$. 
Then, the following hold true: 
\begin{enumerate}[label=(\roman*)]
\item 
\label{item:rewrite-expected-LLR-via-subspace}
the LLR statistic satisfies $\Lambda_n^{\operatorname{log}}(\thetastar;x_{1:n}; Y_{1:n}) = \Gamma_W(Y_{1:n}; v^\star)$; and
\item 
\label{item:worst-case-reduction}
in fixed design, the worst-case LLR statistic satisfies
\[
\WorstCaseFixedDesign{n}{d} = \max_{1 \leq k \leq \twomin{n}{d}} E^\star(n,k) 
\quad \mbox{and} 
\quad 
\WorstCaseFixedDesignQuant{1-\delta}{n}{d}
=
\max_{1 \leq k \leq \twomin{n}{d}} Q^\star_{1-\delta}(n,k),
\]
for any $\delta \in (0, 1)$.
\end{enumerate}
\end{lemma}
\begin{proof}
Claim~\ref{item:rewrite-expected-LLR-via-subspace} follows immediately from the definition of the log-likelihood statistic and the observational model~\eqref{eqn:logistic-model}. Claim~\ref{item:worst-case-reduction}
follows as 
\[
\WorstCaseFixedDesign{n}{d} = 
\sup_{X \in \R^{n \times d}}
\sup_{\thetastar \in \R^d}
\E \Gamma_{\ran(X)}(\eps; X\thetastar) = 
\sup_{\substack{W \subset \R^n \\ \dim(W) \leq d}}
\sup_{v \in W}
\E \Gamma_{W}(\eps; v) = 
\sup_{1 \leq k \leq \min\{n,d\}} 
E^\star(n,k).
\]
The reduction in the last equality holds by letting $E^\star(n,0) = 0$ since a design of rank zero leads to $W=\{0\}$ and $\Gamma_{\{0\}}(\eps;0) = 0$. Thus, the index $k=0$ never attains the maximum.
Lastly, a completely analogous argument holds for the quantile. 
\end{proof}

The next result provides some bounds on the stochastic behavior of $\Gamma_{W}(\eps; v)$. 
\begin{proposition}[Upper bounds on the logistic LLR]
\label{prop:upper-bound-on-LLR-fixed-design}
Fix $1\leq k \leq n$. 
For any $k$-dimensional subspace $W \subset \R^n$ and any $v \in \R^n$, let $\eps\sim p_v$. Then, for any $\delta \in (0, 1)$, it holds that 
\[
\E \Gamma_W(\eps; v) \leq k \log \frac{\e n}{k}, \quad \mbox{and} \quad 
\Quant{1-\delta}\Big(\Gamma_W(\eps; v)\Big) \leq k \log \Big( \frac{\e n}{k} \Big) + \log\Big(\frac{1}{\delta}\Big).
\]
\end{proposition}

\begin{remark}[Equivalent statements to \Cref{prop:upper-bound-on-LLR-fixed-design}]
By way of the equivalence in~\Cref{lem:reformulation-via-subspaces}\ref{item:rewrite-expected-LLR-via-subspace}, 
we can equivalently formulate~\Cref{prop:upper-bound-on-LLR-fixed-design} as follows. 
Suppose $x_1, \dots, x_n \in \R^d$, 
fix $\thetastar \in \R^d$ and $\delta \in (0, 1)$. 
Let $r = \rank(X) \leq \min\{n,d\}$, where $X \in \R^{n \times d}$ is the design matrix formed with rows $x_i^\T$.
If $r=0$, then $\Lambda_n^{\operatorname{log}}(\thetastar;x_{1:n};Y_{1:n})=0$ identically. If $r\geq1$, then the log-likelihood ratio statistic satisfies
\[
\E \Lambda_n^{\operatorname{log}}(\thetastar;x_{1:n}; Y_{1:n}) \leq r \log \frac{\e n}{r} \quad \mbox{and} \quad 
\P\Big(\Lambda_n^{\operatorname{log}}(\thetastar;x_{1:n}; Y_{1:n}) \leq r \log \frac{\e n}{r} + \log\Big(\frac{1}{\delta}\Big)\Big) \geq 1-\delta, 
\]
and these bounds hold uniformly over all deterministic 
choices of $x_1,\dots,x_n$ and all $\thetastar \in \R^d$.
\end{remark}

The proof of the upper bound is based on the Shtarkov sum associated with the
logistic model, introduced by Shtarkov~\cite{Shtarkov1987} in the context of
universal coding. For the sake of presentation, we discuss this object in
detail in \Cref{sec:shtarkov}; here we only use the following immediate
consequence of \Cref{thm:shtarkov-logistic}. For a subspace \(W\subset \R^n\),
define
\[
\mathcal S^{\operatorname{log}}(W)
\defn
\sum_{\eps\in\{-1,1\}^n}\sup_{w\in W}p_w(\eps).
\]
This quantity is equivalently the exponential moment of the likelihood-ratio
statistic at \(\lambda=1\). Indeed, for every \(v\in\R^n\), since
\(p_v(\eps)>0\) for all \(\eps\in\{-1,1\}^n\), we have
\[
\E_{\eps\sim p_v}\exp\bigl(\Gamma_W(\eps;v)\bigr)
=
\sum_{\eps\in\{-1,1\}^n}
p_v(\eps)
\sup_{w\in W}\frac{p_w(\eps)}{p_v(\eps)}
=
\sum_{\eps\in\{-1,1\}^n}\sup_{w\in W}p_w(\eps)
=
\mathcal S^{\operatorname{log}}(W).
\]
In particular, this exponential moment is independent of the center \(v\).

The next lemma is the key input for the upper bound. Its proof is deferred, in
essence, to \Cref{sec:shtarkov}, where we also discuss the sharpness of this
estimate and give a more detailed analysis of the moment generating function at
\(\lambda=1\).

\begin{lemma}[Shtarkov sum upper bound for subspaces]
\label{lem:shtarkov-moment-bound-subspaces}
Let \(1\le k\le n\), and let \(W\subset\R^n\) be a \(k\)-dimensional subspace.
Then, for every \(v\in\R^n\),
\[
\E_{\eps\sim p_v}\exp\bigl(\Gamma_W(\eps;v)\bigr)
=
\mathcal S^{\operatorname{log}}(W)
\le
\sum_{\ell=0}^k \binom{n}{\ell}
\le
\left(\frac{\e n}{k}\right)^k .
\]
\end{lemma}

\begin{proof}
 Writing \(W=\ran(U)\) for some matrix
\(U\in\R^{n\times k}\), the quantity \(\mathcal S^{\operatorname{log}}(W)\) is
exactly the Shtarkov sum of a \(k\)-dimensional logistic model with design
vectors given by the rows of \(U\). Therefore \Cref{thm:shtarkov-logistic}
gives
\[
\mathcal S^{\operatorname{log}}(W)
\le
\sum_{\ell=0}^k \binom{n}{\ell}
\le
\left(\frac{\e n}{k}\right)^k .
\]
The claim follows.
\end{proof}

\begin{proof}(\emph{of \Cref{prop:upper-bound-on-LLR-fixed-design}})
By \Cref{lem:shtarkov-moment-bound-subspaces},
\[
\E_{\eps\sim p_v}\exp\bigl(\Gamma_W(\eps;v)\bigr)
\le
\left(\frac{\e n}{k}\right)^k.
\]
Jensen's inequality gives
\[
\E_{\eps\sim p_v}\Gamma_W(\eps;v)
\le
\log\E_{\eps\sim p_v}\exp\bigl(\Gamma_W(\eps;v)\bigr)
\le
k\log\frac{\e n}{k}.
\]
Similarly, Markov's inequality gives, for every \(t>0\),
\[
\P_{\eps\sim p_v}\left( \Gamma_W(\eps;v)>k\log\Big(\frac{\e n}{k}\Big)+t\right)
\leq \e^{-t}.
\]
Taking \(t=\log(1/\delta)\) yields the claimed quantile bound.
\end{proof}

\subsection{Lower bounds on the logistic LLR and Vandermonde subspaces}

The main difficulty in proving~\Cref{thm:main-result} is obtaining lower bounds that match~\Cref{prop:upper-bound-on-LLR-fixed-design}. It requires us to construct worst-case pairs $(W,v)$, where $W \subset \R^n$ and $v \in W$. In this section, we present multiple lower bounds on the logistic LLR. First, a simple lower bound at the origin, in fixed design, in \Cref{prop:fixed-design-lower}.
Second, in~\Cref{prop:strong-lower-bound-vandermonde}, we present a lower bound that contains the needed multiplicative logarithmic inflation required to establish~\Cref{thm:main-result}.

\begin{proposition}[Lower bounds on the logistic LLR in fixed design]
\label{prop:fixed-design-lower}
Fix $1\leq k \leq n$. 
For any $k$-dimensional subspace $W \subset \R^n$, it holds that 
\[
\Gamma_W(\eps; 0) \geq \half \|P_W \eps\|_2^2.
\]
Above, $P_W$ denotes the orthogonal projection onto the subspace $W$. Consequently, $ \E \Gamma_W(\eps; 0) \geq \frac{k}{2}$. In particular if $k \geq 3$ and $\delta \in (0, 1/2]$, then 
$\Quant{1-\delta} \big(\Gamma_W(\eps; 0)\big) \gtrsim k$. 
\end{proposition}

\begin{proof}
We can write $W = \ran(U)$ for some $U \in \R^{n \times k}$ with $U^\T U = I_k$. Set $h(\delta) = (1+\delta) \log(1+\delta) + (1-\delta) \log(1-\delta)$. Note that $h(\delta) \geq \delta^2$ for $|\delta| \leq 1$. Applying the variational characterization stated and proved later in~\Cref{lem:var-char} with $v=0$ (hence $p_i=1/2$ and
$\kl{\Ber{\frac{1+\delta}{2}}}{\Ber{\tfrac{1}{2}}}=h(\delta)/2$), we obtain

\[
\Gamma_W(\eps; 0) = \half \inf_{\delta \in [-1,1]^n, U^\T \delta = U^\T \eps} 
\sum_{i=1}^n h(\delta_i) 
\geq  \half \|UU^\T \eps\|_2^2 = 
\half \|P_W \eps\|_2^2
\]
which furnishes the lower bound. Put $Z=\|P_W\eps\|_2^2$. Since
\[
\E Z=\trace(P_W)=k,
\qquad
\Var(Z)=4\sum_{i<j}(P_W)_{ij}^2\leq 2k,
\]
the expectation bound follows. Moreover, Cantelli's inequality gives, for
$k\geq3$,
\[
\P\left(Z\leq\frac{k}{10}\right)
\leq
\frac{2k}{2k+(9k/10)^2}
\leq \frac{200}{443}<\frac12.
\]
Thus, for every $t<k/20$,
$\P(\Gamma_W(\eps;0)\leq t)<1/2\leq1-\delta$, and hence
$\Quant{1-\delta}(\Gamma_W(\eps;0))\geq k/20$.
\end{proof}

\begin{proposition}[Lower bound for coordinate subspaces]
\label{prop:coordinate-subspace-lower}
There exists a constant $c > 0$, such that for any $1 \leq k \leq n$ 
and any $\delta \in (0, 1/2]$, it holds that 
\[
\sup_{\lambda \in \R} 
\Quant{1-\delta}\Big(
	\Gamma_{\Span\{e_1, \dots, e_k\}}
	\big(\eps; \lambda (e_1 +\cdots +e_k)\big)
	\Big)
\geq 
c \, \Big(k + \log \frac{1}{\delta}\Big).
\]
\end{proposition}
\begin{proof}
	We use the shorthand notation
	\[
	W_k \defn \Span\{e_1,\dots,e_k\},
	\quad \mbox{and} \quad
	\1_k \defn e_1+\cdots+e_k.
	\]
	Fix $\lambda \in \R$. Then, by construction, 
	for every $\eps \in \{-1,1\}^n$, 
	\[
	\Gamma_{W_k}(\eps;\lambda \1_k)
	=
	\sup_{w \in W_k}\log \frac{p_w(\eps)}{p_{\lambda \1_k}(\eps)}
	=
	\sup_{w \in W_k}\sum_{i=1}^k
	\log \frac{\sigma(\eps_i w_i)}{\sigma(\eps_i \lambda)}
	=
	-\sum_{i=1}^k \log \sigma(\eps_i \lambda).
	\]
	First consider $\lambda = 0$. Then, $\Gamma_{W_k}(\eps;0) = 
	k \log 2$, for all $\eps \in \{-1,1\}^n$, so
	\begin{equation}
	\label{ineq:lower-from-zero-for-coordinate-subspace}
	\Quant{1-\delta}\big(\Gamma_{W_k}(\eps;0)\big)
	=k\log 2.
	\end{equation}
	Next, fix $\gamma\in(\delta,1)$ and consider
	$v = \log(\tfrac{1-\gamma^{1/k}}{\gamma^{1/k}})\1_k$. 
	On the event $A = \{\eps_1 = \cdots = \eps_k = -1\}$, 
	we have,
	\[
	\Gamma_{W_k}\Big(\eps; \log(\tfrac{1-\gamma^{1/k}}{\gamma^{1/k}})\1_k\Big)
	=
	-\sum_{i=1}^k \log \sigma\Big(-\log(\tfrac{1-\gamma^{1/k}}{\gamma^{1/k}})\Big)
	= \log \frac{1}{\gamma}.
	\]
	Additionally, with $\lambda = \log(\tfrac{1-\gamma^{1/k}}{\gamma^{1/k}})$, we have $\P_{p_{\lambda \1_k}}(A) = \gamma>\delta$. Therefore,
	\[
	\P_{p_{\lambda \1_k}}\Big(
	\Gamma_{W_k}(\eps; \lambda\1_k)
	\geq \log\frac{1}{\gamma}
	\Big)\geq\gamma>\delta.
	\]
	By the definition of the $(1-\delta)$-quantile and taking the supremum
	over $\gamma\in(\delta,1)$,
	\begin{equation}
	\label{ineq:lower-from-v-for-coordinate-subspace}
	\sup_{\lambda\in\R}\Quant{1-\delta}\Big(
	\Gamma_{W_k}(\eps; \lambda\1_k)
	\Big)
	\geq \log\frac{1}{\delta}.
	\end{equation}
	Combining~\cref{ineq:lower-from-zero-for-coordinate-subspace} and \cref{ineq:lower-from-v-for-coordinate-subspace}, we obtain
	\[
	\sup_{\lambda\in\R}
	\Quant{1-\delta}\Big(
	\Gamma_{W_k}(\eps;\lambda \1_k)
	\Big)
	\geq
	\twomax{k \log 2}{\log \frac{1}{\delta}} 
	\geq \frac{\log 2}{1 + \log 2}\Big( k + \log \frac{1}{\delta}\Big),
	\]
	as required.
\end{proof}

Next, we study a family of subspaces, which we refer to as \emph{Vandermonde subspaces}, $V_{n,k} \subset \R^n$. These subspaces are at the core of the argument for the lower bound of Theorem~\ref{thm:main-result} showing the necessity of the logarithmic factor.
Given $t_1,\dots,t_n\in[0,1]$, these subspaces are defined as
\[
V_{n, k}(t_1, \dots, t_n) = \Big\{\,\big(p(t_1), \dots, p(t_n)\big) \mid p \in \R[t],~ \deg p \leq k - 1 \,\Big\}.
\]
Above, $\R[t]$ denotes the set of polynomials in the indeterminate $t$ with real coefficients. We refer to $V_{n,k}$ as a Vandermonde subspace because 
it can be realized as the range of a suitably constructed Vandermonde matrix. Indeed, if
$p \in \R[t]$ with $\deg p \leq k - 1$, then there exists 
$a = (a_0, \dots, a_{k-1}) \in \R^k$ such that 
$
p(t) = \sum_{j=0}^{k-1} a_j t^j$. 
Consequently, we have 
\[
\begin{pmatrix}
p(t_1) \\ 
\vdots \\	
p(t_n)
\end{pmatrix}
=
\begin{pmatrix}
1 & t_1 & \cdots & t_1^{k-1} \\
\vdots & \vdots & \ddots & \vdots \\
1 & t_n & \cdots & t_n^{k-1}
\end{pmatrix} \, 
a.
\]
Hence $V_{n,k}(t_1, \dots, t_n)$ is exactly the range of the $n \times k$ Vandermonde matrix above. In particular, if
$t_1, \dots, t_n$ are pairwise distinct and $k \leq n$, then $\dim(V_{n,k}(t_1, \dots, t_n)) = k$, as the first $k$ rows form a square
Vandermonde matrix with determinant $\prod_{1 \leq i < j \leq k} (t_j - t_i) \neq 0$. 
For $3 \leq k \leq n$, we also consider the vector 
\[
v_{n, k} = \sigma^{-1}\bigg(1-\frac{1}{n} \floor{\frac{k-1}{2}}\bigg) \1_n.
\]

\begin{proposition}
[Lower bounds for Vandermonde subspaces]
\label{prop:strong-lower-bound-vandermonde}
There exists a constant $c > 0$ such that 
for any $3 \leq k \leq n$, for any $\delta \in (0, \tfrac{1}{\e}]$, it holds that: 
\begin{enumerate}[label=(\roman*)]
\item \label{item:lower-bound-quantile-vandermonde-fixed}
in fixed design, we have 
\[
\sup_{\lambda \in \R} \inf_{0 \leq t_1 < \cdots < t_n \leq 1} \Quant{1-\delta}\Big(\Gamma_{V_{n,k}(t_1, \dots, t_n)}(\eps; \lambda \1_n)\Big)
\geq c \,\Big(k \log \frac{\e n}{k} + \log \frac{1}{\delta}\Big).
\]
\item \label{item:lower-bound-quantile-vandermonde-random}
for any probability measure $P$ on $[0,1]$ such that
$T_1, \dots, T_n \simiid P$ are almost surely distinct, in random design we have
\[
\sup_{\lambda \in \R} \Quant{1-\delta}\Big(\Gamma_{V_{n,k}(T_1, \dots, T_n)}(\eps; \lambda \1_n)\Big)
\geq c \,\Big(k \log \frac{\e n}{k} + \log \frac{1}{\delta}\Big),
\] 
where the randomness above is over both $T_1,\dots, T_n$ and $\eps_{1:n} \mid T_{1:n}$; 
\item \label{item:lower-bound-expected-vandermonde}
for any probability measure $P$ on $[0,1]$ such that
$T_1, \dots, T_n \simiid P$ are almost surely distinct, we have
\[
\E_{T_1,\dots,T_n \simiid P} \E_\eps \Gamma_{ V_{n,k}(T_1, \dots, T_n)}(\eps; v_{n,k})
\geq 
\inf_{0 \leq t_1 < \cdots < t_n \leq 1} 
\E_\eps \Gamma_{V_{n,k}(t_1, \dots, t_n)}(\eps; v_{n,k})
\geq c \, k \log \frac{\e n}{k}.
\]
\end{enumerate} 
\end{proposition}

The proof is deferred to \Cref{sec:proof-strong-lower-bound-vandermonde}.

\paragraph{Comparison of quantiles: Wilks' phenomenon versus nonasymptotic bounds.}

The limiting behavior of the LLR statistic in classical asymptotic theory,  known as Wilks' phenomenon~\cite{wilks1938large,chernoff1954likelihood},
predicts that $2\Lambda_n^{\log}(\thetastar) \to \chi^2_d$ in distribution as $n \to \infty$, with the dimension $d$ and parameter $\thetastar \in \R^{d}$ being fixed. Consequently, for any fixed $\delta \in (0,1/2]$, $\Quant{1-\delta}(2\Lambda_n^{\log}(\thetastar)) \to \Quant{1-\delta}(\chi^2_d)$, and \Cref{lem:chisquarequantiles} below shows that this quantity is exactly of order $d + \log(1/\delta)$.
By contrast, \Cref{prop:strong-lower-bound-vandermonde}\ref{item:lower-bound-quantile-vandermonde-fixed} shows that for the worst-case (Vandermonde) fixed design, $\Quant{1-\delta}(\Lambda_n^{\log}(\thetastar)) \gtrsim d\log(\e n/d) + \log(1/\delta)$ whenever $n\geq d\geq3$ and $\delta\in(0,1/\e]$.

\begin{lemma}
\label{lem:chisquarequantiles}
There is a constant $C > 0$ such that for any $d \geq 1$ and any $\delta \in (0, 1/2]$, it holds that 
\[
	\frac{1}{C}\, 
	\bigg(d + \log\Big(\frac{1}{\delta}\Big)\bigg) 
	\leq \Quant{1-\delta}(\chi^2_d) \leq 
	C\, \bigg(d + \log\Big(\frac{1}{\delta}\Big)\bigg) 
\]
\end{lemma}
The proof of \Cref{lem:chisquarequantiles} is deferred to Section~\ref{sec:remaining-and-aux}.

\subsection{The lower-bound factor $\log(n/d)$ comes from genuinely non-extreme values of $\theta^\star$.}

In this section we present complementary results showing that, for the Vandermonde design used in the deterministic lower bound construction, neither the pure noise regime nor the very large signal regime is responsible for the logarithmic factor. We begin with the case $\thetastar = 0$. For simplicity of the proof, we focus on the illustrative case $d = 3$, where the logarithmic factor already appears for the worst-case target $\thetastar$. Thus, at least for this design, the lower bound from \Cref{prop:strong-lower-bound-vandermonde} must be generated by genuinely intermediate values of the signal. The proof will essentially reduce the problem to a one-dimensional optimization, where the supremum over three-dimensional vectors is replaced by a supremum over their norm.

\begin{proposition}
\label{prop:vandermonde-zero-signal}
Assume $d = 3$, $n \geq 3$, and consider the Vandermonde design as used in the proof of the lower bound  \Cref{prop:strong-lower-bound-vandermonde}, namely
\[
x_i = (1, i, i^2),
\qquad i \in [n].
\]
Then, for every $\delta \in (0,1)$,
\[
\P\Bigl(\Lambda_n^{\operatorname{log}}(0) \leq 216\log(6/\delta)\Bigr)\geq 1-\delta.
\]
\end{proposition}

\begin{remark}
The approach used in the proof is quite general and goes beyond the special case of the Vandermonde design. We will see that the key quantity is the maximal diagonal entry of the corresponding orthogonal projector. We note, however, that the computation is still quite specific to $\thetastar = 0$, and in general it will not scale well in high dimensions: it might lead to suboptimal dependence on $d$ in higher dimension, but still allows us to remove any dependence on $n$ in our case.
\end{remark}

\begin{proof}
Since $\thetastar = 0$, we have
$
\prod_{i = 1}^n \sigma(\langle x_i, \thetastar\rangle) = 2^{-n}.$
Moreover, $Y = (Y_1, \ldots, Y_n)$ consists of independent Rademacher random variables. Using that for all real $t$,
\begin{equation}
	\sigma(t) = \frac{1}{1+\e^{-t}} = \frac 1 2 \cdot \frac{\e^{t/2}}{\cosh(t/2)} \, ,
\end{equation}
hence $\log(2\sigma(t)) = t/2 - \log\cosh(t/2)$. Define the subspace $V = \{(a + bi + ci^2)_{i = 1}^n : a, b, c \in \R\}$ and, for every $i \in [n]$ and $\theta \in \R^{d}$, let $v_i = \langle \theta, x_i \rangle/2$. Since $\sigma(0)=1/2$, we deduce the following identity
\begin{align}
\Lambda_n^{\log}(0)
	= \sup_{\theta \in \R^{d}} \sum_{i=1}^n
		\log\Big(\frac{\sigma(Y_i\langle x_i,\theta\rangle)}{\sigma(0)} \Big)
	= \sup\limits_{v \in V}\left\{\langle Y, v\rangle - \sum\limits_{i = 1}^n \log(\cosh(v_i))\right\} \, .
\end{align}
 Let $\Pi_{V} \in \R^{n \times n}$ denote the orthogonal projector onto $V$. For every $v \in V$, we have
\[
\langle Y, v\rangle = \langle \Pi_{V} Y, v\rangle \leq \|\Pi_{V} Y\| \|v\|,
\]
and by the same computation
\begin{equation}
\label{eq:boundonvi}
|v_i| \leq |\langle \Pi_{V} e_i, v\rangle| \leq \|\Pi_{V} e_i\| \|v\| = \sqrt{e_i^{\T}\Pi_{V}^2 e_i}\,\|v\| = \sqrt{(\Pi_{V})_{ii}}\,\|v\|.
\end{equation}
We can easily check that an orthogonal basis of $V$ is given by $(u_0)_i = 1$, $(u_1)_i = i - \frac{n+1}{2}$, $(u_2)_i = ((u_1)_i)^2 - \frac{n^2 - 1}{12}$. An elementary computation gives
\begin{equation}
\label{eq:projdiagonal}
(\Pi_{V})_{ii}
= \sum_{k = 0}^2 \frac{(u_k)_i^2}{\|u_k\|_2^2}
=
\frac1n
+\frac{12\left(i-\frac{n+1}{2}\right)^2}{n(n^2-1)}
+\frac{180\left(\left(i-\frac{n+1}{2}\right)^2-\frac{n^2-1}{12}\right)^2}{n(n^2-1)(n^2-4)}
\leq \frac9n.
\end{equation}
Now consider
$
\psi(x) = \frac{\log(\cosh(x))}{x^2},
$
for $x > 0$,
and set $M_n^2 = \max_i (\Pi_{V})_{ii}$. An elementary analysis shows that $\psi$ is decreasing on $(0,+\infty)$. Thus, by \eqref{eq:boundonvi}, we have
\[
\sum\limits_{i = 1}^n \log(\cosh(v_i))
\ge
\sum\limits_{i = 1}^n v_i^2 \psi\left(M_n \|v\|\right)
=
\|v\|^2 \psi\left(M_n \|v\|\right)
=
\frac{\log(\cosh(M_n \|v\|))}{M_n^2},
\]
and therefore, combining the above inequalities, we reduce the upper bound to a one-dimensional optimization by taking the supremum with respect to the norm of $v$:
\[
\Lambda_n^{\operatorname{log}}(0)
\le
\sup_{r \geq 0}\left(r \|\Pi_{V} Y\| - \frac{\log(\cosh(M_n r))}{M_n^2}\right)
=
\frac{1}{M_n^2} \sup\limits_{r \geq 0}\left(r M_n \|\Pi_{V} Y\| - \log(\cosh(r))\right).
\]
It is straightforward to check that for any $w \in [0,1]$, it holds that
$
\sup\limits_{r \geq 0}(w r - \log(\cosh(r))) \leq w^2.
$
Therefore, whenever $M_n \|\Pi_{V} Y\| \leq \frac{1}{2}$, it holds that
\[
\Lambda_n^{\operatorname{log}}(0) \leq \frac{M_n^2 \|\Pi_{V} Y\|^2}{M_n^2} = \|\Pi_{V} Y\|^2.
\]
Note that
$
\|\Pi_{V} Y\|^2 = \sum_{k = 0}^2 \left\langle Y, \frac{u_k}{\|u_k\|} \right\rangle^2.
$
Since for any $k \in \{0,1,2\}$ and $\lambda \in \R$ it holds that
$
\E\exp\left(\lambda \left\langle Y, \frac{u_k}{\|u_k\|} \right\rangle\right) \leq \exp(\lambda^2/2),
$
we have for any $t > 0$,
\[
\P(\|\Pi_{V} Y\|^2 \geq 3 t^2)
\le
\P\left(\max_{k \in \{0,1,2\}} \left|\left\langle Y, \frac{u_k}{\|u_k\|} \right\rangle\right| \geq t \right)
\le
6\exp(-t^2/2),
\]
which implies that, on the intersection of the event $\left\{M_n \|\Pi_{V} Y\| \leq \frac{1}{2}\right\}$ and an event of probability at least $1-\delta$,
\[
\|\Pi_{V} Y\|^2 \leq 6\log(6/\delta).
\]
Note that when $n \geq 216\log(6/\delta)$, on the event $\|\Pi_{V} Y\|^2 \leq 6\log(6/\delta)$ we have, using \eqref{eq:projdiagonal},
\[
\|\Pi_{V} Y\|^2 \leq 6\log(6/\delta) \leq \frac{n}{36} \leq \frac{1}{4M_n^2},
\]
and therefore the event $\left\{M_n \|\Pi_{V} Y\| \leq \frac{1}{2}\right\}$ also holds. This proves the statement for $n \geq 216\log(6/\delta)$. Otherwise, if $n < 216\log(6/\delta)$, we have
\[
\Lambda_n^{\operatorname{log}}(0) \leq \log\left(\frac{1}{2^{-n}}\right) \leq 216\log(2)\log(6/\delta) \leq 216\log(6/\delta).
\]
The claim follows.
\end{proof}

The next elementary result shows that the extreme regime where
$\min_{i \in [n]}|\langle x_i, \thetastar\rangle|$ is large corresponds
essentially to the one-dimensional case. This also confirms that in our lower
bound the worst-case example does not correspond to this extreme regime. The point is that the proof does {not} go through the Shtarkov sum at $\lambda=1$. Indeed, that route necessarily produces a logarithmic factor at the level of the moment generating function. Instead, we work at the scale $\lambda=1/2$. In this way, the square root appearing at $\lambda=1/2$ removes the $\log n$ factor that is unavoidable at $\lambda=1$.

\begin{proposition}
\label{prop:large-margin-fixed-design}
Assume that $n\geq 2, d \geq 1$ and
$
\min_{i\in[n]} |\langle x_i,\thetastar\rangle|
\ge
3\log(n).
$
Then, for every $\delta\in (0, 1)$,
\[
\P_{\thetastar}\left(
\Lambda_n^{\operatorname{log}}(\thetastar)
\le
2\log(2/\delta)
\right)
\ge
1-\delta.
\]
\end{proposition}

\begin{proof}
Set
$
m \defn \min_{i\in[n]} |\langle x_i,\thetastar\rangle|.
$
For $\theta \in \R^d$ and $\eps \in \{-1,1\}^n$, let
$
p_\theta(\eps)
\defn
\prod_{i=1}^n \sigma(\eps_i \langle x_i,\theta\rangle).
$
We have
\[
\E\exp\left(\frac{1}{2}\Lambda_n^{\operatorname{log}}(\thetastar)\right)
=
\sum_{\varepsilon \in \{-1, 1\}^n}\sqrt{p_{\thetastar}(\eps)}
\sup_{\theta \in \R^d}\sqrt{p_\theta(\eps)}
\le
\sum_{\varepsilon \in \{-1, 1\}^n}\sqrt{p_{\thetastar}(\eps)}.
\]
Now, given $\thetastar$, there is just one assignment of signs $\varepsilon^\prime$
such that
$
\varepsilon^\prime_i = \sign(\langle x_i, \thetastar\rangle)$ for all $i \in [n]$.
For $\varepsilon \in \{-1, 1\}^n$, let
$
d_{\operatorname{H}}(\varepsilon, \varepsilon^\prime)
$
denote the Hamming distance between $\varepsilon$ and $\varepsilon^\prime$.
Since $\sigma(\cdot)\leq 1$, and on every coordinate where
$\varepsilon_i\neq \varepsilon_i^\prime$ we have
\[
\sigma(\varepsilon_i\langle x_i,\thetastar\rangle)\leq \sigma(-m)\leq e^{-m},
\]
it follows that
\[
p_{\thetastar}(\varepsilon)
\le
\exp\bigl(-m\,d_{\operatorname{H}}(\varepsilon,\varepsilon^\prime)\bigr).
\]
Therefore, it holds that
\[
\sum_{\varepsilon \in \{-1, 1\}^n}\sqrt{p_{\thetastar}(\eps)}
\le
\sum_{k=0}^n \binom{n}{k} e^{-mk/2}
=
\bigl(1+e^{-m/2}\bigr)^n
\le
\left(1 + \frac{1}{n^{3/2}}\right)^n
\leq 2.
\]
Then, by Markov's inequality, for any $t>0$,
\begin{equation}
\P\big(\Lambda_n^{\log}(\thetastar) > t\big)
	\leq \E\big[ \exp\big(\Lambda_n^{\log}(\thetastar)/2 \big)\big] \cdot \e^{-t/2}
	\leq 2 \e^{-t/2} \, ,
\end{equation}
and the conclusion follows by plugging $t=2\log(2/\delta)$.
\end{proof}

\subsection{Proofs}
\label{sec:main-remaining-proofs}

\subsubsection{Proof of~\Cref{prop:strong-lower-bound-vandermonde}}
\label{sec:proof-strong-lower-bound-vandermonde}

We start with Claim~\ref{item:lower-bound-quantile-vandermonde-fixed}.
Fix $0\le t_1<\cdots<t_n\le 1$. We 
use the shorthand notation
\[
W \defn V_{n,k}(t_1,\dots,t_n),
\quad \mbox{and} \quad
r \defn \floor{\frac{k-1}{2}}.
\]
Choose auxiliary points $t_0 < t_1$ and $t_{n+1} > t_n$.
Define 
\[
u_i \defn \frac{t_{i-1}+t_i}{2},\quad 
v_i \defn \frac{t_i+t_{i+1}}{2}, \quad 
q_i(t)\defn (t-u_i)(t-v_i), \quad 
\mbox{for}~i\in[n],
\]
Then $q_i(t_i)<0$, while $q_i(t_\ell)>0$ for every $\ell\neq i$. Hence, for every
$I\subset[n]$ with $|I|\le r$, the polynomial 
$p_I(t)\defn \prod_{j\in I} q_j(t)$ 
has degree at most $2|I|\le 2r\le k-1$. 
Hence 
\[
w^{(I)}\defn \bigl(p_I(t_1),\dots,p_I(t_n)\bigr)\in W,
\]
and its sign pattern is
\begin{equation}
\sign(w^{(I)}_j)=
\begin{cases}
-1,& j\in I,\\
 1,& j\notin I.
\end{cases}
\end{equation}
Therefore, for every $\eps\in\{-1,1\}^n$ with at most $r$ negative coordinates, there exists
$w\in W$ such that
\begin{equation}
\label{eq:achievable-sign-pattern}
\sign(w_i)=\eps_i \quad	\text{for all} \quad i \, ,
\end{equation}
and hence
$\sup_{a>0} p_{aw}(\eps)=1$. Take $s \defn \twomax{1}{\floor{\frac{r}{4}}}$.
Also set $\lambda_0 = \sigma^{-1}(1-\tfrac{s}{n})$.
If $\eps\sim p_{\lambda_0\1_n}$, then by construction, 
\[
N(\eps) = \left|\Big\{i\in[n] \mid \eps_i=-1\Big\}\right| \sim 
\mathsf{Bin}\Big(n,\frac{s}{n}\Big).
\]
On the event $\{s\leq N(\eps)\leq r\}$, we have 
\[
\Gamma_W(\eps;\lambda_0\1_n)
\ge -\log p_{\lambda_0\1_n}(\eps)
= N(\eps)\log\frac ns + \bigl(n-N(\eps)\bigr)\log\frac{1}{1-s/n}
\ge s\log\frac ns.
\]
If $r \leq 3$, then $s=1$. Consequently, 
\begin{equation}
\label{eq:bin-event-small-r}
\P( N(\eps) \in [s, r]) 
	\geq \P( N(\eps) = 1) 
	= \Big(1-\frac{1}{n}\Big)^{n-1}
	> \frac{1}{\e}.
\end{equation}
If $r \geq 4$, then $s = \floor{\tfrac{r}{4}}$, and hence $r \geq 4s$.
Since $s$ is a median of $\mathsf{Bin}(n, \tfrac{s}{n})$,
we have by the multiplicative Chernoff bound that  
\begin{equation}
\label{eq:bin-event-large-r}
\P(N(\eps)\in [s, r])
	= \P(N(\eps) \geq s) - \P(N(\eps) > r)
	\geq \P(N(\eps) \geq s) - \P(N(\eps)\geq 4s)
	\geq \frac{1}{2} - \Bigl(\frac{\e^3}{4^4}\Bigr)^s > \frac{1}{\e}. 
\end{equation}
Combining the previous three displays, we obtain 
\[
\Quant{1-\delta}\Big(\Gamma_W(\eps;\lambda_0\1_n)\Big)
\geq
\Quant{1-\tfrac{1}{\e}}\Big(\Gamma_W(\eps;\lambda_0\1_n)\Big)
\geq 
s\log\frac{n}{s},
\]
as $\delta \in (0, \tfrac{1}{\e}]$.
If $r\leq 3$, then $k \leq 8$ and $s=1$. In this case, 
\[
s\log\frac{n}{s} =
\log n \geq  \frac{1}{8}\, k\log\frac{\e n}{k}.
\]
On the other hand, if $r\geq 4$, then $s \geq \tfrac{r}{8} \geq \tfrac{k}{32}$, 
as $k \geq 3$ by assumption.
Additionally, $s \leq \tfrac{k}{2}$, so $\tfrac{n}{s} \geq \tfrac{2n}{k}$. 
Therefore
\[
s\log\frac{n}{s} \geq \frac{k}{32} \log\frac{2n}{k} \geq \tfrac{\log 2}{32} \cdot k\log\frac{\e n}{k}.
\]
Combining the previous three displays, we obtain 
\begin{equation}
\label{eq:vander-q-klog}
\Quant{1-\delta}\Big(\Gamma_W(\eps;\lambda_0\1_n)\Big)
\geq \frac{\log 2}{32}\,k\log\frac{\e n}{k}.
\end{equation}
Now set $\lambda_\gamma \defn \sigma^{-1}(\gamma^{1/n})$, 
for $\gamma > \delta$. Note $\P_{p_{\lambda_\gamma\1_n}}(\eps=\1_n)= \gamma$. 
Since $\1_n\in W$ (take the constant polynomial), we have
$\sup_{a>0} p_{a\1_n}(\1_n)=1$.
Hence, on the event $\{\eps=\1_n\}$,
\[
\Gamma_W(\eps;\lambda_\gamma\1_n)
\ge -\log p_{\lambda_\gamma\1_n}(\1_n)
= \log\frac{1}{\gamma}.
\]
Since $\P(\eps=\1_n)=\gamma>\delta$, the event $\{\Gamma_W(\eps;\lambda_\gamma\1_n) \geq \log\frac{1}{\gamma}\}$ has probability at least $\gamma > \delta$, so $\Quant{1-\delta}\bigl(\Gamma_W(\eps;\lambda_\gamma\1_n)\bigr) \geq \log\frac{1}{\gamma}$. As $\lambda_\gamma$ is a valid choice of $\lambda$ and $\gamma>\delta$ may be chosen arbitrarily close to $\delta$,
\begin{equation}
\label{eq:vander-q-logdelta}
\sup_{\lambda \in \R}\Quant{1-\delta}\bigl(\Gamma_W(\eps;\lambda\1_n)\bigr)
\geq \log\frac{1}{\delta}.
\end{equation}
Combining \eqref{eq:vander-q-klog} and \eqref{eq:vander-q-logdelta},
\[
\sup_{\lambda\in\R}
\inf_{0\leq t_1<\cdots<t_n\leq1}
\Quant{1-\delta}\Big(\Gamma_{V_{n,k}(t_1,\dots,t_n)}(\eps;\lambda\1_n)\Big)
\geq
\max\Big\{
\frac{\log 2}{32}\,k\log\frac{\e n}{k}, \log\frac{1}{\delta}
\Big\}
\geq
c
\Big(k\log\frac{\e n}{k}+\log\frac{1}{\delta}\Big),
\]
as claimed. We can take $c = \tfrac{\log 2}{\log 2 + 32}$. This completes the proof of Claim~\ref{item:lower-bound-quantile-vandermonde-fixed}.

\paragraph{}
We now prove Claim~\ref{item:lower-bound-quantile-vandermonde-random}. Let $P$ denote a distribution on $[0,1]$ such that if $T_1, \dots, T_n \simiid P$, they are almost surely distinct. Write for short $T=(T_1,\ldots,T_n) \sim P^{\otimes n}$ and $W_T=V_{n,k}(T_1,\ldots,T_n)$. For each realization with distinct coordinates, let $\pi$ order the coordinates and let $\Pi_\pi$ be the associated permutation matrix. Then
\[
V_{n,k}(T_{\pi(1)},\ldots,T_{\pi(n)})=\Pi_\pi W_T,
\qquad
\Gamma_{\Pi_\pi W_T}(\Pi_\pi\eps;\lambda\1_n)
=\Gamma_{W_T}(\eps;\lambda\1_n).
\]
Since $\Pi_\pi\eps\sim p_{\lambda\1_n}$ whenever $\eps\sim p_{\lambda\1_n}$, the bounds for fixed nodes apply conditionally on every such realization. In particular, \eqref{eq:bin-event-small-r} and \eqref{eq:bin-event-large-r} give
\begin{equation}
	\P\big(\Gamma_{W_T}(\eps; \lambda_0 \1_n) \geq s \log(n/s) \big| T \big) > \frac 1 \e \, .
\end{equation}
Integrating with respect to $P^{\otimes n}$ we deduce that
\begin{equation}
\P\big(\Gamma_{W_T}(\eps; \lambda_0 \1_n) \geq s \log(n/s) \big)
	=\E \big[ \P\big(\Gamma_{W_T}(\eps; \lambda_0 \1_n) \geq s \log(n/s) \big| T \big) \big]
	> \frac 1 \e \, .
\end{equation}
The same conditioning argument yields the random-design analogue of~\eqref{eq:vander-q-logdelta}, and in turn, the claim.

\paragraph{}
Finally, we prove Claim~\ref{item:lower-bound-expected-vandermonde} and start with the first inequality. Since $v_{n,k}$ has coordinates all equal (to $\sigma^{-1}(1-r/n)$), the resulting distribution $p_{v_{n,k}}$ has exchangeable marginals. In particular, the function
\begin{equation}
\phi: t \in [0,1]^n \mapsto \E\big[\Gamma_{V_{n,k}(t_1,\dots,t_n)}(\eps; v_{n,k}) \big| T=t \big]
\end{equation}
is invariant under permutation of the entries of $t$. Hence, with $T \sim P^{\otimes n}$, we have
\begin{equation}
	\E \phi(T) \geq \inf_{0 \leq t_1 < \dots < t_n \leq 1} \phi(t) \, ,
\end{equation}
which is the first inequality.

For the second one, in light of the previous inequality, we fix an increasing sequence $t_1 < \cdots < t_n$ in $[0,1]$ and let $W$ denote the Vandermonde subspace $V_{n,k}(t_1, \dots, t_n)$. By the same argument that led to~\eqref{eq:achievable-sign-pattern} in the proof of Claim~\ref{item:lower-bound-quantile-vandermonde-fixed}, any sign pattern $\eps \in \{-1,1\}^n$ with at most $r$ negative entries is achievable by $W$ (in the sense of~\eqref{eq:achievable-sign-pattern}), hence, similar computations show that 
\begin{equation}
\label{eq:binom-lower-bound}
	\Gamma_W(\eps; v_{n,k}) \geq N(\eps) \log\Big(\frac n r\Big) \cdot \1(N(\eps) \leq r) \, .
\end{equation}
It remains to show that
\begin{equation}
\label{eq:trunc-mom-1}
\E[N(\eps) \cdot \1(N(\eps) \leq r) ] \geq c r	\, ,
\end{equation}
for some absolute constant $c$, which amounts to showing that $N(\eps)$ has enough mass in the lower tail. We do so in the following lines. We split into two cases: for small values of $r$, it is enough to bound from below by an absolute constant, while for large values of $r$, a more refined bound is necessary. Fix some threshold $r_0$, to be chosen later.

A naive bound is the following:
\begin{align}
\E[N(\eps) \cdot \1(1 \leq N(\eps) \leq r) ]
	&\geq \P\big(1 \leq N(\eps) \leq r\big)
	= \P(N(\eps) \leq r) - \P(N(\eps) = 0)  \\
	&\geq \frac{1}{2} - \Big(1 - \frac{r}{n}\Big)^n
	\geq \frac{1}{2} - \e^{-r} \geq \frac{1}{2} - \frac{1}{\e} > \frac{1}{8}. \label{eq:small-r}
\end{align}
This is sufficient for small values of $r$. Now let $r\geq r_0$ and let $\alpha \in (0,1)$, to be chosen later. This time, we condition on $N(\eps)$ being of order $r$, that is, we bound
\begin{equation}
\E[N(\eps) \cdot \1(\alpha r \leq N(\eps) \leq r)]
	\geq \alpha r \big( \P( N(\eps) \leq r ) - \P( N(\eps) < \alpha r) \big) \, .
\end{equation}
As before, we use the fact that $r$ is a median of $N(\eps)$, so that $\P( N(\eps) \leq r ) \geq 1/2$, and for the other term, Okamoto's lower-tail bound~\cite[Theorem~2]{okamoto1958binom}, applied with parameter $p=r/n$ and deviation $u=(1-\alpha)r/n$, gives
\begin{equation}
	\P(N(\eps) < \alpha r ) \leq \exp\big(- (1-\alpha)^2 r/2\big) \, .
\end{equation}
Now, set $r_0=8$ and $\alpha = 1/2$, so that the upper bound above is $\e^{-r/8} \leq 1/\e$ for $r\geq 8 = r_0$. With these choices, we obtain that for $r\geq 8$,
\begin{equation}
\label{eq:large-r}
\E[N(\eps) \cdot \1(r/2 \leq N(\eps) \leq r)]
	\geq \frac r 2 \cdot \frac 1 8 = \frac r {16} \, .
\end{equation}
To conclude, we relate $r$ and $k$ in both cases. First, it always holds that $2r \leq (k-1) <k$, hence $n/r > 2n/k$, and
\begin{equation}
\log\Big(\frac n r \Big) \geq \log(2) \cdot \log\Big(\frac{\e n}{k} \Big)\, .
\end{equation}
In addition, if $r \leq 7$, then $k \leq 2r + 2 \leq 16$, hence, combining~\eqref{eq:binom-lower-bound} and~\eqref{eq:small-r}, we conclude that 
\begin{equation}
\E_\eps \Gamma_W(\eps; v_{n,k}) \geq \frac{\log 2}{128} \,  k \log\Big(\frac{\e n}{k}\Big) \, .
\end{equation}
On the other hand, if $r\geq 8$, using again that $k \leq 2r + 2$, we deduce that $r \geq 4k/9$ and, combining~\eqref{eq:binom-lower-bound} and~\eqref{eq:large-r},
\begin{equation}
\E_\eps\, \Gamma_W(\eps; v_{n,k})
	\geq \frac{r}{16} \log\frac{n}{r}
	\geq \frac{k}{36}\cdot\log 2\cdot\log\frac{\e n}{k} = \frac{\log 2}{36}\, k\log\Big(\frac{\e n}{k}\Big) \, .
\end{equation}
This concludes the proof of Claim~\ref{item:lower-bound-expected-vandermonde} and~\Cref{prop:strong-lower-bound-vandermonde}.

\paragraph{}
In the next section we provide the proof of the main result, Theorem~\ref{thm:main-result}. The proof combines the upper and lower bounds, respectively from Proposition~\ref{prop:upper-bound-on-LLR-fixed-design} and Proposition~\ref{prop:strong-lower-bound-vandermonde}.

\subsubsection{Proof of~\Cref{thm:main-result}}
\label{sec:proof-thm-main-result}

By \Cref{prop:upper-bound-on-LLR-fixed-design} and in view 
of~\cref{eqn:definition-of-largest-subspace-LLR}, we have 
\begin{equation}
\label{ineq:upper-bound-on-E-star-and-Q-star}
E^\star(n, k) \leq k \log \frac{\e n}{k}, 
\quad \mbox{and} \quad 
Q^\star_{1-\delta}(n,k) \leq 
k \log \frac{\e n}{k} + \log\frac{1}{\delta},
\end{equation}
for any $1 \leq k \leq \twomin{n}{d}$ and any $\delta \in (0, 1)$. 
For the random-design lower bound, let $T_1,\ldots,T_n\simiid\mathsf{Unif}(0,1)$ and set
$X_i=(1,T_i,\ldots,T_i^{d-1})^\T$. Then the resulting design matrix satisfies
$\ran(X)=V_{n,d}(T_1,\ldots,T_n)$, while $X\thetastar=\lambda\1_n$ for $\thetastar=\lambda e_1$.
Thus, by 
\Cref{prop:strong-lower-bound-vandermonde}\ref{item:lower-bound-quantile-vandermonde-random}, 
taking $k=d$, we have 
\begin{equation}
\label{ineq:final-lower-bound-on-worst-case-random-design-quantile}
\WorstCaseRandomDesignQuant{1-\delta}{n}{d} \geq 
c'' \, \bigg(\Big(\twomin{d \log \frac{\e n}{d}}{n}\Big) + \log \frac{1}{\delta}\bigg).
\end{equation}
for all $\delta \in (0, \tfrac{1}{\e}]$ and $n\geq d\geq3$. 

We now deduce the result. The upper comparisons between random and fixed design below follow by conditioning on $X_{1:n}$, applying the uniform fixed-design bounds, and integrating over $X_{1:n}$. For $\delta\in(0,1/\e]$, we note that 
\begin{multline}
c'' \, \bigg(\Big(\twomin{d \log \frac{\e n}{d}}{n}\Big) + \log \frac{1}{\delta}\bigg)
\stackrel{{\rm (i)}}{\leq}
\WorstCaseRandomDesignQuant{1-\delta}{n}{d} 
\leq
\WorstCaseFixedDesignQuant{1-\delta}{n}{d} 
\\
\stackrel{{\rm (ii)}}{\leq}
\max_{1 \leq k \leq \twomin{n}{d}} Q^\star_{1-\delta}(n,k)
\stackrel{{\rm (iii)}}{\leq}
\Big(\twomin{d \log \frac{\e n}{d}}{n}\Big) + \log \frac{1}{\delta}.
\end{multline}
Above, inequality (i) follows 
from~\eqref{ineq:final-lower-bound-on-worst-case-random-design-quantile}, 
inequality (ii) follows 
from~\Cref{lem:reformulation-via-subspaces}\ref{item:worst-case-reduction}, and
inequality (iii) follows from~\cref{ineq:upper-bound-on-E-star-and-Q-star}. 
The display above establishes~\cref{eqn:main-worst-case-result-quantile}.
For any nonnegative random variable $Z$, the definition of the quantile gives
$\P(Z\geq\Quant{1-\delta}(Z))\geq\delta$, and hence
$\E Z\geq\delta\Quant{1-\delta}(Z)$. Thus,
\begin{multline}
	\frac{c''}{\e} \, \Big(\twomin{d \log \frac{\e n}{d}}{n}\Big)
	\stackrel{{\rm (i)}}{\leq}
	\frac{1}{\e}\WorstCaseRandomDesignQuant{1-\tfrac{1}{\e}}{n}{d}
	\\\stackrel{{\rm (ii)}}{\leq}
	\WorstCaseRandomDesign{n}{d}
	\leq
	\WorstCaseFixedDesign{n}{d}
	\stackrel{{\rm (iii)}}{\leq} 
	\max_{1 \leq k \leq \twomin{n}{d}} E^\star(n,k)
	\stackrel{{\rm (iv)}}{\leq}
	\twomin{d \log \frac{\e n}{d}}{n}.
\end{multline}
Above, inequality (i) follows from~\eqref{ineq:final-lower-bound-on-worst-case-random-design-quantile}, while inequality (ii) follows from the preceding bound with $\delta=1/\e$, inequality 
(iii) follows from~\Cref{lem:reformulation-via-subspaces}\ref{item:worst-case-reduction}, 
and inequality (iv) follows from~\cref{ineq:upper-bound-on-E-star-and-Q-star}. 
The display above establishes~\cref{eqn:main-worst-case-result-expected}.

\subsection{Remaining proofs and auxiliary results}
\label{sec:remaining-and-aux}
	
	\subsubsection{Proof of~Lemma~\ref{lem:chisquarequantiles}}

	Let $Z_d \sim \chi^2_d$. 
	It is well known (\eg see~\cite[eqn.~(4.3)]{LauMas2000})
	that $\P(Z_d \leq z_\delta) \geq 1-\delta$ holds with $z_\delta \defn d + 2\sqrt{d \log(1/\delta)} + 2 \log(1/\delta)$. In particular,
	\[
	\Quant{1-\delta}(Z_d) \leq z_\delta
	\leq \inf_{\lambda>0}\left\{(1+\lambda)\,d + (2+\lambda^{-1})\log\tfrac{1}{\delta}\right\}
	\leq \frac{3+\sqrt{5}}{2}\Big(d + \log\tfrac{1}{\delta}\Big) \, ,
	\]
where the last inequality follows by choosing $\lambda$ so that $1+\lambda = 2 +1/\lambda$. This establishes the upper bound.

	For the lower bound, we use the moment generating function of $Z_d$, $\E \e^{-t Z_d} = (1+2t)^{-d/2}$ to obtain:
	\[
	\P(Z_d \leq \tfrac{d}{12})
	\leq
	\inf_{t>0} \Big\{\e^{\tfrac{td}{12}}\E[\e^{-tZ_d}]\Big\} 
	=
	\inf_{t>0} \Big\{\e^{td/12}(1+2t)^{-d/2}\Big\}
	\leq \frac{1}{2^d}.
	\]
	The final inequality arose by taking $t=\tfrac{11}{2}$, and $\log 12 - \tfrac{11}{12} > 2 \log 2$. 
	Hence, $\P(Z_d\geq \tfrac{d}{12})\geq \tfrac{1}{2}$ for all $d \geq 1$. 
	Therefore, because $1-\delta\geq \tfrac{1}{2}$, we have
	\begin{subequations}
	\begin{equation}
		\label{eqn:bound-for-quantile-chi2}
		\Quant{1-\delta}(Z_d)\geq \Quant{1/2}(Z_d)\geq \frac{d}{12}.
	\end{equation}
	For $d \geq 2$, we have $Z_d = Z_2 + Z_{d-2}$. Therefore, using the fact that 
	$\P(\chi^2_2 \leq t) = 1 - \e^{-t/2}$, we have 
	\begin{equation}
		\label{eqn:bound-for-quantile-chi2-2}
	\Quant{1-\delta}(Z_d) \geq \Quant{1-\delta}(Z_2) = 2 \log \Big(\frac{1}{\delta}\Big). 
	\end{equation}
	\end{subequations}
	Combining~\cref{eqn:bound-for-quantile-chi2} and \cref{eqn:bound-for-quantile-chi2-2}, we have
	$\Quant{1-\delta}(Z_d) \geq \tfrac{2}{25}(d + \log\tfrac{1}{\delta})$ for $d \geq 2$.
	For $d = 1$, we have $Z_1 = X^2$ where $X \sim \Normal{0}{1}$ which gives
	$\Quant{1-\delta}(Z_1) = \Phi^{-1}(1- \delta/2)^2$ where
	$\Phi(u) = \P(X \leq u)$.
For $\delta \in (0,1/4]$: \Cref{lem:mills-ratio-inequality} gives $\Quant{1-\delta}(Z_1) =\Phi^{-1}(1- \delta/2)^2 \geq \tfrac{2}{3}\log\tfrac{1}{\delta}$. Averaging with~\cref{eqn:bound-for-quantile-chi2} yields $\Quant{1-\delta}(Z_1) \geq \tfrac{2}{27}(1+\log\tfrac{1}{\delta})$.
For $\delta \in (1/4,1/2]$: the quantile satisfies $\Phi^{-1}(1-\delta/2)^2 \geq \Phi^{-1}(3/4)^2 \approx 0.455$, since $\Quant{1-\delta}(Z_1)$ is decreasing in $\delta$. As $\tfrac{2}{27}(1+\log\tfrac{1}{\delta}) \leq \tfrac{2}{27}(1+\log 4) \approx 0.177$, the bound holds.
Combining the two cases for $d=1$ with the bound for $d \geq 2$ gives the claimed inequality with $C = 27/2$. This concludes the proof of Lemma~\ref{lem:chisquarequantiles}.

\begin{lemma}
\label{lem:mills-ratio-inequality}
It holds that $\Phi^{-1}(1 - \tfrac{\delta}{2}) \geq \sqrt{\tfrac{2}{3} \log \tfrac{1}{\delta}}$ for all $\delta \in (0, 1/4]$.
\end{lemma}
\begin{proof}

We will show that $2 \e^{3u^2/2}(1 - \Phi(u)) \geq 1$ holds for all $u \geq u_* := \sqrt{\tfrac{4\log 2}{3}}$. Since $\delta \leq 1/4$ implies $u_\delta = \sqrt{\tfrac{2}{3}\log\tfrac{1}{\delta}} \geq u_*$, this gives $2\e^{3u_\delta^2/2}(1-\Phi(u_\delta)) \geq 1$, and by monotonicity of $\Phi$ we get $\Phi^{-1}(1 - \tfrac{\delta}{2}) \geq u_\delta$ as required. By the Mills-ratio lower bound $1-\Phi(u) \geq \frac{u}{1+u^2}\varphi(u)$,
\[
2\e^{3u^2/2}(1-\Phi(u)) \geq \frac{2}{\sqrt{2\pi}} \cdot \frac{u\,\e^{u^2}}{1+u^2} =: g(u).
\]
An elementary analysis shows that the function $g$ is increasing. Hence $g(u) \geq g(u_0)$ where $u_0 = \sqrt{\tfrac{4\log 2}{3}}$ corresponds to $\delta = 1/4$. A direct computation gives
\[
g(u_0) = \frac{2\cdot 2^{4/3}}{\sqrt{2\pi}}\cdot \frac{u_0}{1+\tfrac{4\log 2}{3}} \approx 1.004 > 1,
\]
completing the proof.
\end{proof}

\section{Situations when the logarithmic factor can be removed}
\label{sec:logfree}

In this section, we develop several results, each showing a situation in which the additional logarithmic overhead of the Vandermonde design is not needed.

\subsection{Univariate models and cumulant generating functions}
\label{sec:univariatemodels}

Although our primary motivation in this work is logistic regression, the results of this section (specifically Section~\ref{sec:log-likelihood-stat}) are developed for general canonical exponential families and generalized linear models. The variational characterization of the log-likelihood ratio statistic via the convex conjugate of the cumulant generating function (\Cref{prop:variational-characterization-of-log-likelihood-stat}) holds in this generality at no extra cost and, along with the resulting nonasymptotic Wilks-type bounds, may be of independent interest for other generalized linear models.

In this section, we consider the cumulant generating function of a random vector $Z \in \R^m$:
\[
\psi_Z(\lambda) = \log \E \e^{\lambda^\T Z}, \quad \mbox{for}~ \lambda \in \R^m.
\]
We also make repeated use of the \emph{convex conjugate} of the cumulant generating function, which is defined as~\cite[pp.~104]{rockafellar1970convex}:
\[
\psi_Z^\ast(z) = 
\sup_{\lambda \in \R^m} \Big\{\, 
\lambda^\T z - \psi_Z(\lambda)
\,\Big\}.
\]
In~\Cref{sec:scalar-random-variables} we study the scalar case, \ie when $m = 1$. For exponential families and generalized linear models, an explicit variational characterization which relates the convex conjugate of the cumulant generating function of the sufficient statistic with the log-likelihood statistic is developed in~\Cref{sec:log-likelihood-stat}. In the univariate case, we are able to prove nonasymptotic variants of Wilks' theorem by combining these results. We defer most proofs to~\Cref{sec:proofs-cgf-univariate}.

\subsubsection{Properties of the cumulant generating function of scalar random variables}
\label{sec:scalar-random-variables}

The main result in this section regards the convex conjugate of the cumulant generating functions of scalar random variables.
The next result, proved in~\Cref{sec:proof-univariate-CGF}, establishes upper bounds on the stochastic behavior of the convex conjugate of cumulant generating function of a scalar random variable applied to itself. 
\begin{proposition}
\label{prop:conjugate-scalar-case}
Let $Z$ be a scalar random variable with cumulant generating function $\psi_Z$
(not necessarily finite on all of $\R$).
Then,
\begin{enumerate}[label=(\roman*)]
\item 
\label{item:deviation-for-ccgf}
for every $t > 0$, it holds that
	$\P(\psi_Z^*(Z) > t) \leq \twomin{2 \e^{-t}}{1}$; and
\item 
\label{item:mean-of-ccgf}
it holds that $\E \psi_Z^\ast(Z) \leq 1 + \log 2$. 
\end{enumerate}
\end{proposition}

Note that results similar, but weaker, than~\Cref{prop:conjugate-scalar-case} exist in the large deviations literature. For instance, results in the monograph by Dembo-Zeitouni~\cite[Eqn.~(5.1.16)]{dembo2010large} imply 
\[
\P\Big(\psi_Z^\ast(Z) > t\Big) \leq \twomin{1}{\frac{2}{1-\gamma}\e^{-\gamma t}}, \quad \mbox{for all}~t > 0,
\]
for any $\gamma \in [0, 1)$. Notably, the result in~\Cref{prop:conjugate-scalar-case} is unimprovable. 

\begin{remark}[\Cref{prop:conjugate-scalar-case} is sharp] 
\label{rem:sharpness-of-1d-cgf-trick}
For each $\alpha \in (0, 1)$, let $Z_\alpha = \eps B_\alpha$, where $\eps$ and $B_\alpha$ are independent, with $\eps \sim \mathsf{Unif}(\{-1, 1\})$ and $B_\alpha \sim \mathsf{Beta}(1, \alpha)$. As $|Z_\alpha| \leq 1$ almost surely, the cumulant generating function $\psi_\alpha \equiv \psi_{Z_\alpha}$ is finite on $\R$.
By taking $\alpha$ sufficiently small, this example makes both bounds in~\Cref{prop:conjugate-scalar-case} arbitrarily sharp: for every fixed $t>0$, the probability $\P(\psi_\alpha^\ast(Z_\alpha)>t)$ can be made arbitrarily close to $\twomin{2\e^{-t}}{1}$, while $\E\psi_\alpha^\ast(Z_\alpha)$ can be made arbitrarily close to $1+\log 2$.
See~\Cref{sec:proof-of-remark-1d-cgf} for the details of this calculation.
In view of this example, we see that~\Cref{prop:conjugate-scalar-case} cannot be improved.
\end{remark}

\subsubsection{LLR statistics for  exponential families and GLMs}
\label{sec:log-likelihood-stat}

The next result describes how the log-likelihood statistic can be written in terms of the convex conjugate for certain exponential families. We recall two definitions in this direction. 

\begin{definition}[Canonical exponential family]
A family of probability measures $\cP = \{P_\theta\}_{\theta \in \Theta}$ is called a \emph{canonical exponential family} if it is dominated by a $\sigma$-finite measure $\mu$ on $\R^n$, with the canonical versions of its densities fixed by
\[
p_\theta(z) \equiv \frac{\ud P_\theta}{\ud \mu}(z) = \exp\big\{\theta^\T T(z) - A(\theta) \big\} h(z) \quad \mbox{for all}~z \in \R^n,~\theta\in\Theta.
\]
Above, the measurable map $T \colon \R^n \to \R^k$ is referred to as the \emph{sufficient statistic}, $h \colon \R^n \to \R_+$ is measurable, and
\[
A(\theta)\defn\log\int_{\R^n}\exp\{\theta^\T T(z)\}h(z)\,\ud\mu(z),
\qquad
\Theta\defn\dom(A)=\{\theta\in\R^k:A(\theta)<\infty\}.
\]
\end{definition}

One important case of canonical exponential families is the case of generalized linear models.

\begin{definition}[Canonical generalized linear model (GLM)]
A family of probability measures $\cP = \{P_\beta\}_{\beta \in \cB}$ is called a \emph{canonical generalized linear model (GLM)} if there exist a $\sigma$-finite measure $\mu$ on $\R^n$, covariates $x_1, \dots, x_n \in \R^d$, and measurable functions $h\colon\R\to\R_+$ and $A\colon\R\to\R\cup\{+\infty\}$ satisfying the log-partition normalization stated below, with the density versions fixed by
\[
p_\beta(y) = \frac{\ud P_\beta}{\ud \mu}(y) = \prod_{i=1}^n h(y_i) \prod_{i=1}^n \exp(y_i x_i^\T \beta - A(x_i^\T \beta)), \quad \mbox{for all}~y \in \R^n,~\beta\in\cB.
\]
Here, the natural parameter space is
$\cB=\{\beta\in\R^d:A(x_i^\T\beta)<\infty\text{ for all }i\in[n]\}$.
For logistic regression with the convention $Y_i\in\{-1,1\}$ used in this
paper, set $B_i=(1+Y_i)/2\in\{0,1\}$. In the display above, take $y=B$. Then $\mu$ is counting measure on
$\{0,1\}^n$, $h\equiv1$, and $A(t)=\log(1+\e^t)$.
\end{definition}

Note that the correspondence between canonical GLMs and exponential families arises through the following representation of the density $p_\beta$. Set
\begin{subequations}
\begin{equation}
\label{eqn:exp-family-parameters-GLM}
\tilde A(\beta)\defn \sum_{i=1}^n A(x_i^\T \beta), \quad 
\tilde h(y) = \prod_{i=1}^n h(y_i), \quad \mbox{and} \quad \tilde T(y) = X^\T y. 
\end{equation}
Then, we have 
\begin{equation}
\label{eqn:exp-family-representation}
p_\beta(y)
= \tilde h(y) \exp\Big\{\beta^\T \tilde T(y) - \tilde A(\beta)\Big\},
\end{equation}
\end{subequations}
As required in the definition, the following identities hold on $\R^d$:
\[
\tilde A(\beta)=\log\int_{\R^n}
\exp\{\beta^\T\tilde T(y)\}\tilde h(y)\,\ud\mu(y),
\qquad
\cB=\dom(\tilde A).
\]
This puts the canonical GLM $\{P_\beta\}_{\beta \in \cB}$ in the form of 
a canonical exponential family. We also note that in a canonical GLM, the function $\tilde A$ and the natural parameter space $\cB$ are convex, as inherited from the more general case of canonical exponential families~\cite[Theorem~1.6.3]{bickel2007mathematical}.

For exponential families, the log-likelihood statistic has a variational characterization in terms of the convex conjugate of the cumulant generating function of the sufficient statistic.
The proof is presented in~\Cref{sec:proof-prop-var-car-log-lik-stat}.

\begin{proposition}
\label{prop:variational-characterization-of-log-likelihood-stat}
The following characterizations for the log-likelihood statistic hold: 
\begin{enumerate}[label=(\roman*)]
\item 
\label{item:var-char-log-stat-exp-family}
If $\cP = \{P_\theta\}_{\theta \in \Theta}$ denotes a canonical exponential family and $Z \sim P_{\thetastar}$ for some $\thetastar \in \Theta$, then the log-likelihood statistic satisfies
\[
\log \frac{\sup_{\theta \in \Theta} p_\theta(z)}{p_{\thetastar}(z)} = \psi_{T(Z)}^\ast(T(z)), \quad \mbox{for $P_{\thetastar}$-almost every}~z.
\] 
\item 
\label{item:var-char-log-stat-GLM}
If $\cP = \{P_\beta\}_{\beta \in \cB}$ denotes a canonical GLM and $Y \sim P_{\betastar}$ for some $\betastar\in\cB$, then
\[
\log \frac{\sup_{\beta \in \cB} p_{\beta}(y)}{p_{\betastar}(y)} = \psi^\ast_{X^\T Y}(X^\T y), \quad \mbox{for $P_{\betastar}$-almost every}~y,
\]
where $X \in \R^{n \times d}$ denotes the matrix with rows $x_i^\T$ corresponding to the covariates in the GLM.
\end{enumerate}
\end{proposition}

Combining~\Cref{prop:conjugate-scalar-case} and \Cref{prop:variational-characterization-of-log-likelihood-stat} immediately yields the following results for log-likelihood statistics in univariate models. 

\begin{corollary}
\label{cor:dimone}
The following hold: 
\begin{enumerate}[label=(\roman*)]
\item If $\cP = \{P_\theta\}_{\theta \in \Theta}$ denotes a univariate canonical exponential family with $Z \sim P_{\thetastar}$ for some $\thetastar \in \Theta \subset \R$, then the log-likelihood statistic satisfies
\[
\E \log \frac{\sup_{\theta \in \Theta} p_\theta(Z)}{p_{\thetastar}(Z)} \leq 1 + \log 2 \quad \mbox{and} \quad \P\Big(\log \frac{\sup_{\theta \in \Theta} p_\theta(Z)}{p_{\thetastar}(Z)} \leq \log \frac{2}{\delta}\Big) \geq 1 - \delta 
\] 
for all $\delta \in (0, 1)$; and 
\item If $\cP = \{P_\beta\}_{\beta \in \cB}$ denotes a univariate canonical GLM with $\cB\subset\R$ and $Y \sim P_{\betastar}$ for some $\betastar \in \cB$, then the log-likelihood statistic satisfies
\[
\E \log \frac{\sup_{\beta \in \cB} p_{\beta}(Y)}{p_{\betastar}(Y)} \leq 1 + \log 2 \quad \mbox{and} \quad \P\Big(\log \frac{\sup_{\beta \in \cB} p_{\beta}(Y)}{p_{\betastar}(Y)} \leq \log \frac{2}{\delta}\Big) \geq 1 - \delta 
\] 
for all $\delta \in (0, 1)$.
\end{enumerate}
\end{corollary}

In particular, the logit model is the canonical GLM for the Bernoulli exponential family. The following result provides a variational characterization of the quantity $\Gamma_W$ defined in~\eqref{eq:subspace-reformulation} appearing in the subspace reformulation of the LLR statistic in~Lemma~\ref{lem:reformulation-via-subspaces}.

\begin{lemma}[Variational characterization of $\Gamma_W$]
\label{lem:var-char}
Let $W \subset \R^n$ be a subspace and let $\Pi_W$ denote the orthogonal projection onto $W$. Then for all $v \in W$ and $\eps \in \{-1,1\}^n$, one has
\[
  \Gamma_W(\eps;v)
  =
  \inf_{\substack{q\in[0,1]^n\\ \Pi_W q = \Pi_W B}}
  \sum_{i=1}^n \kl{\Ber{q_i}}{\Ber{p_i}}
  \;=\;
  \inf_{\substack{\delta\in[-1,1]^n\\ \Pi_W\delta = \Pi_W\eps}}
  \sum_{i=1}^n
  \kl{\Ber{\tfrac{1+\delta_i}{2}}}{\Ber{\sigma(v_i)}},
\]
where $B_i \defn \tfrac{1+\eps_i}{2}$ and $p_i \defn \sigma(v_i)$.
\end{lemma}

\subsection{Proofs}
\label{sec:proofs-cgf-univariate}
\subsubsection{Proof of~\Cref{prop:conjugate-scalar-case}}
\label{sec:proof-univariate-CGF}
To lighten notation, we abbreviate $\psi \equiv \psi_Z$ and $\psi^\ast \equiv \psi_Z^\ast$ throughout.

\paragraph{Claim~\ref{item:deviation-for-ccgf}:}
Fix $t>0$. We define
\begin{equation}
	z_+(t) = \inf_{\substack{\lambda>0\\ \psi(\lambda)<\infty}} \frac{t+\psi(\lambda)}{\lambda}\,  \quad \mbox{and} \quad 
	z_-(t) = \inf_{\substack{\lambda>0\\ \psi(-\lambda)<\infty}} \frac{t+\psi(-\lambda)}{\lambda},
\end{equation}
with the convention $\inf\varnothing=+\infty$. These thresholds cannot equal $-\infty$ (apply the exponential moment bound at any level $a$ with $\P(Z\geq a)>\e^{-t}$, and likewise to $-Z$).
If there exists $z\in \R$ such that $\psi^*(z) >t$, then there exists $\lambda \in \R$ such that $\lambda z - \psi(\lambda) >t$. Therefore, we have the following inclusion of events,

\begin{equation}
\label{eqn:inclusion-of-events}
\{\psi^\ast(Z) > t\} \subset \{Z > z_+(t)\} \cup \{Z < -z_-(t)\}.
\end{equation}

Chernoff's inequality gives
\begin{equation}\label{ineq:main-tail-bound}
	\twomax{\P\big(Z>z_+(t)\big)}{\P\big(Z< -z_-(t)\big)} \leq \e^{-t} \, .
\end{equation}
(If an infimum is not attained, use an arbitrarily close candidate, as usual; if the corresponding domain is empty, the event is empty.)
The claim follows by combining the inclusion~\eqref{eqn:inclusion-of-events}, the inequalities~\eqref{ineq:main-tail-bound}, and a union bound.

\paragraph{Claim~\ref{item:mean-of-ccgf}:}

By Claim~\ref{item:deviation-for-ccgf},
\[
\E\psi_Z^\ast(Z) = \int_0^\infty \P(\psi_Z^\ast(Z) > t)\,dt
\leq \int_0^{\log 2} 1\,dt + \int_{\log 2}^\infty 2\e^{-t}\,dt
= \log 2 + 1.
\]

\subsubsection{Proof of~\Cref{prop:variational-characterization-of-log-likelihood-stat}}
\label{sec:proof-prop-var-car-log-lik-stat}

\paragraph{Claim~\ref{item:var-char-log-stat-exp-family}:} 
Let us denote by $\psi$ the cumulant generating function of $T$ under $P_{\thetastar}$. It is a well-known identity for canonical exponential families that for $\thetastar \in \Theta$, we have~\cite[pp.~31] {keener2010theoretical}
   \begin{equation}
   \label{eq:mgf-partition}
   \psi(\theta - \thetastar) = \begin{cases} 
    A(\theta) - A(\thetastar), & \theta \in \Theta \\ 
    +\infty, & \text{else} \, .
    \end{cases}
   \end{equation}
    Hence, by definition of the convex conjugate it holds that
    \begin{equation}
    \psi^\ast(T(z)) = 
	    \sup_{\theta \in \R^k} 
    \Big\{\, 
    (\theta - \thetastar)^\T T(z) - \psi(\theta - \thetastar) \,\Big\} 
    =
    \sup_{\theta \in \Theta} 
    \Big\{(\theta - \thetastar)^\T T(z)- \big(A(\theta)- A(\thetastar)\big)\Big\}.
      \label{eqn:convex-conjugate-formulation}
    \end{equation}
    On the set where $h(z)>0$, which has full $P_{\thetastar}$-measure, the fixed canonical versions give, simultaneously for every $\theta\in\Theta$,
    $\log p_\theta(z) = \theta^\T T(z) -  A(\theta) + \log h(z)$. Therefore,
    \begin{equation}
    \label{eqn:re-write-of-log-likelihood-stat}
    \log \frac{\sup_{\theta \in \Theta} p_\theta(z)}{p_{\thetastar}(z)}= \sup_{\theta \in \Theta} \Big\{\, \log p_\theta(z) - \log p_{\thetastar}(z) \,\Big\} = 
    \sup_{\theta \in \Theta}
    \Big\{\, 
    (\theta - \thetastar)^\T T(z) - \big(A(\theta) - A(\thetastar)\big)
    \,\Big\}.
    \end{equation}
    Combining~\cref{eqn:convex-conjugate-formulation,eqn:re-write-of-log-likelihood-stat} yields the claim.

\paragraph{Claim~\ref{item:var-char-log-stat-GLM}:}

Applying Claim~\ref{item:var-char-log-stat-exp-family} to the representation
in~\cref{eqn:exp-family-parameters-GLM,eqn:exp-family-representation} gives, for
$P_{\betastar}$-almost every $y$,
\[
\log\frac{\sup_{\beta\in\cB}p_{\beta}(y)}{p_{\betastar}(y)}
=\psi_{\tilde T(Y)}^\ast\big(\tilde T(y)\big)
=\psi_{X^\T Y}^\ast(X^\T y).
\]

\subsubsection{Proof of~Lemma~\ref{lem:var-char}}

We work with the alternative parameterization where the outcomes are $0$ or $1$ to put the problem in the canonical form of the Bernoulli exponential family and GLM. Let $v \in W$ and $\eps \in \{-1,1\}^n$. Recall that
\[
\Gamma_W(\eps; v)
	= \sup_{w\in W} \log\Big( \frac{p_w(\eps)}{p_v(\eps)}\Big) \, .
\]
For every $i \in [n]$, let $B_i = \1(\eps_i = 1) = (1+\eps_i)/2$ and let $B = (B_1, \dots, B_n) \in \R^n$. Then
\begin{align}
\Gamma_W(\eps; v)
	&= \sup_{w\in W} \sum_{i=1}^{n}
	\big\{\log(\sigma(w_i\eps_i))-\log(\sigma(v_i\eps_i))\big\} \\
	&= \sup_{w\in W} \big\{ \langle w - v, B \rangle - \tilde A(w) + \tilde A(v) \big\} \, , \label{eq:bernoulli-cumulant}
\end{align}
where we use the notation~\eqref{eqn:exp-family-parameters-GLM} 
\[
\tilde A(w) =  \sum_{i=1}^{n} A(w_i) =  \sum_{i=1}^{n} \log(1+\e^{w_i}) \, 
\]
and the argument from~\eqref{eqn:re-write-of-log-likelihood-stat}. One can also deduce~\eqref{eq:bernoulli-cumulant} from the simpler observation that for any $y\in\{-1,1\}$ and $u\in\R$, $\log(\sigma(yu)) = u \1\{y=1\} - \log(1+\e^u)$. Using~\eqref{eq:mgf-partition}, we deduce that
\[
\Gamma_W(\eps; v)
	= \sup_{w\in W}\big\{ \langle w-v, B-p \rangle - \psi_v(w-v)\big\}
	= \sup_{u \in W}\big\{ \langle u, B-p \rangle - \psi_v(u)\big\}
\]
where $p = \sigma(v)$, applied coordinate-wise  and
\[
\psi_v(u) = \sum_{i=1}^n h_{p_i}(u_i),
\] 
where for a one-dimensional parameter $p\in [0,1]$, $h_p$ denotes the logarithmic Laplace transform of a (centered) Bernoulli variable with parameter $p$, namely, for all $s \in \R$,
\[
h_p(s)
	= \log\big( \E \e^{s (B - p)} \big)
	= \log(1-p + p\e^{s}) - sp \, , \quad B \sim \Ber{p} \, .
\]

Now, we parameterize the subspace $W$ as $W=\ran(U)$, with the columns of $U$ forming an orthonormal basis of $W$. Since $\psi_v$ is finite everywhere, the Fenchel duality theorem~\cite[Theorem~31.1]{rockafellar1970convex}, applied directly to the preceding constrained supremum, gives
\[
\Gamma_W(\eps; v)
	= \inf_{\alpha\in\R^n:\,U^\T\!\alpha\,=\,U^\T(B-p)}
    \sum_{i=1}^n h_{p_i}^\ast(\alpha_i).
\]

Finally, we compute the scalar conjugates $h_{p_i}^*$ to prove the claim. Since $p_i=\sigma(v_i)\in(0,1)$, it is enough to compute the scalar conjugate for $p\in(0,1)$. We show that
\begin{equation}
h_p^*(\alpha)= 
\begin{cases}
    \kl{\Ber{p+\alpha}}{\Ber{p}} & \mbox{if}~\alpha \in(-p, 1-p) \\ 
    -\log(1-p) & \alpha = -p\\
    -\log(p) & \alpha = 1-p \\
    +\infty & \mbox{if}~ \alpha < -p~\mbox{or}~\alpha > 1 -p \, .
\end{cases}
\end{equation}

Indeed, letting $q = p + \alpha$, we write 
    \[
    h_p^*(\alpha) = 
    \sup_{u \in \R} \Big\{\underbrace{
    u (\alpha + p) - \log(1 - p + p\e^{u})}_{\eqcolon H_{p, \alpha}(u)} 
    \Big\} = 
    \sup_{u \in \R} \Big\{
    u q - \log(1 - p + p\e^{u}) 
    \Big\}
    \]
    We observe that 
    \[
    \frac{\ud}{\ud u}
    H_{p,\alpha}(u) = q - \frac{\e^u}{1 - p + p\e^u} p
    \]
    Thus, if $q \in (0, 1)$, then the optimal $u^\star(\alpha, p)$ satisfies $\tfrac{(1-p)q}{p(1 - q)}=\e^{u^\star(\alpha, p)}$.
    Hence, we then have  
    \[
    h_p^*(\alpha) = 
    (1-q)\log \frac{1-q}{1-p} +
    q \log \frac{q}{p} 
    =
    \kl{\Ber{q}}{\Ber{p}}
    \]
    By a simple computation (we omit this here), we have, in the case $q > 1$ that $\lim_{u \to \infty} H_{p, \alpha}(u) = \infty$. 
    Similarly, if $q < 0$, then $\lim_{u \to -\infty} H_{p, \alpha}(u) = \infty$. If $q = 0$, then we have by continuity that $h_p^*(\alpha) = -\log(1-p)$. If $q = 1$, similarly $h_p^*(\alpha) = -\log(p)$. Setting $q_i = p_i+\alpha_i$, the constraint
$U^\T\alpha=U^\T(B-p)$ rewrites $\Pi_W q = \Pi_W B$, which establishes the first equality in the claim. The second follows by going back to the $\{-1,1\}$ convention, that is, taking $\delta = 2 q-1$.

\subsubsection{Proof of~\Cref{rem:sharpness-of-1d-cgf-trick}}
\label{sec:proof-of-remark-1d-cgf}
Let us reformulate the remark as the following lemma. 
Recall the construction: $\eps, B_\alpha$ are independent random variables, with $\eps \sim \mathsf{Unif}(\{-1, 1\})$ and $B_\alpha \sim \mathsf{Beta}(1, \alpha)$. We define 
\[
Z_\alpha = \eps B_\alpha, \quad \mbox{and} \quad 
\psi_\alpha(\lambda) = \log \E \e^{\lambda Z_\alpha}. 
\]
\begin{lemma}
\label{lem:restated-remark-on-sharpness}
The variables $Z_\alpha$, $\alpha\in(0,1)$, may be coupled with a random variable
$W\sim\mathsf{Exp}(1)+\log 2$ so that
$\psi_\alpha^\ast(Z_\alpha)$ converges almost surely to $W$ as $\alpha$ tends to zero.
\end{lemma}
Since $W$ has a continuous distribution, the lemma gives the tail assertion in the remark. For the mean, nonnegativity and Fatou's lemma give the required lower bound, while~\Cref{prop:conjugate-scalar-case} gives the matching upper bound. (Formally, apply Fatou's lemma along any sequence of positive values of $\alpha$ tending to zero.)
\begin{proof}
Let $U \sim \mathsf{Unif}(0,1)$ and $\eps\sim \mathsf{Unif}(\{-1,1\})$ be independent. Note that $B_\alpha = 1-U^{1/\alpha}$, and moreover $-\log U = \mathsf{Exp}(1)$ in distribution. We define $W = \log 2-\log U$. 
By symmetry, it holds that 
$\psi_\alpha(\lambda)= \log \E \cosh(\lambda B_\alpha)$, which is even. Therefore, $\psi_\alpha^\ast$ is also even, and hence 
$\psi_\alpha^\ast(Z_\alpha)=\psi_\alpha^\ast(B_\alpha)$. Fix $b\in[0,1)$. For $\lambda>0$, Markov's inequality yields $\P(Z_\alpha\geq b )\leq \e^{\psi_\alpha(\lambda)-\lambda b}$. Hence, 
\begin{equation}
\label{ineq:upper-bound-for-ccgf-remark}
\psi_\alpha^\ast(b)\leq -\log \P(Z_\alpha\geq b) 
= -\log \P(\eps = 1, B_\alpha \geq b) = \log 2 - \alpha \log(1-b).
\end{equation}
On the other hand, we have $m_\alpha(\lambda) \defn \E \e^{\lambda B_\alpha} \leq \tfrac{\e^\lambda}{\lambda^\alpha} \Gamma(1+\alpha)$. Therefore, 
\[
\psi_\alpha(\lambda)=\log\E\cosh(\lambda B_\alpha)
\stackrel{{\rm(i)}}{\leq} \log\Big(\frac{m_\alpha(\lambda)}{2}\Big)+\log\Big(1+\frac{1}{m_\alpha(\lambda)}\Big)
\stackrel{{\rm(ii)}}{\leq} \log\Big(\frac{m_\alpha(\lambda)}{2}\Big)+\frac{1}{m_\alpha(\lambda)}.
\]
Above, inequality (i) used $\cosh(x) \leq \tfrac{\e^x + 1}{2}$, while inequality (ii) used $\log(1+x) \leq x$. Using \mbox{$m_\alpha(\lambda) \geq\e^{\lambda/2}\,\P(B_\alpha\ge\tfrac12)
=\tfrac{\e^{\lambda/2}}{2^{\alpha}}$}, we obtain 
\begin{align}
\psi_\alpha^\ast(b) 
\geq \lambda b - \psi_\alpha(\lambda) 
&\geq \lambda b-\log m_\alpha(\lambda)
+\log 2-\frac{1}{m_\alpha(\lambda)} \\ 
&\geq \lambda (b-1) + \alpha \log \lambda - \log \Gamma(1+\alpha) +\log 2-\frac{1}{m_\alpha(\lambda)} \\ 
&\geq 
\lambda (b-1) + \alpha \log \lambda - \log \Gamma(1+\alpha) +\log 2- 2^\alpha \e^{-\lambda/2},
\label{ineq:final-ineq-remark}
\end{align}
for any $b \in (0, 1)$ and any $\lambda > 0$. 
Taking $\lambda = \tfrac{\alpha}{1-b}$  in~\cref{ineq:final-ineq-remark} and combining with~\cref{ineq:upper-bound-for-ccgf-remark}, we find
\[
0 \geq \psi_\alpha^\ast(b)-[
\log 2-\alpha\log(1-b)] \geq \alpha\log\alpha-\alpha-\log\Gamma(1+\alpha)
- 2^\alpha \exp\Big(-\half \frac{\alpha}{1-b}\Big)
\]
Thus, for every fixed $u\in(0,1)$, $\psi_\alpha^\ast(1-u^{1/\alpha})$ converges to $\log 2-\log u$ as $\alpha$ tends to zero. Here we used that the last exponential term vanishes for every such $u$. Therefore, since $U\in(0,1)$ almost surely, evenness and the identity $B_\alpha=1-U^{1/\alpha}$ show that $\psi_\alpha^\ast(Z_\alpha)=\psi_\alpha^\ast(B_\alpha)$ converges almost surely to $\log 2-\log U=W$, as needed.
\end{proof}

\subsection{Gaussian random design}

Our next result shows that for Gaussian random design the logarithmic factor in
\Cref{thm:main-result} disappears, which can be seen as a nonasymptotic analogue of Wilks' theorem for logistic regression with Gaussian design.
The only limitation of this finite sample result is that our numerical constant is large. This proof uses the regularity of the design vectors and a local quadratic expansion around $\thetastar$, following the recent nonasymptotic analysis of the logistic Hessian in the Gaussian case due to Chardon, Lerasle and Mourtada \cite{chardon2024logistic}. Our key observation is that the complementary regime, where $\|\thetastar\|_2$ is large, can be analyzed directly and does not even require, as expected, that the MLE exists. This contrasts with \cite{chardon2024logistic}, where the existence of the MLE (and thus additional assumptions on the interplay between $\thetastar, n, d$) is central since the focus there is on the excess risk with respect to the population version of the loss.

\begin{theorem}
\label{thm:gaussian-design-logfree}
Assume that $X_1,\dots,X_n \simiid \Normal{0}{I_d}$ and that the logistic model is
well specified with parameter $\thetastar \in \R^d$. Then, for every $\delta \in (0,1)$,
\[
\P_{\thetastar}\left(
\Lambda_n^{\operatorname{log}}(\thetastar)
\le
4\cdot 10^{13}\bigl(d+\log(5/\delta)\bigr)
\right)
\ge
1-\delta.
\]
\end{theorem}

\begin{proof}
The case $d=1$ is already covered by \Cref{cor:dimone}, so we assume $d \geq 2$. Set
\[
\widehat L_n(\theta)
\defn
\frac{1}{n}\sum_{i=1}^n \log\bigl(1+\exp(-Y_i\langle X_i,\theta\rangle)\bigr).
\]
Then
\[
\Lambda_n^{\operatorname{log}}(\thetastar)
=
n\Bigl(\widehat L_n(\thetastar)-\inf_{\theta \in \R^d}\widehat L_n(\theta)\Bigr).
\]
Let $b \defn \|\thetastar\|_2$, $B \defn \max\{e,b\}$, and $t \defn \log(5/\delta)$.
If $\thetastar \neq 0$, set $u_\star \defn \thetastar/\|\thetastar\|_2$; otherwise fix any
$u_\star \in S^{d-1}$. Define
\[
H
\defn
\frac{1}{B^3}u_\star u_\star^\T
+
\frac{1}{B}(I_d-u_\star u_\star^\T).
\]
Throughout the proof, we use the notation $\|x\|_S = \langle S x, x \rangle^{1/2}$ for any positive symmetric matrix $S$ and vector $x$.

We split the proof into two regimes. Assume first that $n > (2800000)^2 B(d+t)$.
In particular,
$n \geq 16B(d+t)$
and therefore, by \cite[Proposition~5, Theorem~6, Lemma~3, and Section~8.2]{chardon2024logistic}, with probability at least
$1-5\exp(-t)$, both of the following hold:
\[
\|\nabla \widehat L_n(\thetastar)\|_{H^{-1}}
\le
14\sqrt{\frac{d+t}{n}},
\]
and, on the same event, $\widehat L_n$ has a unique global minimizer
$\wh \theta_{n}$ which satisfies
\[
\|\wh \theta_{n}-\thetastar\|_H
\le
28000\sqrt{\frac{d+t}{n}}.
\]
By convexity of $\widehat L_n$, it holds that
\[
\widehat L_n(\thetastar)-\widehat L_n(\wh \theta_{n})
\le
\langle \nabla \widehat L_n(\thetastar), \thetastar-\wh \theta_{n}\rangle
\le
\|\nabla \widehat L_n(\thetastar)\|_{H^{-1}}
\|\wh \theta_{n}-\thetastar\|_H,
\]
and therefore
\[
\Lambda_n^{\operatorname{log}}(\thetastar)
=
n\bigl(\widehat L_n(\thetastar)-\widehat L_n(\wh \theta_{n})\bigr)
\le
392000(d+t).
\]

We now turn to the complementary regime
$n \leq (2800000)^2 B(d+t)$. Set 
$
Z_i =
-\log\bigl(\sigma(Y_i\langle X_i,\thetastar\rangle)\bigr)
$ for $i \in [n]$.
Since $\inf_{\theta \in \R^d}\widehat L_n(\theta) \geq 0$, we have
\[
\Lambda_n^{\operatorname{log}}(\thetastar)
\le
n\widehat L_n(\thetastar)
=
\sum_{i=1}^n Z_i.
\]
Assume first that $b \geq e$. Let $U_i \defn \langle X_i,\thetastar\rangle$, so
$U_i \sim \Normal{0}{b^2}$. Conditionally on $U_i$,
\[
\E_{\thetastar}\bigl[\exp(Z_i/2)\mid U_i\bigr]
=
\sigma(U_i)^{1/2}+\sigma(-U_i)^{1/2}
\le
1+\exp(-|U_i|/2),
\]
where the last elementary inequality is straightforward to check. 
Hence, we have
\[
\E_{\thetastar}\exp(Z_i/2)
\le
1+\E \exp(-|U_i|/2)
\le
1+\frac{2}{\sqrt{2\pi}\,b}\int_0^\infty \exp(-u/2)\,du
\le
1+\frac{2}{b}
\le
\exp(2/b).
\]
By independence and Markov's inequality,
\[
\P_{\thetastar}\left(
\sum_{i=1}^n Z_i \geq \frac{4n}{b}+2t
\right)
\le
\exp(-t).
\]
Since now $B=b$, on this event
\[
\Lambda_n^{\operatorname{log}}(\thetastar)
\le
\frac{4n}{b}+2t
\le
\bigl(4(2800000)^2+2\bigr)(d+t).
\]
Assume next that $b<e$. Then, for every $u \in \R$,
 it holds that $
\sigma(u)^{1/2}+\sigma(-u)^{1/2}
\le
\sqrt 2.
$
Therefore, it holds that $\E_{\thetastar}\exp(Z_i/2) \leq \sqrt 2$, and again by independence and
Markov's inequality,
\[
\P_{\thetastar}\left(
\sum_{i=1}^n Z_i \geq n\log 2 + 2t
\right)
\le
\exp(-t).
\]
Since now $B=e$, on this event
\[
\Lambda_n^{\operatorname{log}}(\thetastar)
\le
n\log 2 + 2t
\le
\bigl((2800000)^2 e\log 2 + 2\bigr)(d+t).
\]
In both cases the right-hand side is bounded by
$4\cdot 10^{13}(d+t)$. Recalling that $t=\log(5/\delta)$, we prove the theorem.
\end{proof}

Integrating the tail bound gives, uniformly over $\thetastar\in\R^d$,
\[
\E_{\thetastar}\Lambda_n^{\operatorname{log}}(\thetastar)
\leq 4\cdot10^{13}\bigl(d+\log 5+1\bigr).
\]

We next give the corresponding lower bound for subspaces distributed according to Haar measure.

\begin{proposition}[Lower bounds on the logistic LLR in rotationally invariant random design]
\label{prop:lower-bound-rand-rot-invariant-logistic-llr}
Fix $1 \leq k < n$. Suppose that $W$ is chosen uniformly from the Grassmann manifold of $k$-dimensional subspaces of $\R^n$, independently of the Rademacher vector $\eps$. Then,
for any $\delta \in (0, 1/2]$ it holds that
\[
\Quant{1-\delta}\Big(\Gamma_W(\eps; 0)\Big) \geq \half \Quant{1-\delta}\Big(n \cdot \mathsf{Beta}\Big(\frac{k}{2}, \frac{n-k}{2}\Big)\Big)
\asymp \twomin{n}{\Big(k + \log \frac{1}{\delta}\Big)}.
\]
Here, the quantile on the left is with respect to the joint randomness of $W$ and $\eps$.
\end{proposition}
\begin{proof}[Proof of~\Cref{prop:lower-bound-rand-rot-invariant-logistic-llr}]
    Since $\|\eps\|_2^2=n$ and $W$ is Haar-distributed,
    \[
    \|P_W \eps\|_2^2 \stackrel{\rm d}{=} n\frac{A}{A+B},
    \qquad
    \frac{A}{A+B}\sim\mathsf{Beta}\left(\frac{k}{2},\frac{n-k}{2}\right),
    \]
    for independent random variables $A \sim \chi^2_k$ and $B \sim \chi^2_{n-k}$. The claim follows from~\Cref{prop:fixed-design-lower} and \Cref{lem:beta-quantile}.
\end{proof}

\begin{lemma}
\label{lem:beta-quantile}
For $1 \leq k < n$ and $\delta \in (0, 1/2]$, we have 
\[
\Quant{1-\delta}\bigg(n \cdot \mathsf{Beta}\Big(\frac{k}{2}, \frac{n-k}{2}\Big)\bigg) 
\asymp 
\twomin{n}{\Big(k + \log\Big(\frac{1}{\delta}\Big)\Big)}.
\]
\end{lemma}

\begin{proof}[Proof of~\Cref{lem:beta-quantile}]

Let $A$ and $B$ be independent variables, with $A\sim\chi^2_k$ and $B \sim \chi^2_{n-k}$, and let $Z= A/(A+B)$. It is then a standard fact that $Z \sim \mathsf{Beta}(k/2, (n-k)/2)$. For the upper bound, set $t=2\log(2/\delta)$. By the classical chi-square bound~\cite[Lemma~1]{LauMas2000}, it holds with probability at least $1-\delta$ that
\begin{equation}
	A \leq k + 2\sqrt{kt/2} + t  \quad \text{and} \quad B \geq n-k - 2\sqrt{(n-k)t/2} \, ,
\end{equation}
which implies in particular that
\begin{equation}
	A \leq 2 k + 3t/2 \leq 2(k+t)  \quad \text{and} \quad B \geq \frac{n-k}{2} - t \, .
\end{equation}
If $n>4(k+t)$, the last display gives $B\geq n/4$, and hence
$nZ\leq nA/B\leq8(k+t)$. If $n\leq4(k+t)$, we instead use the deterministic
bound $nZ\leq n$. Since $t\lesssim1+\log(1/\delta)$, these two cases establish
the desired upper bound, including the minimum with $n$.

For the lower bound, we first deal with the case $\delta \leq 1/4$.
Applying~\Cref{lem:chisquarequantiles} with level $\eta = 2\delta \leq 1/2$ yields
$\P\bigl(A \geq (k+\log(1/\eta))/C\bigr) \geq \eta$.
By Markov's inequality,
\[
\P(B\leq2n)\geq1-\frac{n-k}{2n}=\frac{n+k}{2n}>\frac12.
\]
Hence, by independence,
\[
\P\Big(A \geq \frac{k+\log(1/\eta)}{C} ;\, B \leq 2n\Big) > \frac{\eta}{2} = \delta.
\]
Setting $u = k + \log(1/\eta)$, on this event we have
\[
nZ = \frac{nA}{A+B} \geq \frac{nu/C}{u/C+2n} = \frac{nu}{u+2Cn} \geq \frac{\min\{u,n\}}{1+2C},
\]
where the last inequality uses $u + 2Cn \leq (1+2C)\max\{u,n\}$.
Since $\delta \leq 1/4$ implies $\log(1/\eta) = \log(1/(2\delta)) \geq \frac{1}{2}\log(1/\delta)$, we have $u \geq (k+\log(1/\delta))/2$, so $\min\{u, n\} \geq \min\{k+\log(1/\delta), n\}/2$, and therefore $\Quant{1-\delta}(nZ) \geq c \min\{n, k+\log(1/\delta)\}$.

For $\delta \in (1/4, 1/2]$, the standard lower-tail Chernoff bound for a chi-square variable gives
\[
\P(A\leq k/100)
\leq \bigl(10^{-2}\e^{0.99}\bigr)^{k/2}<\frac14.
\]
In addition, by Markov's inequality, $\P(B \leq 4(n-k)) \geq 3/4$, and by independence of $A$ and $B$,
\[
\P\bigl(A \geq k/100 ;\, B \leq 4(n-k)\bigr) > 9/16 > 1/2 \geq \delta.
\]
On this event,
\[
nZ \geq \frac{n \cdot k/100}{k/100 + 4(n-k)} \geq \frac{k}{400}.
\]
Since $\log(1/\delta) \leq \log 4$ and $k \geq 1$, one has $k + \log(1/\delta) \leq k(1+\log 4)$, so $\Quant{1-\delta}(nZ) \geq k/400 \geq c(k+\log(1/\delta))$ with a smaller constant $c$.
\end{proof}

The following corollary combines \Cref{thm:gaussian-design-logfree} and \Cref{prop:lower-bound-rand-rot-invariant-logistic-llr} with the elementary bound $\Lambda_n^{\operatorname{log}}(0)\leq n\log 2$ and handles the endpoint $d=n$ separately.

\begin{corollary}[Gaussian design bounds at the origin]
\label{cor:gaussian-design-origin}
\leavevmode\\
Let \(1\le d\le n\), and let \(X_1,\ldots,X_n\simiid \Normal{0}{I_d}\). Assume that the logistic model is well specified with \(\thetastar=0\). Then, for every \(\delta\in(0,1/2]\),
\[
\Quant{1-\delta} \Bigl( \Lambda_n^{\operatorname{log}}(0;X_{1:n};Y_{1:n}) \Bigr) \asymp \min\left\{ n,\, d+\log\left(\frac1\delta\right) \right\},
\]
where the quantile is with respect to both the Gaussian design and the labels.
\end{corollary}

\begin{proof}
For every realization of the design and the labels, the likelihood under the
origin is $2^{-n}$, while the supremum likelihood is at most $1$. Hence,
\[
\Lambda_n^{\operatorname{log}}(0;X_{1:n};Y_{1:n})\leq n\log 2.
\]
Combining this deterministic bound with
\Cref{thm:gaussian-design-logfree}, and absorbing numerical constants, gives
\[
\Quant{1-\delta}
\Bigl(\Lambda_n^{\operatorname{log}}(0;X_{1:n};Y_{1:n})\Bigr)
\lesssim
\min\left\{n,\,d+\log\left(\frac1\delta\right)\right\}.
\]
It remains to prove the matching lower bound.

Let \(X\in\R^{n\times d}\) be the design matrix with rows \(X_i^\T\), and set
\(W=\ran(X)\). Since \(\thetastar=0\), the labels
\(Y=(Y_1,\ldots,Y_n)\) are independent Rademacher random variables, independent
of \(X\). Moreover, by
\Cref{lem:reformulation-via-subspaces}\ref{item:rewrite-expected-LLR-via-subspace},
\[
\Lambda_n^{\operatorname{log}}(0;X_{1:n};Y_{1:n})
=
\Gamma_W(Y;0).
\]

Assume first that \(d<n\). Then \(X\) has rank \(d\) almost surely, and by
rotational invariance of the Gaussian design, \(W\) is uniformly distributed on
the Grassmann manifold of \(d\)-dimensional subspaces of \(\R^n\). Therefore,
\Cref{prop:lower-bound-rand-rot-invariant-logistic-llr}, applied with
\(k=d\), gives
\[
\Quant{1-\delta}
\Bigl(
\Lambda_n^{\operatorname{log}}(0;X_{1:n};Y_{1:n})
\Bigr)
=
\Quant{1-\delta}\bigl(\Gamma_W(Y;0)\bigr)
\gtrsim
\min\left\{
n,\,
d+\log\left(\frac1\delta\right)
\right\}.
\]

It remains only to handle the endpoint \(d=n\). In this case
\(W=\R^n\) almost surely. Hence, for every \(Y\in\{-1,1\}^n\),
\[
\Gamma_{\R^n}(Y;0)
=
\sup_{w\in\R^n}\log\frac{p_w(Y)}{p_0(Y)}
=
n\log 2,
\]
because \(p_0(Y)=2^{-n}\) and the supremum of \(p_w(Y)\) over \(w\in\R^n\) is
equal to \(1\). Thus the \((1-\delta)\)-quantile is \(n\log 2\). This completes the
proof.
\end{proof}

\subsection{Near the origin regimes}
\label{sec:near-origin-regime}

In this subsection we provide a logarithm-free bound in expectation when the signal vector is in a
neighborhood of the origin. The result holds for certain rotationally invariant designs, in particular when the design matrix $X\in \R^{n\times d}$ has \iid\ rows $X_1^\top, \dots, X_n^\top$ with Gaussian distribution. The argument is different from the proof of the worst-case upper bound based on the Shtarkov sum: the core of the argument is \Cref{lem:critical-radius-mgf-bound} below, which relies on the geometry of a random nullspace and the use of an approximate kinematic formula from conic geometry~\cite{amelunxen2014living} that controls the probability of intersection of certain cones with random subspaces depending on a quantity called statistical dimension.

Throughout, we assume that conditionally on $X$, the labels $Y_1,\dots,Y_n\in\{-1,1\}$ follow the well-specified logistic model where $\P(Y_i=1\mid X)=\sigma(\langle \thetastar, X_i \rangle)$. For a subspace \(V\subset\R^n\), define its critical radius by
\begin{equation}
\label{eq:crit-radius-and-cones}
r_\star(V) \defn \inf\left\{ 0<r<\sqrt n: V\cap \mathcal C_r\neq\{0\} \right\},
\qquad\text{where}\quad
\mathcal C_r
\defn
\operatorname{cone}\left((B_\infty^n\cap rB_2^n)+\1_n\right),
\end{equation}
with the convention \(r_\star(V)=\sqrt n\) if the set above is empty.
For $0<r<\sqrt n$, the generating set $(B_\infty^n\cap rB_2^n)+\1_n$ is compact and convex, and its distance from the origin is at least $\sqrt n-r>0$; hence $\mathcal C_r$ is a closed convex cone. For $r=\sqrt n$, the generating set is $[0,2]^n$ and $\mathcal C_r=\R_+^n$, so the same conclusion holds.

\begin{lemma}[Exponential moment of the critical radius]
\label{lem:critical-radius-mgf-bound}
Let \(1\le d\le n\), and let \(W\) be uniformly distributed over the
Grassmann manifold of \((n-d)\)-dimensional subspaces of \(\R^n\). Then there
exist universal constants \(C>0\) and \(\lambda_0\in(0,1)\) such that
\[
\log \E_W \exp\left(\lambda_0 r_\star(W)^2\right)
\leq C d .
\]
\end{lemma}

For a closed convex cone $\mathcal C\subset\R^n$, write
$\Delta(\mathcal C)\defn\E\|\Pi_{\mathcal C}g\|_2^2$, where
$g\sim\Normal{0}{I_n}$, for its statistical dimension. The proof uses the
following control for the cones $\mathcal C_r$.
\begin{lemma}
\label{lem:quadratic-stat-dim}
There exist universal constants \(c_1,c_2 > 0\) such that, for every integer
\(n\geq 1\) and every $r \in [1,\sqrt n]$,
\[
c_1 r^2 \leq \Delta(\mathcal C_r) \leq c_2 r^2.
\]
\end{lemma}
The proof of \Cref{lem:quadratic-stat-dim} is deferred to
\Cref{sec:near-origin-remaining-proofs}.

\begin{proof}[Proof of~\Cref{lem:critical-radius-mgf-bound}]
We use the following conic kinematic estimate
from~\cite[Theorem~7.1 and Eq.~(6.1)]{amelunxen2014living}: if \(W\) is a uniformly distributed
\((n-d)\)-dimensional subspace and \(s>0\) satisfies
\(d\leq \Delta(\mathcal C)-s\), then
\[
\P(W\cap \mathcal C=\{0\})
\leq
4\exp\left(
-\frac{s^2}{
8\bigl(\min\{n-\Delta(\mathcal C),\Delta(\mathcal C)\}+s\bigr)}
\right).
\]
We first derive a tail bound for the critical radius \(r_\star(W)\). Choose a sufficiently large
universal constant \(A>0\), and put
\[
a=A\sqrt d,
\qquad
b=\sqrt{\frac{n}{2c_2}}.
\]
The constant \(A\) is chosen so that \(a\ge 1\) and
\(c_1a^2/2\ge d\) for every \(d\ge1\). If \(a\le b\), then for every
\(r\in[a,b]\) we have
\[
d\leq \frac{c_1r^2}{2}
\leq
\Delta(\mathcal C_r)-\frac{c_1r^2}{2},
\qquad
\Delta(\mathcal C_r)\le c_2r^2\le \frac n2.
\]
Applying the conic kinematic estimate with \(s=c_1r^2/2\), and increasing
$A$ if needed to absorb the leading factor $4$, gives
\[
\P(W\cap \mathcal C_r=\{0\})
\leq
\exp(-c r^2)
\]
for a universal constant \(c>0\). By the definition of \(r_\star\),
\[
\{r_\star(W)>r\}
\subseteq
\{W\cap\mathcal C_r=\{0\}\},
\qquad 0<r<\sqrt n.
\]
Thus, when \(a\le b\),
\[
\P(r_\star(W)>r)
\leq
\begin{cases}
1, & 0\leq r<a,\\
\exp(-c r^2), & a\leq r\leq b,\\
\exp(-c b^2), & b\leq r\leq \sqrt n,\\
0, & r\geq \sqrt n.
\end{cases}
\]
The third line follows by monotonicity of \(r\mapsto\P(r_\star(W)>r)\), and
the last line follows from the deterministic bound \(r_\star(W)\le\sqrt n\).

Choose
\[
\lambda_0
\leq
\min\left\{\frac12,\frac c2,\frac{c}{4c_2}\right\}.
\]
If \(a>b\), then \(n\lesssim d\), and the deterministic bound
\(r_\star(W)\le \sqrt n\) gives
\[
\log \E_W \exp\left(\lambda_0 r_\star(W)^2\right)
\leq
\lambda_0 n
\lesssim d.
\]
It remains to consider the case \(a\le b\). Using
\[
\E_W \exp\left(\lambda_0 r_\star(W)^2\right)
=
1+
\int_0^\infty
2\lambda_0 t\exp(\lambda_0 t^2)
\P(r_\star(W)>t)\,dt,
\]
and splitting the integral over \([0,a]\), \([a,b]\), and
\([b,\sqrt n]\), we get
\begin{align}
\E_W \exp\left(\lambda_0 r_\star(W)^2\right)
&\leq
1+
\int_0^a
2\lambda_0 t e^{\lambda_0 t^2}\,dt
+
\int_a^b
2\lambda_0 t e^{-(c-\lambda_0)t^2}\,dt  \\
&\qquad
+
\int_b^{\sqrt n}
2\lambda_0 t e^{\lambda_0 t^2-cb^2}\,dt .
\end{align}
The first integral is at most
\[
\exp(\lambda_0 a^2)-1
\leq
\exp(Cd)-1.
\]
The second integral is bounded by a universal constant, since
\(\lambda_0\le c/2\). For the third integral, using
\(b^2=n/(2c_2)\) and \(\lambda_0\le c/(4c_2)\), we have
\[
\int_b^{\sqrt n}
2\lambda_0 t e^{\lambda_0 t^2-cb^2}\,dt
\leq
\exp(\lambda_0 n-cb^2)
\leq 1 .
\]
Therefore
\[
\E_W \exp\left(\lambda_0 r_\star(W)^2\right)
\leq
C_0\exp(C_1d)
\]
for universal constants \(C_0,C_1>0\), which proves the claim.
\end{proof}

Next, we present our main result in this section.

\begin{proposition}[Expected LLR near the origin]
\label{prop:near-origin-expected-llr}
Let \(n\geq d\). Let \(X\in\R^{n\times d}\) be a random design matrix with rows
\(X_i^\T\). Assume that 

\begin{enumerate}[label=\textup{(\roman*)}]
\item \(W=\ker(X^\T)\) is uniformly distributed over the Grassmann manifold of \((n-d)\)-dimensional subspaces of \(\R^n\);
\item there exists $L\geq 1$ such that $\E X_iX_i^\T \preceq L \cdot I_d$ for all $ i\in[n]$.
\end{enumerate}

Then, for every \(\thetastar\in\R^d\),
\[
\E_{\thetastar} \Lambda_n^{\operatorname{log}}(\thetastar;X_{1:n};Y_{1:n})
	\lesssim d+n \cdot \min\{L\|\thetastar\|_2^2,1\}.
\]
In particular, for Gaussian random design \(X_i\simiid \Normal{0}{I_d}\), if
\(\|\thetastar\|_2\lesssim \sqrt{d/n}\), then
\[
\E_{\thetastar}
\Lambda_n^{\operatorname{log}}(\thetastar;X_{1:n};Y_{1:n})
\lesssim d .
\]
\end{proposition}

Some discussion is in order before we prove this result. Assumption (i) holds in particular when the vectors $X_i$ are \iid\ Gaussian in $\R^{d}$ with arbitrary covariance $\Sigma \succ 0$. Indeed, one can write $X = Z \Sigma^{1/2}$ with $Z\in \R^{n\times d}$ having standard Gaussian entries. Then, $\ran(X) = \ran(Z)$ is Haar-distributed.

Assumption (i) is not restricted to Gaussian designs. Indeed, any design matrix $X$ of rank $d$ almost surely whose distribution is invariant under the left action of the orthogonal group $O(n)$ satisfies the condition. Namely, if for every matrix $Q \in O(n)$, $X\stackrel{(d)}{=} QX$, then the random subspace $\ker(X^\top)$ is Haar-distributed. A specific situation where this holds is if the \emph{columns} $X^{(1)}, \dots, X^{(d)}$ of $X$ are \iid\ with a rotationally invariant distribution in $\R^n$ and $X$ has rank $d$ almost surely. In addition, in this case, one has $\E X_i X_i^\top = \E[\|X^{(1)}\|^2]/n \cdot I_d$.

\begin{proof}
For a fixed design \(X\), let \(\P_{X,\theta}\) denote the product law on
\(\{-1,1\}^n\) under the logistic model with parameter \(\theta\). For
\(\delta\in[-1,1]^n\), let \(\Q_\delta\) be the product law under which the
\(i\)-th coordinate has mean \(\delta_i\), that is
\[
\Q_\delta\{Y_i=1\}=\frac{1+\delta_i}{2}.
\]
Define
\[
I_n(X,Y)
\defn
\inf_{\substack{\delta\in[-1,1]^n\\ X^\T\delta=X^\T Y}}
\sum_{i=1}^n
\kl{\Ber{\frac{1+\delta_i}{2}}}{\Ber{\frac12}} .
\]
We will then use the following: for any $\delta \in [-1,1]$ and $u \in \R$,
\[
\kl{\Ber{\frac{1+\delta}{2}} \! }{\Ber{\sigma(u)}}
=
\kl{\Ber{\frac{1+\delta}{2}} \! }{\Ber{\frac12}}
-\frac{u\delta}2
+\log\cosh\left(\frac u2\right),
\; |\delta|\leq 1, \, u\in\R.
\]
Applying this identity coordinatewise and using the constraint \(X^\T\delta=X^\T Y\), the usual variational representation of the logistic
likelihood ratio from Lemma~\ref{lem:var-char} (applied to the subspace $\ran(X)$ and vector $v^\star = X \thetastar$) gives
\begin{equation}
\label{eq:ll-var-kl}
\Lambda_n^{\operatorname{log}}(\thetastar;X_{1:n};Y_{1:n})
	= I_n(X,Y) - \log\frac{d\P_{X,\thetastar}}{d\P_{X,0}}(Y) \, ,
\end{equation}
since, using the elementary identity $\log(\sigma(\eps u)) = \eps u/2 - \log(2 \cosh(u/2))$ (valid for any $\eps \in	\{-1,1\}$ and $u\in \R$), one has
\[
\log\bigg(\frac{d\P_{X,\thetastar}}{d\P_{X,0}}(Y)\bigg)
	=  \sum_{i=1}^{n} \Big\{ \log \sigma(Y_i v^\star_i) - \log(1/2) \Big\}
	=\frac 1 2 \langle X \thetastar, Y \rangle -   \sum_{i=1}^{n} \log(\cosh(v^\star_i /2)) \, .
\]

We first compare $I_n(X,Y)$ to the critical radius. Let
$W=\ker(X^\T)$ and $D_Y=\operatorname{diag}(Y_1,\ldots,Y_n)$.
First, by a standard comparison, 
\[
\kl{\Ber{\frac{1+\delta}{2}}}{\Ber{\frac12}} \leq \delta^2
\qquad\textrm{for all } |\delta| \leq 1 \, .
\]
Next, we prove that
\begin{equation}
\label{eq:var-rep-crit-radius}
I_n(X,Y)
	\leq \inf_{\substack{\delta\in[-1,1]^n\\ \delta\in Y+W}} \|\delta\|_2^2
	= r_\star(D_YW)^2.
\end{equation}
To prove~\eqref{eq:var-rep-crit-radius}, observe that the constraint $X^\top \delta = X^\top Y$ is equivalent to $\delta \in Y + W$ since $W = \ker(X^\top)$. By construction $D_Y Y = \1_n$, so this is also equivalent to $D_Y \delta \in \1_n + D_Y W$. Since $D_Y$ is an isometry, for $0<r<\sqrt n$, the existence of
$\delta\in B_\infty^n\cap rB_2^n\cap(Y+W)$ is equivalent, on setting
$v=-D_Y\delta$, to the existence of $v\in B_\infty^n\cap rB_2^n$ such that
$\1_n+v\in D_YW$. This is exactly
the condition $D_YW\cap\mathcal C_r\neq\{0\}$, with the convention
$r_\star(D_YW)=\sqrt n$ covering the boundary case. This establishes~\eqref{eq:var-rep-crit-radius}.

Fix $\lambda\in(0,1)$. Conditioning on $X$ and taking expectation with respect to $Y$ in~\eqref{eq:ll-var-kl} yields
\begin{equation}
\E_{Y\sim\P_{X,\thetastar}}\left[\Lambda_n^{\log}(\thetastar;X_{1:n};Y_{1:n})\right]
	= \E_{Y\sim\P_{X,\thetastar}} I_n(X,Y) - \kl{\P_{X,\thetastar}}{\P_{X,0}} \, .
\end{equation}
Now, by the change-of-measure inequality~\cite[Corollary~4.14]{boucheron2013concentration},
\[
\E_{Y\sim\P_{X,\thetastar}}[I_n(X,Y)]
	\leq \frac 1 \lambda \log\E_{Y\sim\P_{X,0}}\e^{\lambda I_n(X,Y)}
		+ \frac1\lambda\kl{\P_{X,\thetastar}}{\P_{X,0}} \, .
\]
Combining with the previous display we find,
\[
\E_{Y\sim\P_{X,\thetastar}}\!
\left[
\Lambda_n^{\log}(\thetastar;X_{1:n};Y_{1:n})
\right]
\leq
\frac1\lambda
\log
\E_{Y\sim\P_{X,0}} 
\exp\bigl(\lambda I_n(X,Y)\bigr)
+
\left(\frac1\lambda-1\right)
\kl{\P_{X,\thetastar}}{\P_{X,0}}.
\]
Taking expectation in $X$, applying Jensen's inequality to the logarithm, and
using $I_n(X,Y)\le r_\star(D_YW)^2$, we obtain
\[
\E_{\thetastar} \Lambda_n^{\operatorname{log}}(\thetastar;X_{1:n};Y_{1:n})
	\leq \frac1\lambda \log \E_X\E_{Y\sim\P_{X,0}} \exp\bigl(\lambda r_\star(D_YW)^2\bigr)
		+ \left(\frac1\lambda-1\right) \E_X \kl{\P_{X,\thetastar}}{\P_{X,0}}.
\]
Under $\P_{X,0}$, the labels are independent Rademacher variables and are
independent of $X$. Since $W$ is Haar-distributed and multiplication by
$D_Y$ is an orthogonal transformation, $D_YW$ is again Haar-distributed.
Choosing $\lambda=\lambda_0$ in
\Cref{lem:critical-radius-mgf-bound}, the first term is bounded by $Cd$.

It remains to bound the change-of-measure term. For $U_i = \langle X_i, \thetastar \rangle$,
\[
\kl{\P_{X,\thetastar}}{\P_{X,0}}
	= \sum_{i=1}^n \kl{\Ber{\sigma(U_i)}}{\Ber{\frac12}}
	\lesssim \sum_{i=1}^n \left(\sigma(U_i)-\frac12\right)^2 \, .
\]
Moreover,
\[
\left(\sigma(u)-\frac12\right)^2 \lesssim \min\{u^2,1\}, \qquad u\in\R,
\]
using that $t\mapsto\sigma(t)$ is $1/4$-Lipschitz and that the left-hand
side is bounded by $1/4$. Therefore,
\[
\E_X \kl{\P_{X,\thetastar}}{\P_{X,0}}
\lesssim
\sum_{i=1}^n
\E\min\{\langle X_i,\thetastar\rangle^2,1\}.
\]
Since $x\mapsto \min\{x,1\}$ is concave on $\R_+$, Jensen's inequality and
the covariance assumption give
\[
\E\min\{\langle X_i,\thetastar\rangle^2,1\}
\leq
\min\left\{
\E\langle X_i,\thetastar\rangle^2,1
\right\}
\leq
\min\left\{
L\|\thetastar\|_2^2,1
\right\}.
\]
Combining the previous displays gives
\[
\E_X\kl{\P_{X,\thetastar}}{\P_{X,0}}
\lesssim
n\min\{L\|\thetastar\|_2^2,1\}.
\]
Together with the critical-radius bound, this proves
\[
\E_{\thetastar}
\Lambda_n^{\operatorname{log}}(\thetastar;X_{1:n};Y_{1:n})
\lesssim
d+n \cdot \min\{L\|\thetastar\|_2^2,1\}.
\]
For Gaussian random design, $W=\ker(X^\T)$ is Haar-distributed and
$\E X_iX_i^\T=I_d$. Hence $L=1$, and
$\|\thetastar\|_2\lesssim \sqrt{d/n}$ gives the claimed $O(d)$ bound.
\end{proof}

\subsubsection{Remaining proofs}
\label{sec:near-origin-remaining-proofs}
\begin{proof}(\emph{of Lemma~\ref{lem:quadratic-stat-dim}})
Fix $r>0$. It is clear from the definition of $\mathcal{C}_r$ that 
\[
\mathcal C_r
\subset
\operatorname{cone}\left(\1_n+rB_2^n\right),
\]
that we will denote by $\mathcal{C}'_r$.
For any cone $\mathcal{C}\subset \R^n$, write $H_{\mathcal{C}} = \sup_{v \in B_2^n \cap \mathcal{C}} \langle g, v\rangle$, with $g$ a standard Gaussian vector. Then, by~\cite[Prop.~3.1(5)]{amelunxen2014living}, $\Delta(\mathcal{C}'_r) = \E [H_{\mathcal{C}'_r}^2]$. Thus, by monotonicity, $\Delta(\mathcal{C}_r) \leq \E[ H_{\mathcal{C}'_r}^2]$ as well, so we proceed by bounding the latter from above.
Suppose first that $r\in[1,\sqrt n/2]$. Let $v \in \mathcal{C}'_r \cap B_2^n$ and write it as $v = \lambda (\1_n + z)$ for some $\lambda \geq 0$ and $z \in \R^n$ with $\|z\| \leq r$. The Cauchy--Schwarz inequality gives $\langle g,\1_n + z \rangle \leq \langle g,\1_n \rangle + r\|g\|$, and the reverse triangle inequality gives $\|\1_n+z\|\geq\sqrt{n} - r$, hence $\lambda \leq 1/(\sqrt{n} - r)$. It follows that
\[
\langle g, v\rangle
	= \lambda\langle g,\1_n + z\rangle
	\leq \lambda\big(\langle g,\1_n\rangle + r\|g\|\big)_+
	\leq \frac{\langle g,\1_n\rangle_+ + r\|g\|}{\sqrt{n} - r} \, .
\]
The last inequality holds for all $\lambda \geq 0$ and $z \in rB_2^n$, so, taking the supremum, we deduce that 
\begin{equation}
	H_{\mathcal{C}'_r} \leq \frac{\langle g,\1_n\rangle_+ + r\|g\|}{\sqrt{n} - r} \, .
\end{equation}
Since $\sqrt{n} - r \geq \sqrt{n} /2$, taking the square gives
\[
H_{\mathcal{C}'_r}^2
	\leq \frac{2\langle g, \1_n \rangle_+^2 + 2r^2 \|g\|^2}{(\sqrt n-r)^2}
	\leq \frac{8\langle g, \1_n \rangle_+^2 + 8r^2 \|g\|^2}{n}.
\]
Since $\langle g,\1_n\rangle\sim\Normal{0}{n}$, we have $\E[\langle g,\1_n\rangle_+^2]=n/2$, and $\E\|g\|_2^2=n$, so taking expectations gives $\Delta(\mathcal C_r)\leq4+8r^2\leq12r^2$, the last step using $r\geq 1$.
If instead $r\in(\sqrt n/2,\sqrt n]$, then $n < 4r^2$ and we simply use the trivial bound $\Delta(\mathcal{C}_r) \leq n \leq 4 r^2$. In either case, $\Delta(\mathcal C_r)\leq c_2 r^2$ with $c_2=12$.

We now turn to the lower bound. First suppose that $r\in(0,\sqrt n)$. As before, $g$ denotes a standard Gaussian vector. Let $Z \defn(r/\sqrt n)\sign(g)$; then $\|Z\|_\infty=r/\sqrt n\leq1$ and $\|Z\|=r$, so $Z \in B_\infty^n\cap rB_2^n$ and $\1_n + Z$ lies in the generating set of $\mathcal C_r$. Thus, letting
\[
U = \frac{\1_n + Z}{\|\1_n + Z\|} \in \mathcal C_r \cap B_2^n \, ,
\]
 and since $\|\1_n + Z\|_2\leq\sqrt n+r\leq2\sqrt n$,
\[
\langle g,U\rangle=\frac{\xi}{\|\1_n+Z\|_2}, \qquad \xi\defn\langle g,\1_n\rangle+\frac{r}{\sqrt n}\|g\|_1.
\]
Since $0\in\mathcal C_r\cap B_2^n$ as well, $H_{\mathcal{C}_r} \geq \max\{0,\langle g,U\rangle\}$, and this is at least $\xi/(2\sqrt n)$ regardless of the sign of $\xi$. Indeed, when $\xi\geq0$, this follows from $\langle g,U\rangle\geq \xi/(2\sqrt n)$ since $\|\1_n + Z\| \leq 2\sqrt n$, and when $\xi < 0$, the quantity $\xi/(2\sqrt n)$ is negative while $\max\{0,\langle g,U\rangle\} \geq 0$. Hence $H_{\mathcal{C}_r}\geq \xi/(2\sqrt n)$. Taking expectations, using $\E\langle g,\1_n\rangle=0$ and $\E\|g\|_1 = n\sqrt{2/\pi}$, we deduce
\[
\E[H_{\mathcal C_r}(g)]
	\geq \frac{\E[\xi]}{2\sqrt n}
	= \frac{r}{\sqrt{2\pi}} \, .
\]
	By Jensen's inequality, $\Delta(\mathcal C_r) = \E[H_{\mathcal C_r}^2] \geq \E[H_{\mathcal C_r}]^2 \geq r^2/(2\pi)$ for $0<r<\sqrt n$. At the endpoint $r=\sqrt n$, we have $\mathcal C_r=\R_+^n$ and hence $\Delta(\mathcal C_r)=n/2\geq r^2/(2\pi)$. Thus, the claim holds with $c_1=1/(2\pi)$.

\end{proof}

\section{Sharpness and asymptotic bounds}
\label{sec:sharpness-boundaries}

In this section, we provide several results showing the sharpness of previous results. Theorem~\ref{thm:shtarkov-logistic} in Section~\ref{sec:shtarkov} below provides precise upper and lower bounds on Shtarkov sums, a central object in our analysis of the logistic LLR. In particular, when $d$ divides $n$, its bounds have matching leading constants in the regime $n/d\to\infty$.

Next, Section~\ref{sec:bivariatecase} analyzes separately the two-dimensional case. Indeed, Proposition~\ref{prop:strong-lower-bound-vandermonde} addresses the case where $d\geq 3$ and shows that the logarithmic factor is unavoidable, which implies that the behavior of the LLR is structurally different from the $\chi^2$ behavior that one would expect from Wilks' phenomenon. On the other hand, Corollary~\ref{cor:dimone} shows that the logarithmic overhead can be removed in the univariate case. Proposition~\ref{prop:d2-logloglog-lower} shows that this property collapses even for $d=2$. Although quantitatively much smaller than the lower bound from Proposition~\ref{prop:strong-lower-bound-vandermonde}, this result already shows that one cannot structurally expect a $\chi^2$-type behavior in dimension two. The matching worst-case quantile bounds are given in \Cref{cor:d2-logloglog-quantile-lower} and \Cref{thm:d2-logloglog-upper}.

Finally, in Section~\ref{sec:wilksregime}, we show that in random design, the standard high-dimensional asymptotic regime where the relevant parameter is the (asymptotic) aspect ratio $\kappa = \lim d/n$ does not suffice to characterize the behavior of the LLR. Instead, Proposition~\ref{prop:chi-limit-behavior} shows that whether or not the $\chi^2$ distribution is a valid approximation (assessed in Kolmogorov-Smirnov distance) for the distribution of the LLR depends on the limiting behavior of $d^{3/2}/n$, not merely $d/n$.

\subsection{The Laplace transform of the logistic LLR and sharp bounds for Shtarkov sums}
\label{sec:shtarkov}
Before proceeding, we recall the notion of the Shtarkov sum, introduced in~\cite{Shtarkov1987} in the context of universal coding; see also the textbook~\cite{CesaBianchiLugosi2006} for a detailed exposition. In the fixed-design logistic model, the logarithm of this quantity coincides with the optimal regret for sequential prediction in the {transductive} setting, where the covariates \(x_1,\dots,x_n\) are known in advance. In our notation, we define
\begin{equation}
\label{eq:shtarkovsum}
\mathcal{S}_d^{\operatorname{log}}(x_1,\dots,x_n)
\defn
\sum_{\varepsilon \in \{-1,1\}^n}
\sup_{\theta \in \R^d}
\prod_{i=1}^n \sigma(\varepsilon_i\langle x_i,\theta\rangle).
\end{equation}

In our context, the role of the Shtarkov sum is somewhat different: it coincides exactly with the moment generating function of the log-likelihood ratio statistic at \(\lambda = 1\). Indeed, by definition,
\[
\E_{\thetastar}\exp\bigl(\Lambda_n^{\operatorname{log}}(\thetastar)\bigr)
=
\mathcal{S}_d^{\operatorname{log}}(x_1,\dots,x_n).
\]
In particular, the right-hand side does not depend on \(\thetastar\). Thus, the Shtarkov sum serves two purposes in our analysis. First, it is the main tool for proving the general upper tail bound in the fixed-design setting. Second, it captures an important aspect of the aggregated moment behavior of the log-likelihood ratio statistic \(\Lambda_n^{\operatorname{log}}(\thetastar)\). Remarkably, even for finite \(n\), this already reveals a subtle departure from the classical \(\chi_d^2\) picture: at \(\lambda = 1\), the moment generating function of \(\frac12 \chi_d^2\) is not finite, whereas in our setting the corresponding quantity is finite and admits nontrivial structural bounds.

The next theorem summarizes the behavior of the Shtarkov sum in the regimes relevant for our analysis. 

\begin{theorem}[Shtarkov sum for logistic regression]
\label{thm:shtarkov-logistic}
The following hold.
\begin{enumerate}[label=(\roman*)]
\item
\label{item:shtarkov-upper}
If $n \geq d$, then for every collection $x_1,\dots,x_n \in \R^d$,
\[
\log\bigl(\mathcal{S}_d^{\operatorname{log}}(x_1,\dots,x_n)\bigr)
\le
\log\Bigl(\sum_{\ell=0}^{d}\binom{n}{\ell}\Bigr)
\le
d\log\Bigl(\frac{en}{d}\Bigr).
\]

\item
If $d=1$, $n \geq 2$, and $x_1,\dots,x_n \in \R$ are nonzero, then
\[
\log\bigl(\mathcal{S}_1^{\operatorname{log}}(x_1,\dots,x_n)\bigr)
\ge
\frac12 \log n.
\]
More generally, for arbitrary $d \geq 1$, assume that $d$ divides $n$ and that the design has a balanced $d$-block parallel structure: the vectors $x_1,\dots,x_n$ can be partitioned into $d$ groups of size $n/d$, and for each group all vectors in that group are nonzero multiples of a common vector $u_j$, where $u_1,\dots,u_d \in \R^d$ are linearly independent. Then
\[
\log\bigl(\mathcal{S}_d^{\operatorname{log}}(x_1,\dots,x_n)\bigr)
\ge
\frac{d}{2}\log\Bigl(\frac{n}{d}\Bigr).
\]

\item
For any $d \geq 1$, if $d$ divides $n$, then
\[
\sup_{x'_1,\dots,x'_n \in \R^d}
\log\bigl(\mathcal{S}_d^{\operatorname{log}}(x'_1,\dots,x'_n)\bigr)
\ge
d\log\Bigl(\frac{n}{d}+1\Bigr).
\]

\item
For any $n \geq d \geq 1$, if $x_1,\dots,x_n$ are in general position, then
\[
\log\bigl(\mathcal{S}_d^{\operatorname{log}}(x_1,\dots,x_n)\bigr)
\ge
\log\Bigl(2\sum_{\ell=0}^{d-1}\binom{n-1}{\ell}\Bigr).
\]
In particular, under the general position assumption, if $d \geq 2$, then
\[
\log\bigl(\mathcal{S}_d^{\operatorname{log}}(x_1,\dots,x_n)\bigr)
\ge
(d-1)\log\Bigl(\frac{n-1}{d-1}\Bigr).
\]
\end{enumerate}
\end{theorem}

Theorem~\ref{thm:shtarkov-logistic} has several consequences. Combining items~(i) and~(iii), we obtain that whenever \(n \geq d\) and \(d\) divides \(n\),
\[
d\log\Bigl(\frac{n}{d}\Bigr)
\le
\sup_{x_1,\dots,x_n \in \R^d}
\log\bigl(\mathcal{S}_d^{\operatorname{log}}(x_1,\dots,x_n)\bigr)
\le
d\log\Bigl(\frac{n}{d}\Bigr)+d.
\]
Thus, in the fixed-design setting, the worst-case Shtarkov sum grows on the scale \(d\log(en/d)\), and when $n/d\to\infty$ the displayed upper and lower bounds match at the level of the leading constant.

Related quantities for logistic models have recently been studied in \cite{ShamirSzpankowski2021,JacquetShamirSzpankowski2022,DrmotaJacquetWuSzpankowski2024,drmota2025precise,DrmotaJacquetWuSzpankowski2026,QianRakhlinZhivotovskiy2026}, typically with a different emphasis: asymptotics or particular design ensembles (e.g., points sampled uniformly on the sphere), rather than the sharp nonasymptotic worst-case bound considered here. The papers \cite{DrmotaJacquetWuSzpankowski2024,DrmotaJacquetWuSzpankowski2026} claim bounds of order $2d\log n$, which are suboptimal for our problem. The techniques in \cite[Appendix~D]{QianRakhlinZhivotovskiy2026} immediately give $Cd\log(\e n/d)$ for some absolute constant $C$. We do not reproduce that argument, since item~(i) above gives the explicit upper bound $d\log(\e n/d)$, and hence the asymptotically sharp coefficient $1$ in front of $d\log(n/d)$ when $n/d\to\infty$.

Item~(ii) is of a different nature: under the balanced block structure, the Shtarkov sum factorizes into \(d\) one-dimensional logistic Shtarkov sums, each of which is bounded below by the Bernoulli empirical Kullback--Leibler Shtarkov constant. When the absolute multipliers are constant within each block, these factors are exactly Bernoulli empirical Kullback--Leibler problems. This matches, after the identification \(d=m-1\), the large alphabet redundancy asymptotics of Orlitsky and Santhanam~\cite{orlitsky2004speaking}, and for \(d=1\) reduces to the classical Bernoulli phenomenon studied in \cite{maurer2004note, agrawal2020finite, guo2020chernoff}. Item~(ii) thus captures the familiar local quadratic behavior within fixed-design logistic regression.

Item~(iii) shows that \(\tfrac{d}{2}\log(n/d)\) is not the full worst-case picture: at \(\lambda=1\), the Shtarkov sum detects a genuinely global combinatorial complexity of the design, invisible in local quadratic or large alphabet redundancy regimes. Item~(iv) is included for completeness; it appears essentially in \cite{DrmotaJacquetWuSzpankowski2024} and, while weaker than item~(iii), provides a useful general reference point.

\begin{proof}
We begin with part~(i). For $\theta \in \R^d$ and $\eps \in \{-1,1\}^n$, let
\[
p_\theta(\eps)
\defn
\prod_{i=1}^n \sigma(\eps_i \langle x_i,\theta\rangle).
\]
Let $\delta_1,\dots,\delta_n$ be \iid\ real-valued random variables with cumulative distribution function $\sigma$. Then, for every $\theta \in \R^d$ and every $\eps \in \{-1,1\}^n$,
\[
p_\theta(\eps)
=
\P\bigl(\sign(\langle x_i,\theta\rangle-\delta_i)=\eps_i \ \mbox{for all } i\in[n]\bigr).
\]
Therefore,
\[
\mathcal{S}_d^{\operatorname{log}}(x_1,\dots,x_n)
=
\sum_{\eps \in \{-1,1\}^n}\sup_{\theta\in\R^d} p_\theta(\eps)
\le
\E N_X(\delta),
\]
where
\[
N_X(\delta)
\defn
\left|
\Bigl\{
S_{X,\delta}(\theta): \theta\in\R^d,\ \langle x_i,\theta\rangle\neq \delta_i \ \mbox{for all } i\in[n]
\Bigr\}
\right|,
\]
and
\[
S_{X,\delta}(\theta)
\defn
\bigl(
\sign(\langle x_1,\theta\rangle-\delta_1),\dots,\sign(\langle x_n,\theta\rangle-\delta_n)
\bigr).
\]
Indeed, for each $\eps$,
\[
\sup_{\theta\in\R^d} p_\theta(\eps)
\le
\P\Bigl(\eps \in \{S_{X,\delta}(\theta):\theta\in\R^d,\ \langle x_i,\theta\rangle\neq \delta_i \ \mbox{for all } i\in[n]\}\Bigr),
\]
and summing over $\eps$ gives the expected number of realized sign patterns. Fix now $\delta=(\delta_1,\dots,\delta_n)$. On the event that $\delta_i\neq 0$ whenever $x_i=0$, the $i$-th coordinate of $S_{X,\delta}(\theta)$ is independent of $\theta$, so removing such coordinates does not change $N_X(\delta)$. Since this event holds almost surely, for the purpose of the upper bound, it is enough to consider the nonzero covariates. For each such $x_i$, consider the affine hyperplane
\[
H_i(\delta)\defn \{\theta\in\R^d:\langle x_i,\theta\rangle=\delta_i\}.
\]
On each connected component of
$\R^d \setminus \bigcup_{i:\,x_i\neq 0} H_i(\delta),
$
the sign pattern $S_{X,\delta}(\theta)$ is constant. Hence $N_X(\delta)$ is bounded by the number of regions in the arrangement generated by the hyperplanes $H_i(\delta)$ with $x_i\neq 0$. It is standard that the number of regions in any arrangement of at most $n$ affine hyperplanes in $\R^d$ is at most
$
\sum_{\ell=0}^{d}\binom{n}{\ell}.
$
For completeness, $N_X$ is measurable without any general position assumption. Indeed, for each $\eps\in\{-1,1\}^n$, the set
\[
D_\eps=\{\delta\in\R^n:\text{there exists }\theta\in\R^d
\text{ such that }\eps_i(\langle x_i,\theta\rangle-\delta_i)>0
\text{ for all }i\in[n]\}
\]
is open, and $N_X(\delta)=\sum_\eps\1\{\delta\in D_\eps\}$. Thus we may take expectations, yielding
\[
\mathcal{S}_d^{\operatorname{log}}(x_1,\dots,x_n)
\le
\sum_{\ell=0}^{d}\binom{n}{\ell},
\]
which proves part~(i).

We now turn to part~(ii). Set $a_i \defn |x_i|$. Note that the Shtarkov sum only depends on $a_1,\dots,a_n$. Thus, it is enough to consider the case where $a_1,\dots,a_n>0$. For $A \subset [n]$, define
\[
q_A(\theta)
\defn
\prod_{i\in A}\sigma(a_i\theta)\prod_{i\notin A}\sigma(-a_i\theta).
\]
Then
\[
\mathcal{S}_1^{\operatorname{log}}(x_1,\dots,x_n)
=
\sum_{A\subset[n]} \sup_{\theta\in\R} q_A(\theta).
\]
For $k\in\{0,\dots,n\}$, let $
\mathcal A_k \defn \{A\subset[n]: |A|=k\}.
$
Since the sum of suprema dominates the supremum of the sum,
\[
\mathcal{S}_1^{\operatorname{log}}(x_1,\dots,x_n)
\ge
\sum_{k=0}^n
\sup_{\theta\in\R}
\sum_{A\in \mathcal A_k} q_A(\theta).
\]
Fix $\theta\in\R$, and let $B_1(\theta),\dots,B_n(\theta)$ be independent Bernoulli random variables with
\[
\P(B_i(\theta)=1)=\sigma(a_i\theta),
\qquad i\in[n].
\]
If
$
K_\theta \defn \sum_{i=1}^n B_i(\theta),
$
then
$
\sum_{A\in \mathcal A_k} q_A(\theta)=\P(K_\theta=k),
$
and therefore
\[
\mathcal{S}_1^{\operatorname{log}}(x_1,\dots,x_n)
\ge
\sum_{k=0}^n \sup_{\theta\in\R}\P(K_\theta=k).
\]
For $k=0$ and $k=n$, the corresponding suprema are equal to $1$. Now fix $k\in\{1,\dots,n-1\}$. The map
$
\theta \mapsto \E[K_\theta]
=
\sum_{i=1}^n \sigma(a_i\theta)
$
is continuous and strictly increasing from $0$ to $n$, so there exists $\theta_k\in\R$ such that
\[
\E[K_{\theta_k}] = k.
\]
The law of $K_{\theta_k}$ is Poisson-binomial with mean $k$. By Hoeffding's extremal theorem for Poisson-binomial laws~\cite[Theorem~4]{Hoeffding1956}, if
$
\overline K_k \sim \mathrm{Bin}(n,k/n),
$
then
\[
\P(K_{\theta_k}\leq k-1)\leq \P(\overline K_k\leq k-1),
\quad\textrm{and} \quad
\P(K_{\theta_k}\leq k)\geq \P(\overline K_k\leq k).
\]
Subtracting these inequalities gives
$
\P(K_{\theta_k}=k)
\ge
\P(\overline K_k = k)
=
\binom{n}{k}\Bigl(\frac{k}{n}\Bigr)^k\Bigl(1-\frac{k}{n}\Bigr)^{n-k}.
$
Hence
\[
\mathcal{S}_1^{\operatorname{log}}(x_1,\dots,x_n)
\ge
\sum_{k=0}^n
\binom{n}{k}\Bigl(\frac{k}{n}\Bigr)^k\Bigl(1-\frac{k}{n}\Bigr)^{n-k}.
\]
Here and below, the endpoint terms $k=0,n$ are interpreted as $1$.
By \cite[Proof of Theorem~1]{maurer2004note}, the quantity on the right is exactly the Bernoulli Shtarkov constant. That is, for every $p\in(0, 1)$,
\[
\sum_{k=0}^n
\binom{n}{k}\Bigl(\frac{k}{n}\Bigr)^k\Bigl(1-\frac{k}{n}\Bigr)^{n-k}
=
\E_p \exp\big(n\kl{\hat p}{p}\big),
\qquad
\hat p=\frac1n\sum_{i=1}^n X_i \, ,
\qquad
X_1,\dots,X_n \stackrel{\mathrm{iid}}{\sim}\Ber{p},
\]
and \cite[Theorem~1]{maurer2004note} proved that for $n\geq 2$ this quantity is at least $\sqrt n$. 

For the second statement in part~(ii), let \(m \defn n/d\), and after relabeling assume that for $j\in[d], r\in[m]$,
\[
x_{(j-1)m+r}=b_{j,r}u_j,
\]
for some nonzero scalars \(b_{j,r}\) and linearly independent vectors \(u_1,\dots,u_d \in \R^d\).
Since \(u_1,\dots,u_d\) form a basis of \(\R^d\), there exists a basis $(\phi_1, \dots, \phi_d)$ such that \(\langle u_j,\phi_k\rangle=\delta_{jk}\). Hence, for every \(t=(t_1,\dots,t_d)\in\R^d\), the choice \(\theta=\sum_{j=1}^d t_j \phi_j\) satisfies \(\langle u_j,\theta\rangle=t_j\) for all \(j\). Writing \(\eps=(\eps^{(1)},\dots,\eps^{(d)})\) with \(\eps^{(j)}\in\{-1,1\}^m\), we therefore obtain
\[
\sup_{\theta\in\R^d}\prod_{i=1}^n \sigma(\eps_i\langle x_i,\theta\rangle)
=
\sup_{t\in\R^d}\prod_{j=1}^d\prod_{r=1}^m \sigma(\eps_r^{(j)} b_{j,r} t_j)
=
\prod_{j=1}^d \sup_{s\in\R}\prod_{r=1}^m \sigma(\eps_r^{(j)} b_{j,r} s),
\]
since the factors are nonnegative and depend on disjoint coordinates. Summing over \(\eps\) now factorizes blockwise, and yields
\[
\mathcal S_d^{\operatorname{log}}(x_1,\dots,x_n)
=
\prod_{j=1}^d \mathcal S_1^{\operatorname{log}}(b_{j,1},\dots,b_{j,m}).
\]
Applying the one-dimensional lower bound already proved to each factor, with the case $m=1$ being immediate since $\mathcal S_1^{\operatorname{log}}(b)=2\geq\sqrt m$, we get
\[
\mathcal S_d^{\operatorname{log}}(x_1,\dots,x_n)
\ge
(\sqrt m)^d
=
\Bigl(\frac nd\Bigr)^{d/2}.
\]
This proves part~(ii).

We now focus on part~(iii). We first prove the one-dimensional extremal fact
\begin{equation}
\label{eq:d1-extremal-shtarkov}
\sup_{a_1,\dots,a_m>0}\mathcal S_1^{\operatorname{log}}(a_1,\dots,a_m)=m+1.
\end{equation}
The upper bound follows from part~(i) with $d=1$. We now prove the matching lower bound. Fix $m\geq 1$ and consider the  design
$
a_r=R^{r-1},
$
for all $r\in [m]$, 
where $R>1$ will tend to infinity. For $\eps\in\{-1,1\}^m$, define for $t \in \R$,
\[
q_\eps(t)\defn \prod_{r=1}^m \sigma(\eps_r a_r t).
\]
The two constant sign patterns $(1,\dots,1)$ and $(-1,\dots,-1)$ contribute exactly $1$, since
\[
\sup_{t\in\R} q_{(1,\dots,1)}(t)=1
\qquad\text{and}\qquad
\sup_{t\in\R} q_{(-1,\dots,-1)}(t)=1.
\]
We next partition the remaining sign patterns. For each $j\in[m-1]$, let
\[
\mathcal C_j^+
\defn
\{\eps\in\{-1,1\}^m:\ \eps_j=1,\ \eps_r=-1 \ \mbox{for all } r>j\},
\]
\[
\mathcal C_j^-
\defn
\{\eps\in\{-1,1\}^m:\ \eps_j=-1,\ \eps_r=1 \ \mbox{for all } r>j\},
\]
with the coordinates $\eps_1,\dots,\eps_{j-1}$ arbitrary. These families, together with the two constant patterns, form a partition of $\{-1,1\}^m$. Indeed, for a non-constant $\eps$, if $\eps_m=-1$, let $j$ be the largest index such that $\eps_j=1$. Then by construction, $\eps_r=-1$ for all $r>j$, so $\eps\in\mathcal C_j^+$ with $j\leq m-1$. The same argument applies by symmetry with $\mathcal C_j^-$ in the case where $\eps_m=1$. In each case $j$ is uniquely determined by $\eps$. Now summing over the free coordinates gives
\[
\sum_{\eps\in\mathcal C_j^+} q_\eps(t)
=
\sigma(a_j t)\prod_{r>j}\sigma(-a_r t),
\qquad
\sum_{\eps\in\mathcal C_j^-} q_\eps(t)
=
\sigma(-a_j t)\prod_{r>j}\sigma(a_r t),
\]
because for each $r<j$, it holds that 
$
\sum_{\eps_r\in\{-1,1\}}\sigma(\eps_r a_r t)=\sigma(a_r t)+\sigma(-a_r t)=1.
$
We now define two particular levels 
$
t_j^- \defn -\frac{R^{1/2}}{a_{j+1}}$, and $
t_j^+ \defn \frac{R^{1/2}}{a_{j+1}}.
$
Since $a_r=R^{r-1}$, we can immediately verify that
\[
\lim_{R \to \infty}\sum_{\eps\in\mathcal C_j^+} q_\eps(t_j^-)= \frac12,
\qquad
\lim_{R \to \infty}\sum_{\eps\in\mathcal C_j^-} q_\eps(t_j^+)= \frac12.
\]
Finally, we have
\[
\mathcal S_1^{\operatorname{log}}(1,R,\dots,R^{m-1})
=
\sum_{\eps\in\{-1,1\}^m} \sup_{t\in\R} q_\eps(t) 
\ge
1+1+\sum_{j=1}^{m-1}
\left(
\sup_{t\in\R}\sum_{\eps\in\mathcal C_j^+} q_\eps(t)
+
\sup_{t\in\R}\sum_{\eps\in\mathcal C_j^-} q_\eps(t)
\right).\]
Therefore, it holds that
\[
\liminf_{R\to\infty}\mathcal S_1^{\operatorname{log}}(1,R,\dots,R^{m-1})
\ge
1+1+\sum_{j=1}^{m-1}\left(\frac12+\frac12\right)
=
m+1.
\]
Combined with the upper bound, this proves~\eqref{eq:d1-extremal-shtarkov}.

We now tensorize this dimension one construction. Let $m\defn n/d$, which is an integer by assumption, and fix any positive scalars $a_1,\dots,a_m$. Define a design in $\R^d$ by placing one copy of the one-dimensional design on each coordinate axis for $j \in [d], r \in [m]$:
\[
x'_{(j-1)m+r}\defn a_r e_j,
\]
where $e_1,\dots,e_d$ are the canonical basis vectors in $\R^d$. We can write  $\eps\in\{-1,1\}^n$ as $d$ consecutive blocks
\[
\eps=(\eps^{(1)},\dots,\eps^{(d)}),
\quad \textrm{where} \quad
\eps^{(j)}\in\{-1,1\}^m,
\]
Since the factors depend on disjoint coordinates of $\theta$ and are nonnegative, the supremum factorizes as follows:
\[
\sup_{\theta\in\R^d}\prod_{i=1}^n \sigma(\eps_i \langle x'_i,\theta\rangle)
=
\prod_{j=1}^d
\sup_{t\in\R}\prod_{r=1}^m \sigma\bigl(\eps_r^{(j)} a_r t\bigr).
\]
Therefore, it holds that
\[
\sup_{x'_1,\dots,x'_n\in\R^d}\mathcal S_d^{\operatorname{log}}(x'_1,\dots,x'_n)
\ge
\sup_{a_1,\dots,a_m>0}
\Bigl(\mathcal S_1^{\operatorname{log}}(a_1,\dots,a_m)\Bigr)^d
=
(m+1)^d
=
\left(\frac nd+1\right)^d.
\]
This proves part~(iii).

Finally, we prove part~(iv). The argument appears in \cite{DrmotaJacquetWuSzpankowski2024} and we present it for the sake of completeness. Assume that $x_1,\dots,x_n$ are in general position in the sense that every subset of size at most $d$ is linearly independent. If a sign pattern $\eps\in\{-1,1\}^n$ is realizable by some $\theta\in\R^d$, that is,
$
\eps_i \langle x_i,\theta\rangle >0,$  for all $i \in [n]$,
then
$
\sup_{t>0}\prod_{i=1}^n \sigma\bigl(t\,\eps_i\langle x_i,\theta\rangle\bigr)=1.
$
Hence every realizable pattern contributes $1$ to the Shtarkov sum. Consequently,
$
\mathcal{S}_d^{\operatorname{log}}(x_1,\dots,x_n)
$
is bounded from below by the number of dichotomies of the points $x_1,\dots,x_n$ induced by homogeneous halfspaces in $\R^d$. By Cover's formula~\cite{Cover1965}, under this assumption this number is
$
2\sum_{\ell=0}^{d-1}\binom{n-1}{\ell}.
$
This proves part~(iv).
\end{proof}

\subsection{The bivariate case}
\label{sec:bivariatecase}

Our Vandermonde subspaces construction that exhibits a logarithmic overhead only works in dimension at least $3$. In contrast, for a scalar model, the LLR obeys uniform $\chi_1^2$-type mean and tail bounds. The following result shows that this chi-square-type behavior already collapses in dimension $2$. 

Recall the subspace reformulation~\eqref{eq:subspace-reformulation}, that defines for a subspace $W \subset \R^n$, a vector $v \in \R^n$ and $\eps \in \{-1,1\}^n$, the quantity
\begin{equation}
\Gamma_W(\eps; v)
	= \sup_{w \in W} \log \frac{p_w(\eps)}{p_v(\eps)},
		\quad \mbox{where} \quad
	p_v(\eps) = \prod_{i=1}^n \sigma(\eps_i v_i).
\end{equation}

We also define the shorthand
\begin{equation}
\label{eq:expect-short}
E(v, W) = \E_{\eps \sim p_v} \Gamma_W(\eps; v) \, .
\end{equation}

\begin{proposition}
\label{prop:d2-logloglog-lower}
Let $n \geq \ceil{\exp(\exp(\exp(20)))}$. There exists a two-dimensional subspace $W \subset \R^n$ and a vector $v \in W$ for which
\[
E(v,W) \geq \frac{1}{16 \e^2} \, \log \log \log n.
\]
\end{proposition}

For a vector \(z \in \R^m\), a sign vector \(\eps \in \{-1,1\}^m\), and a subset $A \subset \{-1, 1\}^m$, we put
\[
p_z(\eps) \defn \prod_{i=1}^m \sigma(\eps_i z_i),
\qquad
p_z(A) \defn \sum_{\eps \in A} p_z(\eps).
\]
Recall that
\[
E(v,W) = \sum_{\eps \in \{-1,1\}^m} p_v(\eps) \log \sup_{w \in W} \frac{p_w(\eps)}{p_v(\eps)}.
\]
The crux of the proof is the following lemma. 
\begin{lemma}
\label{lem:a-single-block}
Fix an integer $k \geq 9$. Define 
\[
m_k \defn \lfloor \e^k \rfloor,
\qquad
N_{j} \defn \left\lfloor \exp\big(k(4^j-1)\big)\right\rfloor,
\quad j=1,\dots,m_k, 
\qquad \mbox{and} \quad 
M_k \defn \sum_{j=1}^{m_k} N_{j}.
\]
There is a two-dimensional subspace $U \subset \R^{M_k}$ and a vector $u \in U$ such that
\begin{equation}
\label{eq:Ek-lower-claim}
E(u,U) \geq \frac{1}{8\e^2}(k-8).
\end{equation}
\end{lemma}

\begin{proof}[Proof of~\Cref{prop:d2-logloglog-lower}]
We now ``pad'' to obtain the result for general $n$. Put $k = \floor{L} - 1$ where $L = \log \log \log n$. We claim 
\begin{equation}
\label{ineq:strict-ineq}
M_k < n \quad \mbox{for}~n \geq \ceil{\exp(\exp(\exp(20)))}.
\end{equation}
Define
\[
m = n- M_k, \quad W = U \times \{0_m\},
\quad \text{and} \quad
v = (u, 0_{m}).
\]
For every $(\eps,\delta)\in \{-1,1\}^{M_k}\times \{-1,1\}^{m}$ and any $w \in U$, it holds that
\[
p_{v}(\eps,\delta)
=
2^{-m}p_u(\eps)
\quad \mbox{and} \quad 
p_{(w,0_m)}(\eps,\delta)
=
2^{-m}p_w(\eps).
\]
Hence
\[
\sup_{ w\in W}
\frac{p_{ w}(\eps,\delta)}{p_{v}(\eps,\delta)}
=
\sup_{w\in U}
\frac{p_w(\eps)}{p_u(\eps)}.
\]
Summing over \(\delta\) and using~\cref{eq:Ek-lower-claim} yields
\[
E(v,W)=E(u,U)
\geq
\frac{1}{8\e^2} (k-8) \geq 
\frac{1}{8\e^2}(L - 10) 
\geq 
\frac{1}{16 \e^2} \cdot \log \log \log n,
\]
as required. We used that $k \geq L -2$ and $L - 10 \geq \tfrac{L}{2}$ as $L \geq 20$.

\paragraph{Proof of~\cref{ineq:strict-ineq}:}
Note that we have:
\[
M_k
\leq
\sum_{j=1}^{m_k} \exp\big(k(4^j-1)\big)
\leq
m_k \exp\big(k(4^{m_k}-1)\big)
\leq
\exp\big(k4^{\e^k}\big).
\]
Since $k \leq \e^k$ and $1+\log 4 < \e$,
\[
\log\log M_k \leq \log k + \e^k \log 4 \leq (1+\log 4)\e^k < \e^{k+1}
\]
whence,
\[
\log \log \log M_k < k + 1 \leq \log \log \log n,
\]
which implies in particular $M_k < n$.

\end{proof}

\begin{proof}[Proof of~\Cref{lem:a-single-block}]
Partition $[M_k]$ into consecutive blocks:
\[
B_{1},\dots,B_{m_k} \qquad\text{with}\qquad |B_j|=N_j.
\]
We define $x, y \in \R^{M_k}$ according to
\[
x_i = 2^j,
\quad \mbox{and} \quad 
y_i =4^j,
\quad \mbox{for}~i \in B_{j}.
\]
We set $U = \Span\{x, y\}$; clearly $\dim(U) = 2$. We set $u = k y$.
We consider the sets 
\[
A_j = \big\{\eps \in \{-1, 1\}^{M_k} \mid
\eps = \1_{M_k} - 2e_i \text{ for some } i \in B_j\big\}
\quad \mbox{and} \quad 
\cA = \bigcup_{j=1}^{m_k} A_j.
\]
Clearly, $\{A_j\}_{j \in [m_k]}$ is a pairwise disjoint collection of sets.
Suppose we can show that 
\begin{equation}
\label{ineq:required-lower-bounds}
p_u(A_j) \stackrel{{\rm (i)}}{\geq} \frac{1}{4\e^2} \e^{-k}, 
\quad \mbox{and} \quad 
\min_{\eps \in A_j} 
\sup_{v \in U}
\log \frac{p_v(\eps)}{p_u(\eps)} \stackrel{{\rm (ii)}}{\geq} k - 8, \quad \mbox{for all}~j\in[m_k].
\end{equation}
Then~\cref{ineq:required-lower-bounds}(i) and disjointness give $p_u(\cA) \geq \frac{1}{4 \e^2} \e^{-k} m_k
\geq \tfrac{1}{8\e^2}$, as $m_k = \floor{\e^k} \geq \tfrac{\e^k}{2}$.
Combining this with~\cref{ineq:required-lower-bounds}(ii), we obtain
\[
E(u,U)
	\geq \sum_{\eps \in \cA} p_u(\eps) \, \sup_{v \in U} \log \frac{p_v(\eps)}{p_u(\eps)} 
	\geq p_u(\cA) \, (k-8)
	\geq \frac{1}{8\e^2}(k-8),
\]
thereby completing the proof. 

\paragraph{Proof of~\cref{ineq:required-lower-bounds}(i):}

Set
\[
q_\ell \defn \sigma(-k 4^\ell),
\qquad
\rho_\ell \defn \sigma(k 4^\ell) \equiv 1-q_\ell, \quad \mbox{for}~\ell \in [m_k].
\]
Note that because $u$ is constant in the blocks, we have 
\begin{equation}
\label{eqn:expression-for-block-probs}
p_u(A_j) = N_j \cdot \frac{q_j}{\rho_j} \cdot \prod_{\ell = 1}^{m_k} \rho_\ell^{N_\ell}
	\quad \mbox{for any}~j \in [m_k].
\end{equation}
For $t \geq 0$, we have $\tfrac{\sigma(-t)}{\e^{-t}} \in [1/2, 1]$. 
Hence, for every \(\ell\),
\begin{equation}
\label{ineq:upper-on-product}
N_{\ell}q_\ell \leq \e^{k(4^\ell - 1) - k 4^\ell} = \e^{-k}
\end{equation}
Similarly, 
\begin{equation}
\label{ineq:lower-on-product}
N_{\ell}q_\ell \geq \half \floor{\e^{k(4^\ell - 1)}} \e^{-k 4^\ell} \geq \frac{1}{4} \e^{-k}.
\end{equation}
Here, we used $k(4^\ell-1)\geq 3k >\log 2$, so that $\floor{\e^{k(4^\ell - 1)}} \geq \tfrac{1}{2} \e^{k(4^\ell - 1)}$.
It follows that
\begin{equation}
\label{ineq:for-product-of-r}
\prod_{\ell=1}^{m_k} \rho_\ell^{N_{\ell}}
	= \exp\Big(\sum_{\ell=1}^{m_k} N_{\ell}\log(1-q_\ell)\Big)
	\stackrel{{\rm (i)}}{\geq} \exp\Big(-2\sum_{\ell=1}^{m_k} N_{\ell}q_\ell\Big)
	\stackrel{{\rm(ii)}}{\geq} \frac{1}{\e^2}.
\end{equation}
Above, inequality (i) used $q_\ell \leq \e^{-k 4^\ell} \leq \e^{-4k} < 1/2$, so that $\log(1-q_\ell) \geq -2q_\ell$. 
Additionally, inequality (ii) used $\sum_{\ell=1}^{m_k} N_\ell q_\ell \leq \e^{-k} m_k \leq 1$, by~\cref{ineq:upper-on-product}. 
As $\rho_j = \sigma(k 4^j) \leq 1$, the claim follows from~\cref{eqn:expression-for-block-probs,ineq:lower-on-product,ineq:for-product-of-r}. 

\paragraph{Proof of~\cref{ineq:required-lower-bounds}(ii):}
We consider the collection of vectors 
\[
w^{(j)}
\defn
u - \frac{2k}{2^j} x + \frac{k}{4^j} y, 
\quad \mbox{for}~j \in [m_k].
\]
Clearly, $w^{(j)} \in U$ for each $j$. 
For $\ell \in [m_k]$, we have for $i \in B_\ell$ that 
\[
w^{(j)}_i
=
k4^\ell - 2k\,2^{\ell-j} + k4^{\ell-j}
=
k(4^\ell - 1) + k(2^{\ell-j}-1)^2.
\]
In particular, if $i \in B_j$, then 
$w^{(j)}_i = k(4^j-1)$.
We have $\log N_{\ell} \leq k(4^\ell-1)$ by our choice of block sizes. It implies that: 
\begin{equation}
\label{eq:b-lower-bound}
w^{(j)}_i
\geq
\log N_{\ell} + k(2^{\ell-j}-1)^2,
\end{equation}
which holds for every $i \in B_\ell$ and for every $j, \ell \in [m_k]$.

For $j \in [m_k]$ and $\eps \in A_j$, we have 
\begin{align}
\sup_{v \in U} \log \frac{p_v(\eps)}{p_u(\eps)} &\geq \log \frac{p_{w^{(j)}}(\eps)}{p_u(\eps)}\\ 
&= 
\Big( 
\log \frac{\sigma(-k(4^j -1))}{\sigma(-k4^j)} - \log \frac{\sigma(k(4^j -1))}{\sigma(k4^j)}\Big) +  
\sum_{\ell = 1}^{m_k}
N_\ell \log \frac{\sigma(k(4^\ell -1) + k(2^{\ell - j} - 1)^2)}{\sigma(k 4^\ell)} \\ 
&\stackrel{{\rm(a)}}{=} k + N_j \log \frac{\sigma(k(4^j -1))}{\sigma(k 4^j)} + 
\sum_{\substack{\ell \neq j \\ \ell \in [m_k]}} 
N_\ell \log \frac{\sigma(k(4^\ell -1) + k(2^{\ell - j} - 1)^2)}{\sigma(k 4^\ell)} \\ 
&\stackrel{{\rm(b)}}{\geq}
k - 1
- \underbrace{\sum_{\substack{\ell \neq j \\ \ell \in [m_k]}} 
N_\ell \log \frac{\sigma(k 4^\ell)}{\sigma(k(4^\ell -1) + k(2^{\ell - j} - 1)^2)}}_{S_{j}} = k - 1 - S_{j}.
\label{ineq:in-terms-of-sum}
\end{align}
Above, relation (a) used $\log \tfrac{\sigma(t)}{\sigma(-t)} = t$ and inequality (b) used $\log(\sigma(t)) = -\log(1+\e^{-t}) \geq -\e^{-t}$ 
so that \[
N_j \log \frac{\sigma(k(4^j -1))}{\sigma(k 4^j)} \geq -N_j \e^{-k(4^j - 1)} \geq -1.
\]
We now control the sum $S_j$. 
First, note that by~\eqref{eq:b-lower-bound}, we have for $\ell \neq j$, 
\begin{equation}\label{ineq:lower-bound-term}
N_\ell \log \frac{\sigma(k(4^\ell -1) + k(2^{\ell - j} - 1)^2)}{\sigma(k 4^\ell)} \geq -N_\ell \e^{-k(4^\ell - 1) - k(2^{\ell - j} - 1)^2} 
\geq -\exp(-k(2^{\ell - j} - 1)^2). 
\end{equation}
We can write 
\begin{equation}
S_j = \sum_{\substack{\ell < j\\ \ell \in [m_k]}} N_\ell \log \frac{\sigma(k 4^\ell)}{\sigma(k(4^\ell -1) + k(2^{\ell - j} - 1)^2)} + 
\sum_{\substack{\ell > j\\ \ell \in [m_k]}} N_\ell \log \frac{\sigma(k 4^\ell)}{\sigma(k(4^\ell -1) + k(2^{\ell - j} - 1)^2)}
= S_{j}^{<} + S_{j}^{>}.
\end{equation}
From the inequality~\eqref{ineq:lower-bound-term}, we have 
\begin{equation}
\label{ineq:lower-bound-S_j_big}
S_{j}^{>} \leq \sum_{\substack{\ell > j\\ \ell \in [m_k]}}
\exp(-k(2^{\ell - j} - 1)^2) \leq \sum_{\substack{\ell > j\\ \ell \in [m_k]}} \e^{-k} \leq \e^{-k} m_k \leq 1. 
\end{equation}
On the other hand, for $m = \floor{\log_2(2k)}$, we have 
\begin{multline}
\label{ineq:S_j-small}
S_j^{<} \leq \sum_{\substack{\ell < j\\ \ell \in [m_k]}}
\exp(-k(2^{\ell - j} - 1)^2)
\leq \sum_{n = 1}^{m} 
\exp(-k(2^{-n} - 1)^2)
+ \sum_{n=m+1}^{m_k}\exp(-k(2^{-n} - 1)^2)
\\
\leq 2k \e^{-k/4} + \e^k\e^{1-k} \leq \frac{8}{\e} + \e \leq 6.
\end{multline}
Combining inequalities~\cref{ineq:in-terms-of-sum,ineq:lower-bound-S_j_big,ineq:S_j-small}, we obtain the desired result.

\end{proof}

The corresponding worst case quantile lower bound is recorded in
\Cref{cor:d2-logloglog-quantile-lower}. Together with
\Cref{thm:d2-logloglog-upper}, it gives the sharp worst case fixed design
quantile bound in dimension two.

\subsection{The \(d^{3/2}/n\) boundary for Wilks approximation}
\label{sec:wilksregime}

The classical Wilks phenomenon arises in the purely asymptotic regime where the dimension $d$ and the parameter $\thetastar \in \R^{d}$ are both fixed. Then, as the sample size $n$ goes to infinity, (i) the probability that the MLE exists (and is unique) goes to $1$, and (ii) the LLR statistic $2 \Lambda_n^{\log}(\thetastar)$ converges in distribution to a chi-square with $d$ degrees of freedom.
Over the last decade, the proportional high-dimensional asymptotic regime has attracted a lot of attention. There, one studies a sequence of well-specified logistic models with parameter $\thetastar = \thetastar_n \in \R^{d_n}$, and the dimension $d_n$ grows with the sample size $n$.
In the series of works~\cite{sur2019modern,sur2019likelihood}, the authors study the case where the dimension $d_n$ grows linearly with $n$, in a way that $d_n /n \to \kappa \in (0,1)$ as $n \to \infty$.
The limiting aspect ratio $\kappa$ thus emerged as a key parameter in the analysis of high-dimensional, random-design logistic regression. Notably, in the ``null case'' (pure noise, $\thetastar = 0$), the Wendel--Cover formula~\cite{Wendel1962,Cover1965} shows that the critical threshold sits exactly at $\kappa = 1/2$. Below this value, the MLE exists with probability going to 1, whereas if $\kappa >1/2$, the probability that the MLE exists goes to $0$. At $\kappa=1/2$, the aspect ratio alone does not determine the limiting probability; under the centered scaling $d_n-n/2=o(\sqrt n)$, the probability that the MLE exists converges to $1/2$.
However, it turns out that the limiting aspect ratio $\kappa$ does not determine the behavior of the LLR statistic. In particular, whether or not its distribution is well-approximated by the chi-square distribution depends on the ratio $d^{3/2}/n$ instead of $d/n$.

In the rest of this section, we consider the case of random design logistic regression, where the labels are conditionally independent given the design and
\begin{equation}
\label{eq:logistic-regression-model-gaussian}
X_i \simiid \Normal{0}{I_d}, \quad
\P(Y_i=1\mid X_i)=\sigma(X_i^\T \theta^{(d)}),
\quad \theta^{(d)} = 0 \in \R^d.
\end{equation}
Throughout this section, we view $d = d(n)$ as a positive integer-valued
function of $n$; however, we write $d$ for simplicity. 
Recall that the log-likelihood ratio statistic is then given by 
\[
\Lambda_{n,d} = \sup_{\eta \in \R^d} \sum_{i=1}^n \log\big(2 \sigma(Y_i X_i^\T \eta)\big).
\]
We recall the Kolmogorov-Smirnov (KS) distance between real-valued random variables $U, V$,
\[
\KSDist{U}{V} = \sup_{x \in \R} | \P(U \leq x) - \P(V \leq x) |.
\]
The next result characterizes when the log-likelihood ratio (LLR) statistic $\Lambda_{n, d}$ 
is asymptotically well-approximated by $\tfrac{1}{2} \chi^2_d$. 
\begin{proposition}[Approximation of the LLR statistic via $\chi^2_d$ under Gaussian random design]
\label{prop:chi-limit-behavior}
Under the observational model given by~\cref{eq:logistic-regression-model-gaussian}, the following holds: 
\begin{enumerate}[label=(\roman*)]
\item \label{item:chi-limit-still-works}
if $d^{3/2}/n \to 0$, then $\lim_{n \to \infty}\KSDist{2\Lambda_{n,d}}{\chi^2_d} =0$; and 
\item \label{item:chi-limit-fails}
if $d^{3/2}/n \nrightarrow 0$, then $\limsup_{n \to \infty}\KSDist{2\Lambda_{n,d}}{\chi^2_d} > 0$.
\end{enumerate} 
\end{proposition}

To understand~\Cref{prop:chi-limit-behavior},
we first point out that it demonstrates that the 
aspect ratio $d/n$ does not suffice to precisely 
characterize the asymptotic behavior of the LLR statistic. 
More specifically, consider the case that $d = \floor{n^{2/3}}$. 
Although $d/n \to 0$ as $n \to \infty$, Proposition~\ref{prop:chi-limit-behavior}\ref{item:chi-limit-fails} implies that $\limsup_{n\to\infty}\KSDist{2\Lambda_{n,d}}{\chi^2_d} > 0$; stated otherwise, the LLR statistic is not asymptotically well-approximated by $\tfrac{1}{2}\chi^2_d$ in Kolmogorov-Smirnov distance.
On the other hand, consider the case that $d = \floor{n^{1/3}}$. 
In this case, it is true that $d/n \to 0$ as $n \to \infty$, but
Proposition~\ref{prop:chi-limit-behavior}\ref{item:chi-limit-still-works} implies that
$\KSDist{2\Lambda_{n,d}}{\chi^2_d} \to 0$ as $n \to \infty$; 
stated otherwise, the LLR statistic is asymptotically well-approximated by $\tfrac{1}{2}\chi^2_d$ in Kolmogorov-Smirnov distance.

\subsection{Proofs}
\label{sec:proofs-sharpness}
\subsubsection{Proof of~\Cref{prop:chi-limit-behavior}}
\label{sec:proof-prop-chi-limit-behavior}

Let $X\in \R^{n\times d}$ be the design matrix with rows $X_i^\T$, and set
\[
B_i \defn \frac{1+Y_i}{2}\in\{0,1\},
\qquad
Z_i \defn Y_i X_i,
\qquad
i\in[n].
\]
Let $B=(B_1,\ldots,B_n)$. Then $Z_i \simiid \Normal{0}{I_d}$, as $Y_i$ are independent Rademachers. Let $Z\in \R^{n\times d}$ be the matrix with rows $Z_i^\T$, and let $\cV$ denote the range 
of $Z$. 
We define for $|\delta| \leq 1$, 
\[
h(\delta) = \kl{\Ber{\tfrac{1 + \delta}{2}}}{\Ber{1/2}} = \half \Big( (1+\delta) \log(1+\delta) + (1-\delta) \log(1-\delta)\Big).
\]
Recall that by the variational characterization of the log-likelihood ratio statistic provided in Lemma~\ref{lem:var-char}, we have that 
\begin{equation}\label{eqn:var-char-llr}
\Lambda_{n,d}
=
\inf_{p \in [0, 1]^n,\ X^\T p = X^\T B}
\sum_{i=1}^n \kl{\Ber{p_i}}{\Ber{1/2}}
= 
\inf_{\delta \in [-1, 1]^n,\ \delta + \mathbf{1}_n \in \cV^\perp}
\sum_{i=1}^n h(\delta_i)
\end{equation}
where in the second infimum we let $\delta_i \defn Y_i(1-2p_i)$. To see why both formulations are equivalent, note first that, on the one hand, since $h$ is even, it holds that
$\kl{\Ber{p_i}}{\Ber{1/2}}=h(-Y_i\delta_i)=h(\delta_i)$.
On the other hand, regarding the constraints,
\[
p_i-B_i=\tfrac{1-Y_i\delta_i}{2}-\tfrac{1+Y_i}{2}=-\tfrac{Y_i}{2}(1+\delta_i)\, ,
\]
so $p-B=-\tfrac12 D_Y(\1_n+\delta)$ with $D_Y=\mathrm{diag}(Y_1,\dots,Y_n)$. Noting that $Z=D_Y X$, we obtain that $X^\T(p-B)=-\tfrac12 X^\T D_Y(\1_n+\delta)=-\tfrac12 Z^\T(\1_n+\delta)$. Hence the constraint $X^\T p=X^\T B$ is equivalent to $Z^\T(\1_n+\delta)=0$, i.e.\ to $\1_n+\delta\in\ker(Z^\T)=\cV^\perp$.

\paragraph{Claim~\ref{item:chi-limit-still-works}:}

Since $h$ is even, $h(0)=h'(0)=0$, and
$h''(\delta)=\frac{1}{1-\delta^2}$, a Taylor expansion yields 
\begin{equation}\label{eqn:taylor-expansion-h}
\frac{\delta^2}{2}\leq h(\delta) \leq \frac{\delta^2}{2} + C \delta^4, \quad \mbox{for any}~\delta \in [-1, 1],
\end{equation}
where above $C > 0$ is some universal constant. 
Let $P_{\cV}$ be the orthogonal projector onto $\cV$. Let $\cE_n = \{\|P_\cV \mathbf{1}_n\|_\infty \leq 1\}$. 
From~\cref{eqn:var-char-llr,eqn:taylor-expansion-h}, we obtain that on $\cE_n$, 
\begin{equation}
\label{eq:chi-limit-upper-det}
\|P_{\cV} \mathbf{1}_n\|_2^2 \leq 2 \Lambda_{n,d} \leq \|P_{\cV} \mathbf{1}_n\|_2^2  + 2C \|P_{\cV} \mathbf{1}_n\|_4^4.
\end{equation}
The next lemma allows us to control the probability of the complementary event.
\begin{lemma}
\label{lemma:the-bad-event-has-small-probability}
If $d^{3/2}/n \to 0$, then $\lim_{n \to \infty} \P(\cE_n^c)= 0$
\end{lemma}
\begin{proof}
Set $u_n = \tfrac{1}{\sqrt{n}} \1_n$. We can write $\cE_n^c = \{\|P_\cV \1_n\|_\infty > 1\} = \{\sqrt{n}\|P_\cV u_n\|_\infty > 1\}$.

Denote by $\alpha = \|P_\cV u_n\|_2$. We have 
\[
P_\cV u_n = \alpha^2 u_n + \alpha \sqrt{1 - \alpha^2} \theta, 
\]
where $\theta$ is distributed uniformly on the unit sphere in $u_n^\perp$ and is independent of $\alpha$. Hence for each $i$,
\[
(P_\cV u_n)_i=\frac{\alpha^2}{\sqrt{n}} + \alpha\sqrt{1-\alpha^2}\,\theta_i.
\]
By a union bound, it holds that 
\[
\P(\cE_n^c) \leq n \P( |\alpha^2 + \sqrt{n} \alpha\sqrt{1-\alpha^2}\,\theta_1| > 1) 
\leq  n\Big(\P(\alpha^2>1/2)
+\P(\sqrt n\,\alpha|\theta_1|>1/2)\Big).
\]
Hence, using that $\alpha^2 = \mathsf{Beta}(d/2, (n-d)/2)$ in distribution, we have 
\[
\P(\alpha^2>1/2)\leq 8\, \E[\alpha^6] = 8 \frac{d(d+2)(d+4)}{n(n+2)(n+4)} \lesssim \frac{d^3}{n^3}.
\]
Similarly, we have 
\[
\P(\sqrt n\,\alpha|\theta_1|>1/2) \leq 4096 d^3 \E \theta_1^6 \lesssim \frac{d^3}{n^3}. 
\]
Therefore, combining the previous three displays yields $\P(\cE_n^c) \lesssim d^3/n^2 \to 0$, as required.
\end{proof}

We also recall the Lévy concentration function of a real-valued random variable $Z$, defined 
\[
\cL_Z(\eps) \defn \sup_{x\in \R}\P(x-\eps<Z\leq x+\eps).
\]
We make use of the following basic coupling inequality for the Kolmogorov-Smirnov distance.
\begin{lemma}
\label{lem:ks-coupling}
For any coupling of real-valued random variables $U$ and $V$, and any $\eps>0$,
\begin{equation}
\label{eq:chi-limit-ks-basic}
\KSDist{U}{V}
\le
\P(|U-V|>\eps)
+
\cL_V(\eps/2).
\end{equation}
\end{lemma}
Lemma~\ref{lem:ks-coupling} is proved in Section~\ref{sec:asymp-remaining-proofs}.

\begin{lemma}[Lévy concentration function for $\chi^2_d$]
\label{lem:chi-levy-concentration-function}
	There exists a constant $C > 0$ such that the following two inequalities hold:
	\begin{enumerate}[label=(\roman*)]
	\item \label{ineq:chi-levy-bound-one} for $d \geq 2$ and $\eps > 0$, we have $\cL_{\chi^2_d}(\eps) \leq C \eps/\sqrt{d}$; and 
	\item \label{ineq:chi-levy-bound-two} for $d \geq 1$ and $\eps > 0$, we have $\cL_{\chi^2_d}(\eps) \leq C \sqrt{\eps}$. 
	\end{enumerate} 
\end{lemma}
\begin{proof}
	We write the density of $\chi^2_d$ as $f_d(x)=\tfrac{1}{2^{d/2}\Gamma(d/2)}x^{\frac d2-1}e^{-x/2}$, for $x > 0$. 
	For $d=2$, the density is decreasing and $\|f_2\|_\infty=1/2$. For $d\geq3$, differentiation shows that its unique critical point is at $x=d-2$, which is the maximizer. 
	Hence, for $d\geq3$, the Stirling lower bound on the Gamma function, namely, $\Gamma(z)\geq \sqrt{2\pi}\, z^{\,z-\frac12}e^{-z}$ for $z\geq 1$, gives
	\[
	\|f_d\|_\infty
	= \frac{(d-2)^{\frac d2-1}e^{-(d-2)/2}}{2^{d/2}\Gamma(d/2)}
	\leq
	\frac{(d-2)^{\frac d2-1}e^{-(d-2)/2}}{2^{d/2}\sqrt{2\pi}(d/2)^{\frac d2-\frac12}e^{-d/2}}
	=
	\frac{e}{2\sqrt{\pi d}}
	\Bigl(1-\frac{2}{d}\Bigr)^{\frac d2-1}
	\leq \frac{C_1}{2\sqrt{d}}.
	\]
	Above, we set $C_1 = \tfrac{e}{\sqrt{\pi}}$; together with the case $d=2$, claim~\ref{ineq:chi-levy-bound-one} immediately follows after increasing $C$ if needed. 
	First consider $d=1$. If $G\sim \Normal{0}{1}$, then $G^2\sim \chi^2_1$, and for $0 \leq a < b$, we have 
\[
\P(G^2\in (a, b])
=2\int_{\sqrt a}^{\sqrt b}\frac{e^{-t^2/2}}{\sqrt{2\pi}}\,dt
\leq \sqrt{\frac{2}{\pi}}(\sqrt b-\sqrt a)
\leq \sqrt{\frac{2}{\pi}}\sqrt{b-a}.
\]
This immediately implies $\cL_{\chi^2_1}(\varepsilon) \leq \frac{2}{\sqrt{\pi}}\sqrt{\varepsilon}$. 
For $d \geq 2$, write $\chi^2_d = G^2 + Z$, in distribution. Here, $Z\sim \chi^2_{d-1}$ is independent of $G$. Conditioning on $Z$,
we have 
\[
\P(x -\eps < \chi^2_d \leq x+\eps)
=\E \P(G^2\in (x-Z-\eps, x -Z +\eps]\mid Z ) 
\leq \cL_{\chi^2_1}(\eps). 
\]
Consequently, claim~\ref{ineq:chi-levy-bound-two} follows with $C \geq C_2 \defn \frac{2}{\sqrt{\pi}}$. 
We may take $C = \max\{C_1, C_2\}$. 
\end{proof}

As $d^{3/2}/n\to 0$, we may assume without loss of generality that $n$ is large enough such that $d<n$ holds.
Consequently, noting that $\cV$ is distributed uniformly among $d$-dimensional subspaces of $\R^n$, we 
have 
\[
\|P_{\cV} \mathbf{1}_n\|_2^2 = n \, \mathsf{Beta}\Bigl(\frac d2,\frac{n-d}{2}\Bigr),
\]
in distribution; indeed this is easily seen by writing $P_\cV = U U^\T$ for $U \in \R^{n\times d}$ a matrix with orthonormal columns, 
so that $\|P_{\cV} \mathbf{1}_n\|_2/\sqrt n$ is distributed as the norm of the first $d$ coordinates of a vector uniformly distributed on the unit sphere of $\R^n$. We recall the following standard result that will also be used in the second part of the proof.

\begin{fact}
\label{fact:beta-gamma-coupling}
Let $1\leq d<n$ be integers. Let $B \sim \mathsf{Beta}\big(\frac d2,\frac{n-d}{2}\big)$ and $Z\sim \chi^2_n$ be independent. Then, $Z B \sim \chi^2_d$.
\end{fact}

\noindent Let $Z_n \sim \chi^2_n$ be independent of $\cV$. By~\cref{eq:chi-limit-upper-det}, it holds on $\cE_n$ that 
\[
\big|2\Lambda_{n,d} - \tfrac{Z_n}{n} \|P_\cV \1_n\|_2^2\big| 
\leq \Big|2 \Lambda_{n,d} - \|P_\cV \1_n\|_2^2\Big| +  \|P_\cV \1_n\|_2^2 \big|1- \tfrac{Z_n}{n}\big| 
\leq C' \|P_\cV \1_n\|_4^4 +  \|P_\cV \1_n\|_2^2 \big|1- \tfrac{Z_n}{n}\big| \, .
\]
For completeness, the required fourth moment estimate follows from the
decomposition used in the proof of~\Cref{lemma:the-bad-event-has-small-probability}.
Put $u_n=n^{-1/2}\1_n$ and $\rho=\|P_\cV u_n\|_2$. Then
\[
P_\cV u_n=\rho^2u_n+\rho\sqrt{1-\rho^2}\,\theta,
\]
where $\rho^2\sim\mathsf{Beta}(d/2,(n-d)/2)$ and $\theta$ is independent of
$\rho$ and uniform on the unit sphere of $u_n^\perp$. In particular,
$\E\theta_1^2=1/n$, $\E\theta_1^4\leq3/n^2$, and symmetry gives
\[
\begin{aligned}
\E\|P_\cV\1_n\|_4^4
&=n\E\left(\rho^2+\sqrt n\,\rho\sqrt{1-\rho^2}\,\theta_1\right)^4\\
&\leq n\left(\E\rho^8+6\E\rho^6+3\E\rho^4\right)\\
&\lesssim n\left(\frac{d^4}{n^4}+\frac{d^3}{n^3}+\frac{d^2}{n^2}\right)
\lesssim\frac{d^2}{n}.
\end{aligned}
\]
Now, let $\alpha, \beta > 0$. It holds that 
\begin{align}
\P(\big|2\Lambda_{n,d} &- \tfrac{Z_n}{n} \|P_\cV \1_n\|_2^2\big| > \alpha + \beta) 
\leq \P(\cE_n^c) + \P(\|P_\cV \1_n\|_4^4 > \alpha/C') + \P(\|P_\cV \1_n\|_2^2 \big|1- \tfrac{Z_n}{n}\big| > \beta) \\ 
&\leq \P(\cE_n^c) + \frac{C}{\alpha} \frac{d^2}{n} + \frac{C}{\beta^2} \frac{d^2}{n}.
\end{align} 
Above, $C > 0$ denotes a sufficiently large universal constant. We used Markov's inequality for the first term and, for the second term, $\Var(Z_n) = 2n$ and $\E \|P_\cV \1_n\|_2^4 = n d(d+2)/(n+2) \lesssim d^2$ 
and Chebyshev's inequality. Fix $\eps > 0$. Put 
\[
\alpha_n = \frac{2C}{\eps}\frac{d^2}{n}, 
\quad \beta_n^2 = \frac{2C}{\eps}\frac{d^2}{n}, 
\quad \mbox{and} \quad 
\delta_n = \alpha_n + \beta_n.
\]
Then, 
\begin{equation}\label{ineq:upper-bound-to-chi^2-term-1}
	\P\Big(\big|2\Lambda_{n,d} - \tfrac{Z_n}{n} \|P_\cV \1_n\|_2^2\big| > \delta_n\Big) 
	\leq 
	\P(\cE_n^c) + \eps. 
\end{equation}
Now, observe that by Fact~\ref{fact:beta-gamma-coupling}, $\tfrac{Z_n}{n} \|P_\cV \1_n\|_2^2$ is distributed as $\chi^2_d$. Thus, if $d\geq2$, then by~\Cref{lem:chi-levy-concentration-function}\ref{ineq:chi-levy-bound-one},
\[
\cL_{\tfrac{Z_n}{n} \|P_\cV \1_n\|_2^2}(\delta_n) = 
\cL_{\chi^2_d}(\delta_n) \lesssim \frac{\delta_n}{\sqrt{d}} 
= \frac{2C}{\eps}\frac{d^{3/2}}{n}
+\sqrt{\frac{2C}{\eps}}\sqrt{\frac{d}{n}} \longrightarrow 0.
\]
If $d=1$, then by~\Cref{lem:chi-levy-concentration-function}\ref{ineq:chi-levy-bound-two},
\[
	\cL_{\tfrac{Z_n}{n} \|P_\cV \1_n\|_2^2}(\delta_n) = 
	\cL_{\chi^2_1}(\delta_n) \lesssim \sqrt{\delta_n}\longrightarrow0.
\]
Combining this with~\cref{ineq:upper-bound-to-chi^2-term-1,lem:ks-coupling} and using~\Cref{lemma:the-bad-event-has-small-probability}, we have 
\[
\limsup_{n \to \infty} \KSDist{2\Lambda_{n,d}}{\chi^2_d} \leq \eps.
\]
Now recall that $\eps > 0$ was arbitrary, and hence we obtain the result.

\paragraph{Claim~\ref{item:chi-limit-fails}:}
Now assume that $d^{3/2}/n\nrightarrow 0$; equivalently 
$\limsup_{n \to \infty} d^{3/2}/n > 0$. 
First suppose that $d\geq n$ along an infinite subsequence, relabeled by $n_k$. Then $Z\in\R^{n_k\times d_{n_k}}$ has rank $n_k$ almost surely, and hence $\Lambda_{n_k,d_{n_k}}=n_k\log2$ almost surely. Since the chi-square distribution is continuous, its KS distance from any point mass is at least $1/2$. Consequently,
\[
\limsup_{n \to \infty} \KSDist{2\Lambda_{n, d}}{\chi^2_d} 
\geq \frac12.
\]
It remains to consider the case where $d<n$ eventually.
Choose a subsequence, relabeled by $n$, such that
\[
\tau_n \defn \frac{d^{3/2}}{n}\to \tau\in (0, \infty],
\qquad
\frac{d}{n}\to \rho\in [0, 1].
\]
Set $r(\delta)\defn h(\delta)-\tfrac{\delta^2}{2}$. 
Since $r(0)=r'(0)=0$ and $r''(\delta) = \tfrac{\delta^2}{1-\delta^2} \geq \delta^2$ for $|\delta|\leq 1$, 
an integration argument yields
\[
h(\delta)\geq \frac{\delta^2}{2}+\frac{\delta^4}{12},
\qquad \mbox{for}~|\delta|\leq 1.
\]
Let $\delta\in[-1,1]^n$ be a minimizer in~\cref{eqn:var-char-llr}, which exists by compactness,
and write $\delta=-P_\cV \1_n+s$ where $s\in \cV^\perp$.
Note that 
	\[
	\|\delta\|_2^2 = \|P_\cV \1_n\|_2^2+\|s\|_2^2\geq \|P_\cV \1_n\|_2^2.
	\]
Cauchy-Schwarz implies that $\|x\|_4^4 \geq \tfrac{\|x\|_2^4}{n}$ for any $x \in \R^n$.
Hence,
\begin{equation}\label{eq:chi-limit-lower-det}
	2 \Lambda_{n,d}
\geq
	\|\delta\|_2^2 + \frac{1}{6}\|\delta\|_4^4
	\geq \|\delta\|_2^2 + \frac{1}{6n}\|\delta\|_2^4
	\geq \|P_\cV \1_n\|_2^2 + \frac{1}{6n}\|P_\cV \1_n\|_2^4,
\end{equation}
where the first inequality uses $h(\delta_i) \geq \delta_i^2/2 + \delta_i^4/12$, the second uses $\|\delta\|_4^4 \geq \|\delta\|_2^4/n$ (Cauchy--Schwarz), and the third uses $\|\delta\|_2^2 \geq \|P_\cV \1_n\|_2^2$.
Let $\kappa = \tfrac{1}{20}$ and $a_n = \kappa\tfrac{d^2}{n}$. 
Observe that if $\|P_\cV \1_n\|_2^2\geq d-a_n$, then by~\cref{eq:chi-limit-lower-det},
	\[
	2\Lambda_{n,d}
	\geq
	d-a_n+\frac{(d-a_n)^2}{6n}
=
d+\Big(\frac{(1-\kappa d/n)^2}{6} - \kappa\Big)\frac{d^2}{n} 
\geq d + 2a_n. 
\]
The final inequality uses $d<n$ and $\kappa=1/20$. 
Hence, $\{\|P_\cV \1_n\|_2^2\geq d-a_n\} \subset \{
2\Lambda_{n,d}>d + a_n\}$. Thus, 
\begin{equation}\label{eq:chi-limit-lower-ks}
\KSDist{2\Lambda_{n,d}}{\chi^2_d}
\geq
\P(\chi^2_d\leq d + a_n)-\P(\|P_\cV \1_n\|_2^2<d-a_n).
	\end{equation}
	Now, suppose that $\rho \in (0, 1]$. 
	Then, for $n$ large enough, $d/n \geq \rho/2$. 
	Using the fact that ${\|P_\cV \1_n\|_2^2/n \sim \mathsf{Beta}(\tfrac{d}{2},\tfrac{(n-d)}{2})}$, we have
	$\E\|P_\cV \1_n\|_2^2=d$, while $\Var(\|P_\cV \1_n\|_2^2)=\frac{2d(n-d)}{n+2}\leq 2d$. 
	Hence Chebyshev's inequality yields
		\[
		\max\Big\{\P(\|P_\cV \1_n\|_2^2<d-a_n), \P(\chi^2_d > d + a_n)\Big\} 
	\leq
	\frac{\max\{\Var(\|P_\cV \1_n\|_2^2), \Var(\chi^2_d)\}}{a_n^2}
	\asymp 
	\frac{n^2}{d^3} 
	\lesssim \frac{1}{\rho^3 n}
\] 
Combined with display~\cref{eq:chi-limit-lower-ks}, we have 
$\KSDist{2\Lambda_{n,d}}{\chi^2_d}\to 1$, along this subsequence.

Finally, suppose $\rho = 0$ so that $d/n\to 0$ and necessarily $d \to \infty$. 
We again introduce $Z_n\sim\chi^2_n$, independent of $\cV$, and therefore, of $\|P_\cV\1_n\|_2^2$. Let us also define $T_n\defn\tfrac{Z_n}{n}\|P_\cV\1_n\|_2^2$, which by Fact~\ref{fact:beta-gamma-coupling}, has distribution $\chi^2_d$. The conclusion will follow from Slutsky's theorem by combining the convergence $\tfrac{1}{\sqrt{2d}}\bigl(\|P_\cV\1_n\|_2^2-T_n\bigr)\to0$ in probability with the convergence $\tfrac{T_n-d}{\sqrt{2d}}\distto\Normal{0}{1}$ in distribution, which we establish in the following lines. 
Fix $\eps>0$ and $M>0$. Then, observe that 
\begin{align}
\P\Big(\frac{\big|\|P_\cV \1_n\|_2^2-\tfrac{Z_n}{n}\|P_\cV \1_n\|_2^2\big|}{\sqrt{2d}}>\eps\Big)
&\leq
\P\big(\|P_\cV \1_n\|_2^2>Md\big)
+
\P\Big(\big|Z_n - n\big|>\frac{n\eps\sqrt{2}}{M\sqrt{d}}\Big) \\ 
&\leq \frac{1}{M} + \frac{M^2d}{n \eps^2} \to \frac{1}{M}, \quad \mbox{as $n \to \infty$.}
\end{align}
Since this holds for each $M$ and $\eps> 0$, we have 
$\tfrac{1}{\sqrt{2d}}(\|P_\cV \1_n\|_2^2-\tfrac{Z_n}{n}\|P_\cV \1_n\|_2^2)
\to 0$ in probability as $n \to \infty$. 
Additionally, we have by the central limit theorem, that 
\[
\frac{\tfrac{Z_n}{n}\|P_\cV \1_n\|_2^2-d}{\sqrt{2d}}
\overset{\rm{d}}{=}
\frac{\chi^2_d-d}{\sqrt{2d}}
\distto 
\Normal{0}{1},
\]
as $n \to \infty$. 
By Slutsky's theorem, we have $\tfrac{\|P_\cV \1_n\|_2^2-d}{\sqrt{2d}} \distto \Normal{0}{1}$ as $n \to \infty$. 
Set
\[
\alpha_n \defn \frac{a_n}{\sqrt{2d}}
=
\frac{\kappa}{\sqrt{2}}\frac{d^{3/2}}{n}
=
\frac{\kappa}{\sqrt{2}}\tau_n
\to \alpha\in (0, \infty].
\]
Hence, as $n \to \infty$, it holds that 
\[
\P(\|P_\cV \1_n\|_2^2<d-a_n)
=
\P\Bigl(\frac{\|P_\cV \1_n\|_2^2-d}{\sqrt{2d}}<-\alpha_n\Bigr)
\to
\begin{cases}
\Phi(-\alpha), & \alpha<\infty,\\
0, & \alpha=\infty,
\end{cases}
\]
On the other hand, by the central limit theorem, as $n \to \infty$, it holds that 
\[
\P(\chi^2_d\leq d+a_n)
=
\P\Bigl(\frac{\chi^2_d-d}{\sqrt{2d}}\leq \alpha_n\Bigr)
\to
\begin{cases}
\Phi(\alpha), & \alpha<\infty,\\
1, & \alpha=\infty.
\end{cases}
\]
Combining the previous two displays along with~\cref{eq:chi-limit-lower-ks}, we obtain
\[
\liminf_{n \to \infty}\KSDist{2\Lambda_{n,d}}{\chi^2_d}
\geq
\begin{cases}
2\Phi(\alpha)-1, & \alpha<\infty,\\
1, & \alpha=\infty,
\end{cases}
\]
as $n \to \infty$. This proves part~\ref{item:chi-limit-fails}.

\subsubsection{Remaining proofs}
\label{sec:asymp-remaining-proofs}

\begin{proof}(\emph{of Lemma~\ref{lem:ks-coupling}})
Let $F_U$ and $F_V$ denote the CDFs of $U$ and $V$. Put $A\defn \{|U-V|>\eps\}$. 
Then, the following inclusions hold on $A^c$, for any $x \in \R$: 
\[
\{U\leq x\}\subseteq \{V\leq x+\eps\}
\qquad\mbox{and}\qquad
\{V\leq x-\eps\}\subseteq \{U\leq x\}.
\]
Therefore, by a union bound on $A, A^c$, we have 
\[
F_V(x-\eps)-\P(A)\leq F_U(x)\leq F_V(x+\eps)+\P(A).
\]
Using $\max\{F_V(x) - F_V(x -\eps), F_V(x+\eps) - F_V(x)\} \leq \cL_V(\eps/2)$, we have
\[
\sup_{x \in \R}|F_U(x)-F_V(x)|\leq \P(A)+ \cL_V(\eps/2),
\]
as required.
\end{proof}

\begin{proof}(\emph{of Fact~\ref{fact:beta-gamma-coupling}})
Write $Z\sim \chi^2_n$ as $Z = A + A'$, with $A\sim\chi^2_d$ and $A'\sim\chi^2_{n-d}$, independent of $A$. The ratio $A/(A+A')\sim\mathsf{Beta}(\tfrac d2,\tfrac{n-d}{2})$ is independent of the total $A+A'\sim\chi^2_n$, and their product is $A\sim\chi^2_d$.
\end{proof}

\subsubsection*{Acknowledgements}

Reese Pathak gratefully acknowledges support from the National Science
Foundation, under grant DMS-2503579.

{\footnotesize
\bibliographystyle{abbrv}
\bibliography{references}
}

\appendix

\section{AI disclosure and the proof of the \(d=2\) upper bound}
\label{app:ai-disclosure-d2}

The results in the main body of the paper were developed by the authors from summer 2025 to March 2026. During the
development of the main body, the authors carried out experiments
with public OpenAI models up to GPT-5.4. These experiments provided no decisive mathematical AI contribution that would warrant disclosure.

As an informal benchmark for current AI capabilities, the authors used the \emph{lower bound} from
\Cref{prop:d2-logloglog-lower}. In tests conducted on June 11, 2026, GPT-5.5 Pro and Claude Fable 5 were the first among the tested models to
recover a similar construction from the problem statement alone in a single
response, without hints or access to previous conversations.

\Cref{thm:d2-logloglog-upper} below, which gives the \emph{upper bound} for $d=2$, is the
only significant result not contained in the March 2026 draft. Its proof was produced in July 2026 through a process involving GPT-5.6-Sol Ultra and multiple runs of GPT-5.6-Sol Pro. To keep this distinction clear, the result is placed in this appendix. The
authors verified the argument and present a simplified proof sketch below. 

The pipeline used by the authors for the proof of \Cref{thm:d2-logloglog-upper} was
subsequently refined into a single prompt containing a list of key
hints. In repeated tests in July 2026, GPT-5.6-Sol Pro
recovered versions of the proof of the upper bound in a single response from this prompt and could not do so when the last hint was removed. The final prompt is presented after the results below.

\begin{theorem}[Worst case upper bound in dimension two]
\label{thm:d2-logloglog-upper}
There exists a constant $C>0$ such that, for any $n\geq16$ and any
$\delta\in(0,1/\e]$, it holds that
\[
\sup_{x_1,\ldots,x_n\in\R^2}
\sup_{\thetastar\in\R^2}
\Quant{1-\delta}
\Bigl(
\Lambda_n^{\operatorname{log}}
(\thetastar;x_{1:n};Y_{1:n})
\Bigr)
\leq
C\left(
\log\log\log n+\log\frac{1}{\delta}
\right).
\]
\end{theorem}

\begin{proofsketch}
Throughout the proof, $C,c>0$ denote universal constants whose values may
change from line to line.
For simplicity, in the derivation below we assume that the relevant
suprema are attained.
By \Cref{lem:reformulation-via-subspaces}, it is enough to fix a subspace
$W\subset\R^n$ with $\dim(W)\leq2$, a vector $v\in W$, and
$\eps\sim p_v$.  Recall that
\[
p_v(\eps)=\prod_{i=1}^n\sigma(\eps_i v_i),
\qquad
\Gamma_W(\eps;v)=\sup_{w\in W}\log\frac{p_w(\eps)}{p_v(\eps)}.
\]
We prove the moment bound
\begin{equation}
\label{eq:d2-sketch-mgf}
\E_{\eps\sim p_v}\exp\bigl(\eta\Gamma_W(\eps;v)\bigr)
\leq C(2+\log\log n)^C
\end{equation}
for universal $\eta>0$ and $n\geq16$.  Since the right hand side is uniform
in $W$ and $v$, Markov's inequality will give the desired quantile bound.
We treat the different margin ranges separately and combine the bounds at the
end.

\paragraph{The large margin coordinates.}
We first separate the coordinates according to the size of their true margins.
Let $a_i=|v_i|$ and $s_i=\sign(v_i)$, with $s_i=1$ when $v_i=0$, and define
\[
B_i=\1\{\eps_i=-s_i\},
\qquad q_i=\sigma(-a_i)\in(0,1/2].
\]
Thus $B_i=1$ means that the observed label goes against the more likely label,
and the $B_i$ are independent with $B_i\sim\Ber{q_i}$.  Let
\[
\mathcal E=\{i:a_i>2\log n\}
\]
be the set of large margin coordinates.  Define
\[
\begin{aligned}
p_{\mathcal E}(b)&=\prod_{i\in\mathcal E}q_i^{b_i}(1-q_i)^{1-b_i},
&b&\in\{0,1\}^{|\mathcal E|},\\
p_{w,\mathcal E}(b)&=\prod_{i\in\mathcal E}
\sigma\bigl(s_i(1-2b_i)w_i\bigr),\\
\Gamma_{\mathcal E}(b)&=\sup_{w\in W}
\log\frac{p_{w,\mathcal E}(b)}{p_{\mathcal E}(b)}.&&
\end{aligned}
\]
Since every fitted mass is at most one,
$\Gamma_{\mathcal E}(b)\leq-\log p_{\mathcal E}(b)$.  Thus, for any fixed
$0<\eta<1/2$,
\begin{equation}
\label{eq:d2-sketch-large-margin-moment}
\E\exp\bigl(\eta\Gamma_{\mathcal E}(B_{\mathcal E})\bigr)
\leq\sum_{b_{\mathcal E}}p_{\mathcal E}(b_{\mathcal E})^{1-\eta}
=\prod_{i\in\mathcal E}\bigl((1-q_i)^{1-\eta}+q_i^{1-\eta}\bigr)\leq C.
\end{equation}
Here the last inequality uses $q_i\leq n^{-2}$ on $\mathcal E$.

\paragraph{The retained coordinates and their block likelihood ratios.}
Write $\mathcal R=[n]\setminus\mathcal E$ for the retained coordinates.  Let
$I_{\mathrm{cen}}=\{i\in\mathcal R:a_i\leq1\}$ and split the remaining
coordinates into the nonempty dyadic blocks
$\{i\in\mathcal R:2^j<a_i\leq2^{j+1}\}$, $j\geq0$.  Enumerate all nonempty
retained blocks as $I_1,\ldots,I_H$.  Since $a_i\leq2\log n$ on $\mathcal R$,
\begin{equation}
\label{eq:d2-sketch-block-count}
H\leq C(1+\log\log n).
\end{equation}
Choose a linear map $T:\R^2\to W$ with range $W$; then $v+Tu$ runs through
$W$ as $u$ runs through $\R^2$.
For $b_{I_s}\in\{0,1\}^{|I_s|}$, define the likelihood ratio of block $I_s$ by
\[
L_s(u;b_{I_s})
=\prod_{i\in I_s}
\frac{\sigma\bigl(s_i(1-2b_i)(v_i+(Tu)_i)\bigr)}
{\sigma\bigl(s_i(1-2b_i)v_i\bigr)},
\qquad
\Gamma_s(b_{I_s})=\sup_{u\in\R^2}\log L_s(u;b_{I_s}).
\]
By subadditivity,
\begin{equation}
\label{eq:d2-sketch-margin-decomposition}
\Gamma_W(\eps;v)
\leq\Gamma_{\mathcal E}(B_{\mathcal E})
+\sup_{u\in\R^2}\sum_{s=1}^H\log L_s(u;B_{I_s}).
\end{equation}
We first prove, uniformly for every retained block,
\begin{equation}
\label{eq:d2-sketch-block-moment}
\E\exp\bigl(\eta_0\Gamma_s(B_{I_s})\bigr)\leq C,
\qquad 1\leq s\leq H,
\end{equation}
for universal constants $\eta_0,C>0$, where the expectation is under the true
block law.

\paragraph{The dyadic block bound.}
Consider a dyadic block with $A\leq a_i\leq2A$.  The normalized space
\[
\left\{\left(\frac{s_iw_i}{a_i}\right)_{i\in I_s}:w\in W\right\}
\]
has dimension at most two and contains the constant vector $\1$.  We may
therefore write its elements as $\alpha+\beta t_i$, with $t_i=0$ in the case
of rank one.  Set
\[
\omega_i=\frac{a_i}{A}-1\in[0,1],
\qquad y_i=A(1-\alpha-\beta t_i),
\qquad
\ell_i(x;b)=bx-\log\bigl(1-q_i+q_i\exp(x)\bigr).
\]
Fix for now $b\in\{0,1\}^{|I_s|}$; later we take $b=B_{I_s}$ under the true
block law.  The block log likelihood ratio at $(\alpha,\beta)$ is
$\sum_{i\in I_s}\ell_i((1+\omega_i)y_i;b_i)$.  Since
$\ell_i(\,\cdot\,;b)$ is concave and $\ell_i(0;b)=0$, we have
\[
\ell_i((1+\omega_i)y;b)
\leq \ell_i(y;b)+\omega_i\ell_i(y;b).
\]
Taking suprema gives
\[
\Gamma_s(b)\leq G_s^{(0)}(b)+G_s^{(1)}(b),\qquad
G_s^{(0)}(b)=\sup_{\alpha,\beta}\sum_{i\in I_s}\ell_i(y_i;b_i),\qquad
G_s^{(1)}(b)=\sup_{\alpha,\beta}\sum_{i\in I_s}
\omega_i\ell_i(y_i;b_i).
\]
We bound $G_s^{(0)}(b)$ using two likelihood ratios in dimension one, both
covered by \Cref{cor:dimone}.  Reparameterize $y_i=\gamma+\lambda t_i$ and set
\[
P_0=\bigotimes_{i\in I_s}\Ber{q_i},
\qquad
P_{\gamma,\lambda}=\bigotimes_{i\in I_s}\Ber{q_i(\gamma,\lambda)},
\]
where
\[
q_i(\gamma,\lambda)=\sigma\left(\log\frac{q_i}{1-q_i}
+\gamma+\lambda t_i\right).
\]
Also write
\[
N(b)=\sum_{i\in I_s}b_i,
\qquad
U(b)=\sum_{i\in I_s}t_i b_i.
\]
For the random vector $B$, write $N=N(B)$ and $U=U(B)$.
With this notation,
\[
\log\frac{P_{\gamma,\lambda}(b)}{P_0(b)}
=\sum_{i\in I_s}\ell_i(\gamma+\lambda t_i;b_i),
\qquad
G_s^{(0)}(b)=\sup_{\gamma,\lambda}
\log\frac{P_{\gamma,\lambda}(b)}{P_0(b)}.
\]
Write $m=N(b)$ and suppose first that $0<m<|I_s|$.  For each
$\lambda$, let $\gamma(\lambda)$ be the unique intercept for which
$\E_{P_{\gamma(\lambda),\lambda}}N=m$; it is unique because the expected count
increases continuously from $0$ to $|I_s|$ with $\gamma$.  Direct
differentiation shows that $\gamma(\lambda)$ is the maximizing intercept.  We compare this law
with the law obtained by varying only the intercept and having the same
expected count.  Let
$\widehat\lambda$ be a maximizer of $\lambda\mapsto
\log(P_{\gamma(\lambda),\lambda}(b)/P_0(b))$.  Define
\[
\begin{gathered}
Q=P_{\gamma(\widehat\lambda),\widehat\lambda},
\qquad
R=P_{\gamma(0),0},
\qquad
\E_QN=\E_RN=m.
\end{gathered}
\]
Set
\[
\begin{aligned}
G_m(b)&=\sup_{\gamma\in\R}\log\frac{P_{\gamma,0}(b)}{P_0(b)}
=\log\frac{R(b)}{P_0(b)},\\
C_m(b)&=\sup_{\lambda\in\R}
\log\frac{P_{\gamma(\lambda),\lambda}(b\mid N=m)}
{P_0(b\mid N=m)}\\
&=\sup_{\lambda\in\R}
\left(\lambda U(b)-
\log\E_{P_0}\left[\exp(\lambda U)\mid N=m\right]\right).
\end{aligned}
\]
The first quantity is the log likelihood ratio obtained by varying only the
intercept.  In the second, conditioning on $N=m$ cancels the intercept factor,
leaving the conditional likelihood ratio obtained by varying the slope with
the observed count fixed.  Differentiating the profiled likelihood at
$\widehat\lambda$ gives $\E_QU=U(b)$.  Since $\log(\ud Q/\ud R)$ is affine in $N$ and $U$, and
$\E_QN=N(b)=m$, its value at $b$ equals its expectation under $Q$.  Hence
\[
\log\frac{Q(b)}{R(b)}
=\kl{Q}{R}.
\]
Moreover,
\[
\log\frac{Q(b)}{R(b)}
=\log\frac{Q(N=m)}{R(N=m)}
+\log\frac{Q(b\mid N=m)}{R(b\mid N=m)}.
\]
Since $R(\,\cdot\mid N=m)=P_0(\,\cdot\mid N=m)$, the definition of $C_m(b)$
bounds the second term, and therefore
\[
\kl{Q}{R}
\leq\log\frac{Q(N=m)}{R(N=m)}+C_m(b).
\]

Since $N$ is a sum of independent Bernoulli variables with integer mean $m$
under both $Q$ and $R$, the standard estimate
\cite[Theorem~1]{BaillonCominettiVaisman2016}, together with the mode property
\cite[Theorem~4]{Darroch1964} and Chebyshev's inequality for the reverse bound,
gives
\[
Q(N=m)\asymp(1+\Var_Q(N))^{-1/2},
\qquad
R(N=m)\asymp(1+\Var_R(N))^{-1/2}.
\]
The atom bounds give
\[
\log\frac{Q(N=m)}{R(N=m)}
\leq C+\frac12\log\frac{1+\Var_R(N)}{1+\Var_Q(N)},
\qquad
\Var_R(N)\leq\frac32\Var_Q(N)+3\kl{Q}{R},
\]
where the second inequality follows by summing
$r(1-r)\leq\tfrac32p(1-p)+3\kl{\Ber{p}}{\Ber{r}}$. Moreover, with
$K=\kl{Q}{R}$, the variance bound gives
$(1+\Var_R(N))/(1+\Var_Q(N))\leq C(1+K)$.
Consequently,
\[
\log\frac{Q(N=m)}{R(N=m)}
\leq C+\frac12\kl{Q}{R}.
\]
Since $\log\frac{Q(b)}{P_0(b)}=G_m(b)+\kl{Q}{R}$, the preceding bounds give
\begin{equation}
\label{eq:d2-sketch-scalar-reduction}
G_s^{(0)}(b)\leq G_m(b)+C+2C_m(b).
\end{equation}
When $m=0$ or $m=|I_s|$, the family obtained by varying only the intercept has
the same supremum for the observed mass, and the conclusion holds with
$C_m(b)=0$.  We now take
$b=B=(B_i)_{i\in I_s}\sim P_0$ and set
\[
G_N=G_{N(B)}(B),
\qquad
C_N=C_{N(B)}(B).
\]
Thus both terms are covered by our existing result in dimension one: $G_N$
corresponds to varying the intercept, while, conditionally on $N$, $C_N$
corresponds to varying the slope.
Part~(i) of \Cref{cor:dimone} gives, for $t\geq\log2$,
\[
\P_{P_0}(G_N>t)\leq2\exp(-t),
\qquad
\P_{P_0}(C_N>t\mid N)\leq2\exp(-t).
\]
Integrating these tails and applying H\"older's inequality to the pointwise
bound above gives
\[
\E_{P_0}\exp\bigl(\eta G_s^{(0)}(B)\bigr)\leq C
\]
for a sufficiently small universal $\eta>0$.

For the weighted term, let $S_i\sim\Ber{\omega_i}$ independently of $B$.
Since $\sum_{i\in I_s}\omega_i\ell_i(y_i;B_i)
=\E_S\sum_{i\in I_s}S_i\ell_i(y_i;B_i)$, moving the supremum through the
expectation and applying Jensen's inequality gives
\[
\exp\bigl(\eta G_s^{(1)}(B)\bigr)
\leq\E_S\exp\left(\eta\sup_{\alpha,\beta}
\sum_{i\in I_s}S_i\ell_i(y_i;B_i)\right).
\]
For fixed values of the variables $S_i$, the supremum on the right is
$G_s^{(0)}$ on the selected coordinates and is zero if no coordinate is
selected.
Hence the preceding bound for $G_s^{(0)}$ gives
\[
\E_B\exp\bigl(\eta G_s^{(1)}(B)\bigr)
\leq\E_{B,S}\exp\left(\eta\sup_{\alpha,\beta}
\sum_{i\in I_s}S_i\ell_i(y_i;B_i)\right)\leq C.
\]
Thus H\"older's inequality, with a smaller $\eta$ if necessary,
proves~\eqref{eq:d2-sketch-block-moment} for every dyadic block.

\paragraph{The central block bound.}
It remains to consider the central block $I_{\mathrm{cen}}$; the empty case is
trivial, so assume it is nonempty.  Here $|v_i|\leq1$.  Let $r_i\in\R^2$ be
the rows of $T$, so that $(Th)_i=\langle r_i,h\rangle$.  Returning to the original label
encoding, set
\[
\xi_i=\1\{\eps_i=1\}-\sigma(v_i),
\qquad
k_i(z)=\log\frac{1+\exp(v_i+z)}{1+\exp(v_i)}-\sigma(v_i)z.
\]
Suppressing the observed labels from the notation, write
\[
\Gamma_{\mathrm{cen}}
=\sup_{h\in\R^2}\sum_{i\in I_{\mathrm{cen}}}
\left(\xi_i\langle r_i,h\rangle-k_i(\langle r_i,h\rangle)\right).
\]
It is standard to show that there is a universal $\kappa>0$ such that, uniformly for $|v_i|\leq1$,
\begin{equation}
\label{eq:d2-sketch-central-curvature}
k_i(z)\geq\kappa z^2\quad\text{when }|z|\leq1,
\qquad
\kl{\Ber{\sigma(v_i+z)}}{\Ber{\sigma(v_i)}}
\geq\kappa\min\{z^2,1\}.
\end{equation}
Fix $t\geq1$ and consider the event
$t\leq\Gamma_{\mathrm{cen}}<2t$.  Let $\widehat h$ be a maximizer.  The score
equation gives
\[
\Gamma_{\mathrm{cen}}
=\sum_{i\in I_{\mathrm{cen}}}
\kl{\Ber{\sigma(v_i+\langle r_i,\widehat h\rangle)}}
{\Ber{\sigma(v_i)}}<2t,
\]
and hence
\[
\sum_{i\in I_{\mathrm{cen}}}
\min\left\{\langle r_i,\widehat h\rangle^2,1\right\}\leq Ct.
\]
Consider the deterministic set
\[
\mathcal A_t=
\left\{i\in I_{\mathrm{cen}}:\exists h\in\R^2,
\ \sum_{j\in I_{\mathrm{cen}}}
\min\left\{\langle r_j,h\rangle^2,1\right\}\leq Ct,
\ |\langle r_i,h\rangle|\geq1\right\}.
\]
We claim that $|\mathcal A_t|\leq C(1+t)$.  For each
$i\in\mathcal A_t$, choose a corresponding $h_i$ and change its
sign and scale it toward zero so that $\langle r_i,h_i\rangle=1$, while
$\sum_{j\in I_{\mathrm{cen}}}
\min\left\{\langle r_j,h_i\rangle^2,1\right\}\leq Ct$.  Draw a directed edge
$i\to j$ when $i\ne j$ and $|\langle r_j,h_i\rangle|\geq1/2$.  Each
outgoing neighbor contributes at least $1/4$ to the preceding sum, so the total
number of directed edges is at most $Ct|\mathcal A_t|$.  The underlying graph
therefore has average degree at most $Ct$ and an independent set of size at
least $|\mathcal A_t|/(Ct+1)$.  On an independent set $S$, let
$M=(\langle r_j,h_i\rangle)_{j,i\in S}$.  The matrix $M$ has rank at most two,
has ones on its diagonal, and all its other entries have absolute value less
than $1/2$.  Since $\rank(M)\leq2$,
$|\operatorname{tr}(M)|^2\leq\rank(M)\|M\|_{\mathrm F}^2$, and hence
\[
|S|^2=|\operatorname{tr}(M)|^2
\leq2\sum_{i,j\in S}\langle r_j,h_i\rangle^2
<2|S|+\frac12|S|(|S|-1),
\]
so $|S|\leq3$, which proves the claimed bound.

For $i\in I_{\mathrm{cen}}\setminus\mathcal A_t$ we have
$|\langle r_i,\widehat h\rangle|<1$.  Thus, on this event,
$\Gamma_{\mathrm{cen}}$ is at most the sum of the unrestricted supremum over
$\mathcal A_t$ and the supremum over $I_{\mathrm{cen}}\setminus\mathcal A_t$
subject to $|\langle r_i,h\rangle|\leq1$ for all such $i$.
The unrestricted term is the likelihood ratio statistic on $\mathcal A_t$
for a subspace of dimension at most two.  If
this subspace has dimension zero, the supremum is zero; otherwise,
\Cref{lem:shtarkov-moment-bound-subspaces} and Markov's inequality give
\[
\P\left(
\sup_{h\in\R^2}\sum_{i\in\mathcal A_t}
\left(\xi_i\langle r_i,h\rangle-k_i(\langle r_i,h\rangle)\right)
\geq x\right)
\leq C(1+|\mathcal A_t|)^2\exp(-x).
\]
For the constrained supremum on the complement, let $\Pi_t$ be the orthogonal
projection onto the space
\[
\left\{(\langle r_i,h\rangle)_{i\in I_{\mathrm{cen}}\setminus\mathcal A_t}:
h\in\R^2\right\},
\]
whose dimension is at most two.  The first bound
in~\eqref{eq:d2-sketch-central-curvature} gives
\[
\sup_{\substack{h\in\R^2\\
|\langle r_i,h\rangle|\leq1,\ i\in I_{\mathrm{cen}}\setminus\mathcal A_t}}
\sum_{i\in I_{\mathrm{cen}}\setminus\mathcal A_t}
\left(\xi_i\langle r_i,h\rangle-k_i(\langle r_i,h\rangle)\right) \leq
\frac{1}{4\kappa}
\left\|\Pi_t(\xi_i)_{i\in I_{\mathrm{cen}}\setminus\mathcal A_t}\right\|_2^2.
\]
Indeed, applying Hoeffding's inequality to an orthonormal basis of the range
of $\Pi_t$ and taking a union bound shows that the right hand side has upper
tail $C\exp(-cx)$.  Together with
$|\mathcal A_t|\leq C(1+t)$, this proves
\[
\P(t\leq\Gamma_{\mathrm{cen}}<2t)
\leq C(1+t)^2\exp(-ct).
\]
Therefore, for sufficiently small universal $\eta_0>0$,
\[
\E\exp\bigl(\eta_0\Gamma_{\mathrm{cen}}\bigr)
\leq \exp(\eta_0)+\sum_{k\geq0}\exp\bigl(2\eta_0 2^k\bigr)
\P(2^k\leq\Gamma_{\mathrm{cen}}<2^{k+1})
\leq C.
\]
This proves~\eqref{eq:d2-sketch-block-moment} for the central block.

\paragraph{Aggregating the retained blocks.}
Below we write $L_s(u)=L_s(u;B_{I_s})$ and
$\Gamma_s=\Gamma_s(B_{I_s})$.  Thus
\[
\Gamma_{\mathrm{ret}}
=\sup_{u\in\R^2}\sum_{s=1}^H\log L_s(u).
\]
The direct bound $\Gamma_{\mathrm{ret}}\leq\sum_{s=1}^H\Gamma_s$ is of order
$H$, whereas we need order $\log(H+1)$.  The planar quantitative Helly theorem
reduces the relevant area comparison to at most four blocks.  There are only
polynomially many such choices, which gives the required logarithm.  To use
this reduction, we pass from maxima to integrals: concavity relates each
integral to the area of a set of near maximizers, while Fubini preserves the
mean one identity for likelihood ratios over the omitted blocks.

The case $H=0$ is trivial.  Assume that $H\geq1$, fix $R>0$, and define
$
\mathcal B_R=\left\{u\in\R^2:\|u\|_2\leq R\right\}.
$
For $P\subseteq\{1,\ldots,H\}$, set $L_\varnothing\equiv1$ and define
\[
L_P(u)=\prod_{s\in P}L_s(u),
\qquad
\Gamma_{P,R}=\max_{u\in\mathcal B_R}\log L_P(u),
\qquad
I_{P,R}=\int_{\mathcal B_R}L_P(u)\,du,
\]
and
\[
\mathcal K_{P,R}
=\left\{u\in\mathcal B_R:\log L_P(u)\geq\Gamma_{P,R}-1\right\}.
\]
Thus $\Gamma_{P,R}$ is the maximum block log likelihood ratio on the ball,
while $I_{P,R}$ is the corresponding integrated likelihood ratio.  We now
compare the integral with the area of $\mathcal K_{P,R}$.  The lower bound
follows by integrating over this set.  Let $u_{P,R}\in\mathcal B_R$ be a
maximizer in the definition of $\Gamma_{P,R}$.  Concavity of $\log L_P(\cdot)$
gives, for every $t\geq1$,
\[
\left\{u\in\mathcal B_R:
\Gamma_{P,R}-\log L_P(u)\leq t\right\}
\subseteq
u_{P,R}+t\left(\mathcal K_{P,R}-u_{P,R}\right).
\]
Consequently, we have
\[
\exp(-\Gamma_{P,R})I_{P,R}
=\int_0^\infty \exp(-t)
\operatorname{area}\left(\left\{u\in\mathcal B_R:
\Gamma_{P,R}-\log L_P(u)\leq t\right\}\right)\,dt \asymp\operatorname{area}(\mathcal K_{P,R}).
\]

We next measure the gap between maximizing the blocks separately over $\R^2$
and maximizing them jointly with one common parameter on $\mathcal B_R$.  Define this gap by
\[
\Delta_R=\sum_{s=1}^H\Gamma_s-\Gamma_{\{1,\ldots,H\},R}.
\]
The set on which no individual block
loses more than $\Delta_R+1$ is
\[
\mathcal C_s
=\left\{u\in\mathcal B_R:
\Gamma_s-\log L_s(u)\leq\Delta_R+1\right\},
\qquad
\mathcal C=\bigcap_{s=1}^H\mathcal C_s.
\]
Each term $\Gamma_s-\log L_s(u)$ is nonnegative.  Since
\[
\sum_{s=1}^H\left(\Gamma_s-\log L_s(u)\right)
=\Delta_R+\Gamma_{\{1,\ldots,H\},R}
-\sum_{s=1}^H\log L_s(u),
\]
we have
\[
\mathcal K_{\{1,\ldots,H\},R}\subseteq\mathcal C
\subseteq
\left\{u\in\mathcal B_R:
\sum_{s=1}^H\log L_s(u)
\geq\Gamma_{\{1,\ldots,H\},R}-H(\Delta_R+1)\right\}.
\]
The preceding area comparison therefore gives
\[
\operatorname{area}(\mathcal C)
\leq C(H+1)^2(\Delta_R+1)^2
\operatorname{area}(\mathcal K_{\{1,\ldots,H\},R}).
\]
Each $\mathcal C_s$ is compact and convex.  Moreover,
$\mathcal K_{\{1,\ldots,H\},R}\subseteq\mathcal C$ has positive area by
continuity.  Applying the planar quantitative Helly
theorem~\cite{BaranyKatchalskiPach1982} pathwise
gives a set
$\widehat P\subseteq\{1,\ldots,H\}$, which may depend on the observations,
such that $1\leq|\widehat P|\leq4$ and
\[
\operatorname{area}\left(\bigcap_{s\in\widehat P}\mathcal C_s\right)
\leq C\operatorname{area}(\mathcal C).
\]
Moreover, we have
\[
\sum_{s\in\widehat P}\Gamma_s-\Gamma_{\widehat P,R}
=\min_{u\in\mathcal B_R}\left\{
\sum_{s\in\widehat P}\left(\Gamma_s-\log L_s(u)\right)\right\}
\leq\min_{u\in\mathcal B_R}\left\{
\sum_{s=1}^H\left(\Gamma_s-\log L_s(u)\right)\right\}
=\Delta_R.
\]
It follows that
$\mathcal K_{\widehat P,R}\subseteq\bigcap_{s\in\widehat P}\mathcal C_s$.
Using the area comparison once again gives
\[
\frac{\exp(-\Gamma_{\widehat P,R})I_{\widehat P,R}}
{\exp(-\Gamma_{\{1,\ldots,H\},R})I_{\{1,\ldots,H\},R}}
\leq C(H+1)^2(\Delta_R+1)^2.
\]

To avoid conditioning on the set $\widehat P$, which depends on the data, let
$\mathcal P$ be the fixed family of all nonempty subsets of
$\{1,\ldots,H\}$ with at most four elements:
\[
\mathcal P
=\left\{P\subseteq\{1,\ldots,H\}:1\leq|P|\leq4\right\}.
\]
In particular, $|\mathcal P|\leq C(H+1)^4$.  For $P\in\mathcal P$, Fubini's
theorem, block independence, and conditional Jensen's inequality give, for
$0<r\leq1$,
\[
\begin{aligned}
&\E\left[\frac{I_{\{1,\ldots,H\},R}}{I_{P,R}}
\,\middle|\,(B_{I_s})_{s\in P}\right]
=\frac{1}{I_{P,R}}\int_{\mathcal B_R}L_P(u)
\E L_{\{1,\ldots,H\}\setminus P}(u)\,du =1,\\
&\E\left[\left(\frac{I_{\{1,\ldots,H\},R}}{I_{P,R}}\right)^r
\,\middle|\,(B_{I_s})_{s\in P}\right]\leq1.
\end{aligned}
\]
Here $I_{P,R}$ depends only on the blocks indexed by $P$, while the conditional
expectation is over the omitted blocks.  Thus the ratio of the two integrals
has conditional mean one for every fixed $P$.  The area comparison above connects this ratio back
to the maxima.  Rearranging the
preceding ratio gives the pointwise bound
\[
\exp(\Gamma_{\{1,\ldots,H\},R})
\leq C(H+1)^2(\Delta_R+1)^2
\max_{P\in\mathcal P}\exp(\Gamma_{P,R})
\frac{I_{\{1,\ldots,H\},R}}{I_{P,R}}.
\]

Choose a fixed $0<\eta\leq\min\{\eta_0/2,1/4\}$.
For every fixed $P\in\mathcal P$, the conditional bound above, block
independence, and~\eqref{eq:d2-sketch-block-moment} give
\[
\E\left[\exp(2\eta\Gamma_{P,R})
\left(\frac{I_{\{1,\ldots,H\},R}}{I_{P,R}}\right)^{2\eta}\right]
\leq\E\exp\left(2\eta\sum_{s\in P}\Gamma_s\right)
\leq C.
\]
Also, \eqref{eq:d2-sketch-block-moment} gives $\E\Gamma_s\leq C$, and
$0\leq\Delta_R\leq\sum_{s=1}^H\Gamma_s$.  Since $4\eta\leq1$, concavity gives
\[
\E(\Delta_R+1)^{4\eta}\leq C(H+1)^{4\eta}.
\]
Using the Cauchy--Schwarz inequality and
$\E\max_{P\in\mathcal P}Z_P\leq\sum_{P\in\mathcal P}\E Z_P$ for
nonnegative $Z_P$ gives
\[
\begin{aligned}
\E\exp(\eta\Gamma_{\{1,\ldots,H\},R})
&\leq C(H+1)^{2\eta}
\left(\E(\Delta_R+1)^{4\eta}\right)^{1/2}
\left(\sum_{P\in\mathcal P}
\E\left[\exp(2\eta\Gamma_{P,R})
\left(\frac{I_{\{1,\ldots,H\},R}}{I_{P,R}}\right)^{2\eta}\right]\right)^{1/2}\\
&\leq C(H+1)^{4\eta}|\mathcal P|^{1/2}
\leq C(H+1)^{4\eta+2}.
\end{aligned}
\]
The last bound does not depend on $R$.  Since
$\Gamma_{\{1,\ldots,H\},R}\to\Gamma_{\mathrm{ret}}$ as $R\to\infty$,
monotone convergence gives the same moment bound for
$\Gamma_{\mathrm{ret}}$.  Combining it with
\eqref{eq:d2-sketch-large-margin-moment} through
\eqref{eq:d2-sketch-margin-decomposition}, and applying H\"older's inequality
after decreasing $\eta$ if necessary, proves~\eqref{eq:d2-sketch-mgf}.
Markov's inequality then gives
\[
\Quant{1-\delta}\bigl(\Gamma_W(\eps;v)\bigr)
\leq C\left(\log(H+2)+\log\frac1\delta\right)
\leq C\left(\log\log\log n+\log\frac1\delta\right).
\]
Here the last inequality uses~\eqref{eq:d2-sketch-block-count}.
The claim follows.
\end{proofsketch}

\begin{remark}[Where $d = 2$ enters]
Arguably the essential use of $d=2$ in the proof occurs in the scalar reduction
\eqref{eq:d2-sketch-scalar-reduction}.  After profiling the common intercept
and conditioning on the number of errors, only one slope parameter remains,
so \Cref{cor:dimone} gives the block bound
\eqref{eq:d2-sketch-block-moment}.  For example, in dimension $d = 3$, two slope parameters
remain and the same reduction to the one-dimensional bound is unavailable.
\end{remark}

For completeness, we also state the corresponding quantile lower bound.
There is almost nothing new in its proof compared with the proof of
\Cref{prop:d2-logloglog-lower}.  We retain the event used in that proof
instead of averaging over it, and then use
\Cref{prop:coordinate-subspace-lower} for the term $\log(1/\delta)$.

\begin{corollary}[Worst case quantile lower bound in dimension two]
\label{cor:d2-logloglog-quantile-lower}
There exists a constant $c>0$ such that, for any $n\geq16$ and any
$\delta\in(0,1/(16\e^2)]$, it holds that
\[
\sup_{x_1,\ldots,x_n\in\R^2}
\sup_{\thetastar\in\R^2}
\Quant{1-\delta}
\Bigl(
\Lambda_n^{\operatorname{log}}
(\thetastar;x_{1:n};Y_{1:n})
\Bigr)
\geq
c\left(
\log\log\log n+\log\frac{1}{\delta}
\right).
\]
\end{corollary}

\begin{proofsketch}
For all sufficiently large $n$, the proof of
\Cref{prop:d2-logloglog-lower} gives a two-dimensional subspace
$W\subset\R^n$ and a vector $v\in W$ such that
\[
\P_{\eps\sim p_v}\left(
\Gamma_W(\eps;v)\geq\frac12\log\log\log n
\right)
\geq\frac{1}{8\e^2}.
\]
Thus, by~\Cref{lem:reformulation-via-subspaces}, there are design vectors and
a target parameter for which the likelihood ratio is at least
$\tfrac12\log\log\log n$ with probability at least
$1/(8\e^2)>\delta$. Its $(1-\delta)$ quantile has the same lower bound. The
term $\log(1/\delta)$ follows from
\Cref{prop:coordinate-subspace-lower}, applied with $k=2$. Taking the better of the two
constructions and decreasing the universal constant gives the displayed sum. The remaining bounded range of $n$ is covered by the
coordinate subspace bound after a further adjustment of the constant.
\end{proofsketch}

\newpage
\begin{promptbox}
Let
\[
\sigma(t)=\frac1{1+e^{-t}}.
\]
For an integer \(n\ge1\), let \(W\subseteq\mathbb R^n\) be a deterministic linear
subspace with \(\dim W\le2\), and let \(v\in W\). For
\(y\in\{-1,1\}^n\), define
\[
p_v(y)=\prod_{i=1}^n\sigma(y_iv_i),\qquad
\Gamma_W(y;v)=\sup_{w\in W}\log\frac{p_w(y)}{p_v(y)}
=\sup_{w\in W}\sum_{i=1}^n
\log\frac{\sigma(y_iw_i)}{\sigma(y_iv_i)}.
\]
The supremum is extended: it need not be attained by a finite \(w\).

Prove completely that there are absolute constants \(C>0\) and \(N\in\mathbb N\)
such that, for every \(n\ge N\), every such \(W,v\), every \(0<\delta<1\), and
\(Y\sim p_v\),
\[
\Pr_v\!\left\{\Gamma_W(Y;v)>
C\left(\log\log\log n+\log\frac1\delta\right)\right\}\le\delta .
\]
All logarithms are natural, and \(N\) is large enough that
\(\log\log\log n\ge1\). Constants must be uniform in the design, multiplicities,
conditioning of the design, and existence of a finite MLE.

Give a clean, unconditional, publication-quality proof. Do not assume bounded fitted
predictors or true parameters, general position, distinct rows, full rank, good
conditioning, or attainment. Treat rank zero and rank one, empty blocks, repeated
and collinear rows, all-equal outcomes, endpoint suprema, separation,
quasi-separation, and lower-rank limits. Prove every nonstandard auxiliary assertion
and verify the hypotheses of every quoted standard theorem. Do not leave
placeholders such as “assuming the following lemma”; the submitted answer itself
must finish every step.

The hints below are a roadmap, not granted lemmas. Verify every assertion and supply
all arguments needed to make it valid.

Hints.

1. Extreme block: modal-recode and set \(a_i=|v_i|\). For
   \(E=\{i:a_i>2\log n\}\) and the true product law \(P_E\) of their modal-error
   vector \(B_E\), prove
   \[
   \Gamma_E(B_E)\le-\log P_E(B_E),\qquad
   \mathbb E e^{\eta[-\log P_E(B_E)]}\le C
   \]
   for universal \(\eta,C>0\).

2. Offset profile: for arbitrary fixed logistic offsets and a common intercept,
   profile the intercept, condition on total error count, regularize the scalar
   slope, and prove a universal fractional exponential moment for the extended
   gain. Treat all-equal fibers and endpoint suprema directly.

3. Factor-two block: normalize by its true margins; use concavity to compare the
   rescaled likelihoods; then reduce its gain, including fractional rescaling
   through Bernoulli thinning, to a bounded number of arbitrary-offset
   common-intercept profiles.

4. Central geometry: use uniform local logistic curvature and
   \[
   Q(h)=\sum_i\min\{z_i(h)^2,1\}.
   \]
   On a shell \(s\le\Gamma<2s\), the ridge/KL identity gives
   \(Q(h_\varepsilon)\le Cs\). For \(u\ge1\), prove by planar rank-two packing
   \[
   \#\{i:\exists h,\ Q(h)\le u,\ |z_i(h)|\ge1\}=O(u).
   \]

5. Central gain: control the exceptional rows by a finite planar ray profile
   and the rest by a rank-two quadratic profile, obtaining a universal exponential
   tail and fractional exponential moment for the central-block gain.

6. Aggregation: use convex deficits and the sublevel-set/partition-function
   comparison. With \(\Delta=\sum_{s=0}^H\Gamma_s-\Gamma_\varepsilon\), prove that
   planar quantitative Helly selects \(|P|\le4\) with
   \[
   \frac{Z_P}{Z}\le C(H+1)^2(\Delta+1)^2.
   \]
\end{promptbox}

\end{document}